\documentclass[12pt]{amsart}       
\usepackage{txfonts}
\usepackage{amssymb}
\usepackage{eucal}
\usepackage{graphicx}
\usepackage{amsmath}
\usepackage{amscd}
\usepackage[all]{xy}           
\usepackage{amsfonts,latexsym}
\usepackage{xspace}
\usepackage{epsfig}
\usepackage{float}
\usepackage{color}
\usepackage{fancybox}
\usepackage{colordvi}
\usepackage{multicol}
\usepackage[active]{srcltx} 

\usepackage{tikz}

\usepackage{graphicx}

\ifpdf
\usepackage[colorlinks,final,backref=page,hyperindex]{hyperref}

\newtheorem{theorem}{Theorem}[section]
\newtheorem{prop}[theorem]{Proposition}
\theoremstyle{definition}
\newtheorem{defn}[theorem]{Definition}
\newtheorem{lemma}[theorem]{Lemma}
\newtheorem{coro}[theorem]{Corollary}
\newtheorem{prop-def}{Proposition-Definition}[section]
\newtheorem{coro-def}{Corollary-Definition}[section]

\newtheorem{remark}[theorem]{Remark}

\newtheorem{exam}[theorem]{Example}

\newcommand{\nc}{\newcommand}
\nc{\tred}[1]{\textcolor{red}{#1}}
\nc{\tblue}[1]{\textcolor{blue}{#1}}
\nc{\tgreen}[1]{\textcolor{green}{#1}}
\nc{\tpurple}[1]{\textcolor{purple}{#1}}
\nc{\btred}[1]{\textcolor{red}{\bf #1}}
\nc{\btblue}[1]{\textcolor{blue}{\bf #1}}
\nc{\btgreen}[1]{\textcolor{green}{\bf #1}}
\nc{\btpurple}[1]{\textcolor{purple}{\bf #1}}
\nc{\NN}{{\mathbb N}}
\nc{\ncsha}{{\mbox{\cyr X}^{\mathrm NC}}} \nc{\ncshao}{{\mbox{\cyr
			X}^{\mathrm NC}_0}}

\newcommand{\efootnote}[1]{}

\renewcommand{\textbf}[1]{}

\newcommand{\delete}[1]{}

	\nc{\mlabel}[1]{\label{#1}}  
	\nc{\mcite}[1]{\cite{#1}}  
	\nc{\mref}[1]{\ref{#1}}  
	\nc{\mbibitem}[1]{\bibitem{#1}} 

\delete{
	\nc{\mlabel}[1]{\label{#1}{\hfill \hspace{1cm}{\bf{{\ }\hfill(#1)}}}}
	\nc{\mcite}[1]{\cite{#1}{{\bf{{\ }(#1)}}}}  
	\nc{\mref}[1]{\ref{#1}{{\bf{{\ }(#1)}}}}  
	\nc{\mbibitem}[1]{\bibitem[\bf #1]{#1}} 
}

\nc{\opa}{\ast} \nc{\opb}{\odot} \nc{\op}{\bullet} \nc{\pa}{\frakL}
\nc{\arr}{\rightarrow} \nc{\lu}[1]{(#1)} \nc{\mult}{\mrm{mult}}
\nc{\diff}{\mathfrak{Diff}}
\nc{\opc}{\sharp}\nc{\opd}{\natural}
\nc{\ope}{\circ}
\nc{\dpt}{\mathrm{d}}
\nc{\hck}{H_{RT}}
\nc{\vdf}{\calf}
\nc{\ldf}{\calf_\ell}
\nc{\hlf}{H_\ell}
\nc{\onek}{\mathbf{1}_\bfk}
\nc{\diam}{alternating\xspace}
\nc{\Diam}{Alternating\xspace}
\nc{\cdiam}{canonical alternating\xspace}
\nc{\Cdiam}{Canonical alternating\xspace}
\nc{\AW}{\mathcal{A}}

\nc{\ari}{\mathrm{ar}}

\nc{\lef}{\mathrm{lef}}

\nc{\Sh}{\mathrm{ST}}

\nc{\Cr}{\mathrm{Cr}}

\nc{\st}{{Schr\"oder tree}\xspace}
\nc{\sts}{{Schr\"oder trees}\xspace}

\nc{\vertset}{\Omega} 

\nc{\assop}{\quad \begin{picture}(5,5)(0,0)
		\line(-1,1){10}
		\put(-2.2,-2.2){$\bullet$}
		\line(0,-1){10}\line(1,1){10}
	\end{picture} \quad \smallskip}

\nc{\operator}{\begin{picture}(5,5)(0,0)
		\line(0,-1){6}
		\put(-2.6,-1.8){$\bullet$}
		\line(0,1){9}
\end{picture}}

\nc{\idx}{\begin{picture}(6,6)(-3,-3)
		\put(0,0){\line(0,1){6}}
		\put(0,0){\line(0,-1){6}}
\end{picture}}

\nc{\pb}{{\mathrm{pb}}}
\nc{\Lf}{{\mathrm{Lf}}}

\nc{\lft}{{left tree}\xspace}
\nc{\lfts}{{left trees}\xspace}

\nc{\fat}{{fundamental averaging tree}\xspace}

\nc{\fats}{{fundamental averaging trees}\xspace}
\nc{\avt}{\mathrm{Avt}}

\nc{\rass}{{\mathit{RAss}}}

\nc{\aass}{{\mathit{AAss}}}

\nc{\vin}{{\mathrm Vin}}    
\nc{\lin}{{\mathrm Lin}}    
\nc{\inv}{\mathrm{I}n}
\nc{\gensp}{V} 
\nc{\genbas}{\mathcal{V}} 
\nc{\bvp}{V_P}     
\nc{\gop}{{\,\omega\,}}     

\nc{\bin}[2]{ (_{\stackrel{\scs{#1}}{\scs{#2}}})}  
\nc{\binc}[2]{ \left (\!\! \begin{array}{c} \scs{#1}\\
		\scs{#2} \end{array}\!\! \right )}  
\nc{\bincc}[2]{  \left ( {\scs{#1} \atop
		\vspace{-1cm}\scs{#2}} \right )}  
\nc{\bs}{\bar{S}} \nc{\cosum}{\sqsubset} \nc{\la}{\longrightarrow}
\nc{\rar}{\rightarrow} \nc{\dar}{\downarrow} \nc{\dprod}{**}
\nc{\dap}[1]{\downarrow \rlap{$\scriptstyle{#1}$}}
\nc{\md}{\mathrm{dth}} \nc{\uap}[1]{\uparrow
	\rlap{$\scriptstyle{#1}$}} \nc{\defeq}{\stackrel{\rm def}{=}}
\nc{\disp}[1]{\displaystyle{#1}} \nc{\dotcup}{\
	\displaystyle{\bigcup^\bullet}\ } \nc{\gzeta}{\bar{\zeta}}
\nc{\hcm}{\ \hat{,}\ } \nc{\hts}{\hat{\otimes}}
\nc{\barot}{{\otimes}} \nc{\free}[1]{\bar{#1}}
\nc{\uni}[1]{\tilde{#1}} \nc{\hcirc}{\hat{\circ}} \nc{\lleft}{[}
\nc{\lright}{]} \nc{\lc}{\lfloor} \nc{\rc}{\rfloor}
\nc{\curlyl}{\left \{ \begin{array}{c} {} \\ {} \end{array}
	\right .  \!\!\!\!\!\!\!}
\nc{\curlyr}{ \!\!\!\!\!\!\!
	\left . \begin{array}{c} {} \\ {} \end{array}
	\right \} }
\nc{\longmid}{\left | \begin{array}{c} {} \\ {} \end{array}
	\right . \!\!\!\!\!\!\!}
\nc{\onetree}{\bullet} \nc{\ora}[1]{\stackrel{#1}{\rar}}
\nc{\ola}[1]{\stackrel{#1}{\la}}
\nc{\ot}{\otimes} \nc{\mot}{{{\boxtimes\,}}}
\nc{\otm}{\overline{\boxtimes}} \nc{\sprod}{\bullet}
\nc{\scs}[1]{\scriptstyle{#1}} \nc{\mrm}[1]{{\rm #1}}
\nc{\margin}[1]{\marginpar{\rm #1}}   
\nc{\dirlim}{\displaystyle{\lim_{\longrightarrow}}\,}
\nc{\invlim}{\displaystyle{\lim_{\longleftarrow}}\,}
\nc{\mvp}{\vspace{0.3cm}} \nc{\tk}{^{(k)}} \nc{\tp}{^\prime}
\nc{\ttp}{^{\prime\prime}} \nc{\svp}{\vspace{2cm}}
\nc{\vp}{\vspace{8cm}} \nc{\proofbegin}{\noindent{\bf Proof: }}
\nc{\proofend}{$\blacksquare$ \vspace{0.3cm}}
\nc{\modg}[1]{\!<\!\!{#1}\!\!>}
\nc{\intg}[1]{F_C(#1)} \nc{\lmodg}{\!
<\!\!} \nc{\rmodg}{\!\!>\!}
\nc{\cpi}{\widehat{\Pi}}
\nc{\sha}{{\mbox{\cyr X}}}  
\nc{\shap}{{\mbox{\cyrs X}}} 
\nc{\shpr}{\diamond}    
\nc{\shp}{\ast} \nc{\shplus}{\shpr^+}
\nc{\shprc}{\shpr_c}    
\nc{\msh}{\ast} \nc{\zprod}{m_0} \nc{\oprod}{m_1}
\nc{\vep}{\epsilon} \nc{\labs}{\mid\!} \nc{\rabs}{\!\mid}
\nc{\sqmon}[1]{\langle #1\rangle}

\nc{\mmbox}[1]{\mbox{\ #1\ }} \nc{\dep}{\mrm{dep}} \nc{\fp}{\mrm{FP}}
\nc{\rchar}{\mrm{char}} \nc{\End}{\mrm{End}} \nc{\Fil}{\mrm{Fil}}
\nc{\Mor}{Mor\xspace} \nc{\gmzvs}{gMZV\xspace}
\nc{\gmzv}{gMZV\xspace} \nc{\mzv}{MZV\xspace}
\nc{\mzvs}{MZVs\xspace} \nc{\Hom}{\mrm{Hom}} \nc{\id}{\mrm{id}}
\nc{\im}{\mrm{im}} \nc{\incl}{\mrm{incl}} \nc{\map}{\mrm{Map}}
\nc{\mchar}{\rm char} \nc{\nz}{\rm NZ} \nc{\supp}{\mathrm Supp}

\nc{\Alg}{\mathbf{Alg}} \nc{\Bax}{\mathbf{Bax}} \nc{\bff}{\mathbf f}
\nc{\bfk}{{\bf k}} \nc{\bfone}{{\bf 1}} \nc{\bfx}{\mathbf x}
\nc{\bfy}{\mathbf y}
\nc{\base}[1]{\bfone^{\otimes ({#1}+1)}} 
\nc{\Cat}{\mathbf{Cat}}

\nc{\detail}{\marginpar{\bf More detail}
\noindent{\bf Need more detail!}
\svp}
\nc{\Int}{\mathbf{Int}} \nc{\Mon}{\mathbf{Mon}}
\nc{\rbtm}{{shuffle }} \nc{\rbto}{{Rota-Baxter }}
\nc{\remarks}{\noindent{\bf Remarks: }} \nc{\Rings}{\mathbf{Rings}}
\nc{\Sets}{\mathbf{Sets}} \nc{\wtot}{\widetilde{\odot}}
\nc{\wast}{\widetilde{\ast}} \nc{\bodot}{\bar{\odot}}
\nc{\bast}{\bar{\ast}} \nc{\hodot}[1]{\odot^{#1}}
\nc{\hast}[1]{\ast^{#1}} \nc{\mal}{\mathcal{O}}
\nc{\tet}{\tilde{\ast}} \nc{\teot}{\tilde{\odot}}
\nc{\oex}{\overline{x}} \nc{\oey}{\overline{y}}
\nc{\oez}{\overline{z}} \nc{\oef}{\overline{f}}
\nc{\oea}{\overline{a}} \nc{\oeb}{\overline{b}}
\nc{\weast}[1]{\widetilde{\ast}^{#1}}
\nc{\weodot}[1]{\widetilde{\odot}^{#1}} \nc{\hstar}[1]{\star^{#1}}
\nc{\lae}{\langle} \nc{\rae}{\rangle}
\nc{\lf}{\lfloor}
\nc{\rf}{\rfloor}
\nc{\dpl}{\varepsilon_{RT}}

\nc{\QQ}{{\mathbb Q}}
\nc{\RR}{{\mathbb R}} \nc{\ZZ}{{\mathbb Z}}

\nc{\cala}{{\mathcal A}} \nc{\calb}{{\mathcal B}}
\nc{\calc}{{\mathcal C}}
\nc{\cald}{{\mathcal D}} \nc{\cale}{{\mathcal E}}
\nc{\calf}{{\mathcal F}} \nc{\calg}{{\mathcal G}}
\nc{\calh}{{\mathcal H}} \nc{\cali}{{\mathcal I}}
\nc{\call}{{\mathcal L}} \nc{\calm}{{\mathcal M}}
\nc{\caln}{{\mathcal N}} \nc{\calo}{{\mathcal O}}
\nc{\calp}{{\mathcal P}} \nc{\calr}{{\mathcal R}}
\nc{\cals}{{\mathcal S}} \nc{\calt}{{\mathcal T}}
\nc{\calu}{{\mathcal U}} \nc{\calw}{{\mathcal W}} \nc{\calk}{{\mathcal K}}
\nc{\calx}{{\mathcal X}} \nc{\CA}{\mathcal{A}}

\nc{\fraka}{{\mathfrak a}} \nc{\frakA}{{\mathfrak A}}
\nc{\frakb}{{\mathfrak b}} \nc{\frakB}{{\mathfrak B}}
\nc{\frakD}{{\mathfrak D}} \nc{\frakF}{\mathfrak{F}}
\nc{\frakf}{{\mathfrak f}} \nc{\frakg}{{\mathfrak g}}
\nc{\frakH}{{\mathfrak H}} \nc{\frakL}{{\mathfrak L}}
\nc{\frakM}{{\mathfrak M}} \nc{\bfrakM}{\overline{\frakM}}
\nc{\frakm}{{\mathfrak m}} \nc{\frakP}{{\mathfrak P}}
\nc{\frakN}{{\mathfrak N}} \nc{\frakp}{{\mathfrak p}}
\nc{\frakS}{{\mathfrak S}} \nc{\frakT}{\mathfrak{T}}
\nc{\frakX}{{\mathfrak X}}
\nc{\BS}{\mathbb{S
}}

\font\cyr=wncyr10 \font\cyrs=wncyr7
\nc{\li}[1]{\textcolor{red}{Li:#1}}
\nc{\yi}[1]{\textcolor{blue}{Yi: #1}}
\nc{\xing}[1]{\textcolor{purple}{Xing:#1}}
\nc{\revise}[1]{\textcolor{red}{#1}}

\nc{\ID}{{\rm I}}\nc{\lbar}[1]{\overline{#1}}\nc{\bre}{{\rm bre}}
\nc{\sd}{\cals}\nc{\rb}{\rm RB}\nc{\A}{\rm A}\nc{\LL}{\rm L}\nc{\tx}{\tilde{X}}
\nc{\col}{\Delta_{\lambda,\mu}}\nc{\mul}{m_{\mathrm{RT}}}\nc{\ul}{u_{RT}}\nc{\epl}{\epsilon_{RT}}
\nc{\hl}{H_{RT}}\nc{\arro}[1]{#1}\nc{\px}{P_{\tx}}\nc{\pw}{P_{\mathfrak{w}}}\nc{\pl}{B^+}
\nc{\pp}{\pl}\nc{\ppp}[1]{B^+(#1)}\nc{\dw}{\diamond_{\mathfrak{w}}}\nc{\dl}{\diamond_{\rm \ell}}
\nc{\ncshaw}{\sha^{{\rm NC}}_{\mathfrak{w}}}\nc{\ncshal}{\sha^{{\rm NC}}_{{\rm \ell}}}
\nc{\ver}{\rm V}\nc{\ld}{l}\nc{\del}{\Delta_{{\rm \ell}}}\nc{\epsl}{\epsilon_{{\rm \ell}}}
\nc{\uul}{u_{{\rm \ell}}}\nc{\oneh}{\mathbf{1}}\nc{\onew}{\mathbf{1}}
\nc{\etree}{1} \nc{\conc}{m_{RT}} \nc{\subq}{\bfk Q_l} \nc{\fid}{\unlhd}  \nc{\sfid}{\lhd}
\nc{\lhl}{\leq_{h,l}} \nc{\ghl}{\geq_{hl}}
\nc{\RT}{\mathrm{RT}}

\nc{\hrtb}{\mathcal{H}_{RT}(X\sqcup\Omega)} \nc{\hrts}{\mathcal{H}_{\mathrm{RT}}(X, \Omega)}\nc{\rts}{\mathcal{T}(X, \Omega)}\nc{\rfs}{\mathcal{F}(X, \Omega)} \nc{\counit}{\varepsilon_{\mathrm{RT}}}

\newcommand{\tun}{\begin{picture}(5,0)(-2,-1)
	\put(0,0){\circle*{2}}
	\end{picture}}
	
	\newcommand{\tdeux}{\begin{picture}(7,7)(0,-1)
	\put(3,0){\circle*{2}}
	\put(3,0){\line(0,1){5}}
	\put(3,5){\circle*{2}}
	\end{picture}}
	
	\newcommand{\ttroisun}{\begin{picture}(15,8)(-5,-1)
	\put(3,0){\circle*{2}}
	\put(-0.65,0){$\vee$}
	\put(6,7){\circle*{2}}
	\put(0,7){\circle*{2}}
	\end{picture}}
	\newcommand{\ttroisdeux}{\begin{picture}(5,12)(-2,-1)
	\put(0,0){\circle*{2}}
	\put(0,0){\line(0,1){5}}
	\put(0,5){\circle*{2}}
	\put(0,5){\line(0,1){5}}
	\put(0,10){\circle*{2}}
	\end{picture}}
	
	\newcommand{\tquatreun}{\begin{picture}(15,12)(-5,-1)
	\put(3,0){\circle*{2}}
	\put(-0.65,0){$\vee$}
	\put(6,7){\circle*{2}}
	\put(0,7){\circle*{2}}
	\put(3,7){\circle*{2}}
	\put(3,0){\line(0,1){7}}
	\end{picture}}
	\newcommand{\tquatredeux}{\begin{picture}(15,18)(-5,-1)
	\put(3,0){\circle*{2}}
	\put(-0.65,0){$\vee$}
	\put(6,7){\circle*{2}}
	\put(0,7){\circle*{2}}
	\put(0,14){\circle*{2}}
	\put(0,7){\line(0,1){7}}
	\end{picture}}
	\newcommand{\tquatretrois}{\begin{picture}(15,18)(-5,-1)
	\put(3,0){\circle*{2}}
	\put(-0.65,0){$\vee$}
	\put(6,7){\circle*{2}}
	\put(0,7){\circle*{2}}
	\put(6,14){\circle*{2}}
	\put(6,7){\line(0,1){7}}
	\end{picture}}
	\newcommand{\tquatrequatre}{\begin{picture}(15,18)(-5,-1)
	\put(3,5){\circle*{2}}
	\put(-0.65,5){$\vee$}
	\put(6,12){\circle*{2}}
	\put(0,12){\circle*{2}}
	\put(3,0){\circle*{2}}
	\put(3,0){\line(0,1){5}}
	\end{picture}}
	\newcommand{\tquatrecinq}{\begin{picture}(9,19)(-2,-1)
	\put(0,0){\circle*{2}}
	\put(0,0){\line(0,1){5}}
	\put(0,5){\circle*{2}}
	\put(0,5){\line(0,1){5}}
	\put(0,10){\circle*{2}}
	\put(0,10){\line(0,1){5}}
	\put(0,15){\circle*{2}}
	\end{picture}}

	\newcommand{\tcinqdeux}{\begin{picture}(15,14)(-5,-1)
	\put(3,0){\circle*{2}}
	\put(-0.65,0){$\vee$}
	\put(6,7){\circle*{2}}
	\put(0,7){\circle*{2}}
	\put(3,7){\circle*{2}}
	\put(3,0){\line(0,1){7}}
	\put(0,7){\line(0,1){7}}
	\put(0,14){\circle*{2}}
	\end{picture}}
	\newcommand{\tcinqtrois}{\begin{picture}(15,15)(-5,-1)
	\put(3,0){\circle*{2}}
	\put(-0.65,0){$\vee$}
	\put(6,7){\circle*{2}}
	\put(0,7){\circle*{2}}
	\put(3,7){\circle*{2}}
	\put(3,0){\line(0,1){7}}
	\put(3,7){\line(0,1){7}}
	\put(3,14){\circle*{2}}
	\end{picture}}
	\newcommand{\tcinqquatre}{\begin{picture}(15,14)(-5,-1)
	\put(3,0){\circle*{2}}
	\put(-0.65,0){$\vee$}
	\put(6,7){\circle*{2}}
	\put(0,7){\circle*{2}}
	\put(3,7){\circle*{2}}
	\put(3,0){\line(0,1){7}}
	\put(6,7){\line(0,1){7}}
	\put(6,14){\circle*{2}}
	\end{picture}}

	\newcommand{\tdun}[1]
	{\begin{picture}(10,5)(-2,-1)
	\put(0,0){\circle*{2}}
	\put(3,-2){\tiny #1}
	\end{picture}}
	
	\newcommand{\tddeux}[2]{\begin{picture}(12,5)(0,-1)
	\put(3,0){\circle*{2}}
	\put(3,0){\line(0,1){5}}
	\put(3,5){\circle*{2}}
	\put(6,-3){\tiny #1}
	\put(6,3){\tiny #2}
	\end{picture}}
	
	\newcommand{\tdtroisun}[3]{\begin{picture}(20,12)(-5,-1)
	\put(3,0){\circle*{2}}
	\put(-0.65,0){$\vee$}
	\put(6,7){\circle*{2}}
	\put(0,7){\circle*{2}}
	\put(5,-2){\tiny #1}
	\put(8,5){\tiny #2}
	\put(-6,5){\tiny #3}
	\end{picture}}
	\newcommand{\tdtroisdeux}[3]{\begin{picture}(12,12)(-2,-1)
	\put(0,0){\circle*{2}}
	\put(0,0){\line(0,1){5}}
	\put(0,5){\circle*{2}}
	\put(0,5){\line(0,1){5}}
	\put(0,10){\circle*{2}}
	\put(3,-2){\tiny #1}
	\put(3,3){\tiny #2}
	\put(3,9){\tiny #3}
	\end{picture}}
	
	\newcommand{\tdquatreun}[4]{\begin{picture}(20,12)(-5,-1)
	\put(3,0){\circle*{2}}
	\put(-0.6,0){$\vee$}
	\put(6,7){\circle*{2}}
	\put(0,7){\circle*{2}}
	\put(3,7){\circle*{2}}
	\put(3,0){\line(0,1){7}}
	\put(5,-2){\tiny #1}
	\put(8.5,5){\tiny #2}
	\put(1,10){\tiny #3}
	\put(-5,5){\tiny #4}
	\end{picture}}
	\newcommand{\tdquatredeux}[4]{\begin{picture}(20,20)(-5,-1)
	\put(3,0){\circle*{2}}
	\put(-.65,0){$\vee$}
	\put(6,7){\circle*{2}}
	\put(0,7){\circle*{2}}
	\put(0,14){\circle*{2}}
	\put(0,7){\line(0,1){7}}
	\put(5,-2){\tiny #1}
	\put(9,5){\tiny #2}
	\put(-6,5){\tiny #3}
	\put(-6,12){\tiny #4}
	\end{picture}}
	\newcommand{\tdquatretrois}[4]{\begin{picture}(20,20)(-5,-1)
	\put(3,0){\circle*{2}}
	\put(-.65,0){$\vee$}
	\put(6,7){\circle*{2}}
	\put(0,7){\circle*{2}}
	\put(6,14){\circle*{2}}
	\put(6,7){\line(0,1){7}}
	\put(5,-2){\tiny #1}
	\put(8,5){\tiny #2}
	\put(-6,5){\tiny #4}
	\put(8,12){\tiny #3}
	\end{picture}}
	\newcommand{\tdquatrequatre}[4]{\begin{picture}(20,14)(-5,-1)
	\put(3,5){\circle*{2}}
	\put(-.65,5){$\vee$}
	\put(6,12){\circle*{2}}
	\put(0,12){\circle*{2}}
	\put(3,0){\circle*{2}}
	\put(3,0){\line(0,1){5}}
	\put(6,-3){\tiny #1}
	\put(6,4){\tiny #2}
	\put(9,12){\tiny #3}
	\put(-5,12){\tiny #4}
	\end{picture}}
	\newcommand{\tdquatrecinq}[4]{\begin{picture}(12,19)(-2,-1)
	\put(0,0){\circle*{2}}
	\put(0,0){\line(0,1){5}}
	\put(0,5){\circle*{2}}
	\put(0,5){\line(0,1){5}}
	\put(0,10){\circle*{2}}
	\put(0,10){\line(0,1){5}}
	\put(0,15){\circle*{2}}
	\put(3,-2){\tiny #1}
	\put(3,3){\tiny #2}
	\put(3,9){\tiny #3}
	\put(3,14){\tiny #4}
	\end{picture}}

\begin{document}

\title[Cocycle weighted infinitesimal  bialgebras of rooted forests]{Cocycle weighted infinitesimal  bialgebras and pre-Lie algebras on rooted forests}
%

\author{Lo\"\i c Foissy}
\address{Univ. Littoral C\^ote d'Opale, UR 2597 LMPA, Laboratoire de Math\'ematiques Pures et Appliqu\'ees Joseph Liouville F-62100 Calais, France}
\email{loic.foissy@univ-littoral.fr}

\author{Yunzhou Xie}
\address{Department of Mathematics, Imperial College London, London SW7 2AZ, UK}
\email{yx3021@ic.ac.uk}

\author{Dawei Zhang}
\address{School of Mathematics and Statistics,
	Nanjing University of Information Science \& Technology, Nanjing, Jiangsu 210044, P.\,R. China}
\email{zhangdw2025@nuist.edu.cn}

\author{Yi Zhang}
\address{School of Mathematics and Statistics,
	Nanjing University of Information Science \& Technology, Nanjing, Jiangsu 210044, P.\,R. China}
\email{zhangy2016@nuist.edu.cn}

\date{\today}
\begin{abstract}
	The concept of weighted infinitesimal bialgebras provides an algebraic framework for understanding the non-homogeneous associative Yang-Baxter equation. In this paper, we endow the space of decorated planar rooted forests with a two-parameters family of coproducts, making it into a weighted infinitesimal bialgebra. A combinatorial characterization of the coproducts is given via the notion of forest biideals. Furthermore, by constructing a bilinear symmetric form and introducing a new grafting operation on rooted forests, we describe the associated dual products.
	
	We also introduce the notion of the pair-weight 1-cocycle condition and investigate the universal properties of decorated planar rooted forests satisfying this condition. This leads to the definition of a weighted $\Omega$-cocycle infinitesimal unitary bialgebra. As applications, we identify the initial object in the category of free cocycle infinitesimal unitary bialgebras on undecorated planar rooted forests, corresponding to the well-known noncommutative Connes-Kreimer Hopf algebra. In addition, we establish isomorphisms between different coproduct structures and construct a pre-Lie algebra structure on decorated planar rooted forests.
	
\end{abstract}

\subjclass[2010]{
	16W99, 
	05C05, 
	16S10, 
	16T10, 
	16T30,  
	17B60, 
}

\keywords{Rooted forest; Infinitesimal bialgebra; Cocycle condition; Pre-Lie algebra}

\maketitle

\tableofcontents

\setcounter{section}{0}

\allowdisplaybreaks

\section{Introduction}

The main aim of this paper is to explore the relationship between weighted infinitesimal bialgebras, operated algebras and pre-Lie algebras. By introducing a pair weight Hochschild 1-cocycle condition, we construct a cocycle weighted infinitesimal  bialgebras of rooted forests, which  generalizes both the infinitesimal  bialgebras introduced  by Gao-Wang~\mcite{GW19} and the infinitesimal algebras proposed by Foissy~\mcite{Foi09, Foi10}.
\subsection{Weighted infinitesimal bialgebras}

A weighted infinitesimal unitary bialgebra is a module $A$ which is simultaneously an algebra (possibly without a unit) and a coalgebra (possibly without a counit) such that the coproduct $\Delta$  satisfy  a weighted derivation property
$$\Delta(ab)=a\cdot\Delta(b)+\Delta(a)\cdot b+\lambda (a\ot b)\,\text{ for } a, b\in A,$$
where $\lambda$ is a fixed constant.

Weighted infinitesimal unitary bialgebras were initially introduced in~\mcite{Fard06}, where they served to provide an algebraic interpretation of the non-homogeneous associative classical Yang-Baxter equations. Subsequent developments can be found in~\mcite{AGR25, ZCGL18, ZZL19}. This structure can be regarded as a unified generalization of two distinct types of infinitesimal bialgebras. The first type, formulated by Joni and Rota~\mcite{JR}, was motivated by the desire to formalize the algebraic underpinning of Newton's divided difference calculus.

Later, Aguiar~\mcite{MA} enriched the theory by introducing an antipode, thus defining what he termed an infinitesimal Hopf algebra. This enhancement preserved many of the essential combinatorial features~\mcite{Agu02} and facilitated applications in several areas, including associative Yang-Baxter equations, Drinfeld doubles, and pre-Lie algebras~\mcite{MA, Agu01, Aguu02}.

Moreover, infinitesimal bialgebras have a profound connection to Drinfeld's notion of Lie bialgebras. This connection is made precise through the framework of balanced infinitesimal bialgebras, as discussed in~\mcite{Agu01, Que23}, which provides a conceptual bridge between these two algebraic structures.

Recent studies on infinitesimal bialgebras have advanced along several key directions. One major line involves Hom-type generalizations. Using twisting maps, Yau~\mcite{Ya} introduced infinitesimal Hom-bialgebras and examined subclasses such as those on quivers, upper-edge structures, and quasitriangular cases. The connections among infinitesimal Hom-bialgebras, the Hom-Yang-Baxter equation, and Hom-Lie bialgebras were further explored in~\mcite{Ya} and extended in~\mcite{LMMP, MM25, MM23}.
Another approach centers on $O$-operators. Bai~\mcite{Bai10} introduced antisymmetric infinitesimal bialgebras and revealed their close relations with Dendriform $D$-bialgebras, Frobenius algebras, $O$-operators, and the generalized associative Yang-Baxter equation. This approach makes it possible to establish  bialgebra theory for different algebraic structures, such as differential algebras~\mcite{LLB23}, Rota-Baxter (Lie) bialgebras~\mcite{BGM24, BGM242}, conformal algebras~\mcite{HB21},  noncommunicative Novikov algebras~\mcite{ HBG24, ZLYG25}, perm algebra~\mcite{BLZ25}, quasi-triangular algebras~\mcite{SW}  and some other Rota-Baxter type algebras~\mcite{HC24}.

In 2006, Loday and Ronco~\mcite{LR06} introduced infinitesimal bialgebras and infinitesimal Hopf algebras from a different perspective, and established a structure theorem for connected infinitesimal bialgebras. Later, Foissy~\mcite{Foi09} applied combinatorial algebra techniques to construct infinitesimal Hopf algebras of the Loday-Ronco on rooted forests, and further explored their connections with the  operads~\mcite{Foi10}. These developments sparked significant interest in the study of infinitesimal bialgebras. However, the coexistence of two distinct approaches to infinitesimal bialgebras posed challenges for researchers. To address this issue, Ebrahimi-Fard~\mcite{Fard06} employed weighted derivations to unify the compatibility conditions of both versions, thereby introducing the notion of weighted infinitesimal bialgebras.

A second motivation for introducing weighted infinitesimal bialgebras arises from the study of the weighted associative Yang-Baxter equation~\mcite{Fard06}, which was referred to by Ogievetsky and Popov~\mcite{OP10} as the non-homogeneous associative classical Yang-Baxter equation. Analogous to the well-known fact that solutions to the classical Yang-Baxter equation give rise to Lie bialgebras, it has been shown that solutions to the weighted associative Yang-Baxter equation naturally induce weighted infinitesimal bialgebras~\mcite{Fard06}. As such, weighted infinitesimal bialgebras can be regarded as an algebraic abstraction of the weighted associative Yang-Baxter equation~\mcite{OP10}.

\subsection{Cocycle operated algebras and rooted forests}

Motivated by Higgins'~\mcite{Hig56} foundational work on multi-operator groups, Kurosh~\mcite{Kur60} was among the first to explore algebras endowed with linear operators. Despite its conceptual novelty, this line of research remained relatively dormant for decades until it was revitalized through the systematic study by Guo~\mcite{Gub, Guo09}; see also~\mcite{BCQ10}. In his work, Guo~\mcite{Guo09} constructed free algebras with linear operators by leveraging various combinatorial structures, such as Motzkin paths, rooted forests, and bracketed words. These algebras, defined over a set $\Omega$ indexing the operators, are now commonly referred to as $\Omega$-operated algebras, or more broadly, multi-operated algebras.

The rooted forest is a significant object studied in combinatorics and algebra.
In 2016, Zhang, Gao, and Guo~\mcite{ZGG16} introduced the notion of operated Hopf algebras by combining the theories of operated algebras and Hopf algebras on decorated rooted forests.
Their study provided a systematic construction of free objects in this new category. This framework was subsequently expanded in~\mcite{ZGG22} using the Gr\"obner-Shirshov basis method, leading to the development of a general theory of $\Omega$-operated cocycle Hopf algebras.

A particularly noteworthy example within this framework is the Connes-Kreimer Hopf algebra on rooted forests. When equipped with the grafting operation $B^{+}$ which satisfies the  1-cocycle condition
\begin{equation}
	\Delta B^{+}:= B^{+} \otimes \etree+ (\id\otimes B^{+})\Delta,
\end{equation}
it exemplifies a specific case of an cocycle operated algebra. We would like to point out that Bruned, Hairer and Zambotti~\mcite{BHZ} used in  typed decorated rooted trees to give a description of a renormalisation procedure of stochastic PDEs and a generalize Connes-Kreimer Hopf algebra on typed decorated tooted trees was studied by Foissy~\mcite{Foi18}.

The Loday-Ronco Hopf algebra~\mcite{LR98} can also be studied within the unified framework of operated algebras. In ~\mcite{ZG20}, Zhang and Gao introduced the
concepts of $\vee_\Omega$-Hopf algebra and proved that the Loday-Ronco Hopf algebra can be characterized as the free multiple 1-cocycle $\vee_\Omega$-Hopf algebra generated by the empty set. Later, Marcolli, Berwick, and Chomsky~\mcite{MBC23} have identified applications of this algebra in the Minimalist program within generative linguistics. In~\cite[Page 72]{MBC25}, it was observed that the external merge in Stabler's Minimalism can be interpreted as a cocycle $\vee_\Omega$-algebraic structure on the Loday-Ronco Hopf algebra. Building on this perspective, Marcolli-Berwick-Chomsky formalized syntactic merge in the Minimalist Program as algebraic operations within cocycle Hopf algebras and further investigated its computational properties~\mcite{MBC23, MBC25}.

Hence, the construction of bialgebra structures on  rooted forests is of particular interest, not only due to their rich underlying algebraic properties but also because of their wide ranging applications in mathematical physics and logical linguistics.

In this paper we introduce the notion of a generalized Hochschild 1-cocycle condition of pair weight $(\lambda, \mu)$:
\begin{align*}
	\col B_{\omega}^{+}=-\lambda B_{\omega}^{+} \otimes \etree+\mu  \id \otimes \etree + (\id\otimes B_{\omega}^{+})\col, \forall \lambda, \mu \in \bfk,
\end{align*}
where $(B_{\omega}^{+})_{\omega\in \Omega}$ is a family of grafting operators.
When $(\lambda, \mu) = (0, 1)$, the condition reduces to the infinitesimal 1-cocycle condition, which was employed by Gao and Wang~\mcite{GW19} in the construction of infinitesimal unitary Hopf algebras on rooted forests.

When $(\lambda, \mu) = (-1, 0)$, it recovers the classical Hochschild 1-cocycle condition, which was utilized by Foissy and Holtkamp to construct a Hopf algebra~\mcite{Foi02, Hol03} as well as an infinitesimal bialgebra structure~\mcite{Foi09} on rooted forests.
In the case $(\lambda, \mu) = (\lambda, 0)$, it gives rise to the Hochschild 1-cocycle condition of weight $\lambda$, as introduced in~\mcite{ZCGL18}, which serves as the foundation for constructing weighted infinitesimal unitary bialgebras on rooted forests.
For $(\lambda, \mu) = (-1, \lambda)$, the condition was adopted by Zheng and Liu~\mcite{ZL25} in the construction of cocycle Hopf algebra structures on free modified Rota-Baxter algebras.
Finally, the case $(\lambda, \mu) = (0, 0)$ corresponds to the setting used in~\mcite{ZhG20} for the construction of left counital Hopf algebra structures on free commutative Nijenhuis algebras.

Having this generalized 1-cocycle condition in hand, we combine the $\Omega$-operated algebra $(\mathcal{H_{\mathrm{RT}}}(X, \Omega), \{B^+_\omega\mid \omega\in \Omega\})$
with weighted  unitary bialgebras and introduce the concept of $\Omega$-cocycle weighted bialgebras. We prove that the decorated planar rooted forests $\mathcal{H_{\mathrm{RT}}}(X, \Omega)$ is the free objects in these categories provided suitable operations are equipped. What's more, we give a combinatorial description of the coproduct of this cocycle weighted bialgebras by using Foissy' s rooted forests biideals.

\subsection{Pre-Lie algebras}

Our third source of inspiration stems from pre-Lie algebras, which exhibit profound connections across various domains in mathematics and mathematical physics. These include Lie groups and Lie algebras, the classical and quantum Yang-Baxter equations, vertex algebras, quantum field theory, and the theory of operads (see~\mcite{Bai} and the references therein). Also known as Vinberg algebras, pre-Lie algebras were originally introduced in Vinberg's study of convex homogeneous cones~\mcite{Vin63}, and emerged independently in the context of associative algebra deformations and cohomology~\mcite{Ger63}. A prominent example of their appearance is in perturbative quantum field theory~\mcite{Kre98}. In a notable contribution, Chapoton and Livernet~\mcite{CL01} described the operad associated with pre-Lie algebras in terms of rooted trees equipped with a grafting operation.

Our exploration of pre-Lie algebra structures on decorated planar rooted forests is driven not merely by an interest in infinitesimal bialgebras, but also by deeper combinatorial and algebraic motivations. Within the framework developed by Aguiar~\mcite{Aguu02} for infinitesimal bialgebras, one can construct a natural pre-Lie algebra structure from any given infinitesimal bialgebra. Building upon this approach, a pre-Lie algebra arising from a general infinitesimal unitary bialgebra of weight~$\lambda$ was formulated in~\mcite{CLPZ}. As an application, Theorem~\mref{thm:rt2} establishes a new pre-Lie structure defined on decorated planar rooted forests. Furthermore, leveraging the notion of forest biideals, a combinatorial interpretation of the binary bilinear operation defining this pre-Lie algebra is provided. These results emphasize the rich combinatorial underpinnings of pre-Lie algebras and open new avenues for their systematic study.

{\bf Structure of the Paper.}
In Section~\mref{sec:ibw}, we recall the concept of a weighted infinitesimal (unitary) bialgebra and show that some well-known algebras possess a weighted infinitesimal (unitary) bialgebra.

In Section~\mref{sec:infbi}, after summarizing concepts and basic facts on rooted forests, we give a new way to decorate planar rooted forests, which makes it possible to construct more general free objects.  By posing a pair weighted version of a Hochschild 1-cocycle condition (Eq.~(\mref{eq:cdbp})), we construct a new coproduct on decorated planar rooted forests $H_{\mathrm{RT}}(X,\Omega)$ to equip it with a new coalgebra structure (Theorem~\mref{thm:rt1}). Further $H_{\mathrm{RT}}(X,\Omega)$ can be turned into an infinitesimal unitary bialgebra of weight $\lambda$ with respect to the concatenation product and the empty tree as its unit (Theorem~\mref{thm:rt2}). By the forests biideals, we  give a combinatorial description of this new coproduct (Theorem~\mref{thm:comb}). We end this section by giving a description for dual products by construction of a bilinear symmetric form and a new grafting on rooted forests (Theorem~\mref{theoproduitdual}).

Combining weighted infinitesimal bialgebras with operated algebras, Section~\mref{uni} propose the concept of weighted $\Omega$-operated infinitesimal bialgebras (Definition~\mref{defn:xcobi}~(\mref{it:def1})).
When a pair weighted 1-cocycle condition is involved,  the concept of weighted $\Omega$-cocycle infinitesimal unitary bialgebras is also introduced~(Definition~\mref{defn:xcobi}~(\mref{it:def3})).
Thanks to these concepts, we show that $H_{\RT}(X,\Omega)$ is the free $\Omega$-cocycle infinitesimal unitary bialgebra of weight $\lambda$ on a set $X$ (Theorem~\mref{thm:propm}). As an application, we obtain that the undecorated planar rooted forests is the free cocycle infinitesimal unitary bialgebra of weight $\lambda$ on the empty set (Corollary~\mref{coro:rt16}).

In Section~\mref{sec:preLie}, by investigating the relationship between weighted infinitesimal unitary bialgebras and pre-Lie algebras (Theorem~\mref{thm:preL}), we equip $\hrts$ with a pre-Lie algebraic structure $(\hrts, \rhd_{\RT})$ and a Lie algebraic structure $(\hrts, [_{-}, _{-}]_{\RT})$ (Theorem~\mref{thm:preope}). The combinatorial descriptions of $\rhd_{\RT}$ and $[_{-}, _{-}]_{\RT}$ are also given (Corollary~\mref{coro:preLcomb}).

{\bf Notation.}
Throughout this paper, let $\bfk$ be a unitary commutative ring unless the contrary is specified,
which will be the base ring of all modules, algebras, coalgebras, bialgebras, tensor products, as well as linear maps.
By an algebra we mean an associative \bfk-algebra (possibly without unit)
and by a coalgebra we mean a coassociative \bfk-coalgebra (possibly without counit).
We use Sweedler notation:$$\Delta(a) = \sum_{(a)} a^{(1)} \ot a^{(2)}.$$
For an algebra $A$, $A\ot A$ is viewed as an $(A,A)$-bimodule in the standard way
\begin{equation}
	a\cdot(b\otimes c):=ab\otimes c\,\text{ and }\, (b\otimes c)\cdot a:= b\otimes ca,
	\mlabel{eq:dota}
\end{equation}
where $a,b,c\in A$.

\section{Weighted infinitesimal unitary bialgebras and some examples}\label{sec:ibw}

In this section, we begin by recalling the notion of weighted infinitesimal (unitary) bialgebras, as introduced in~\mcite{Fard06}. This framework serves as a unifying generalization of the constructions originally proposed by Joni and Rota~\mcite{JR}, as well as by Loday and Ronco~\mcite{LR06}.

\begin{defn}\mcite{ZCGL18}
	Let $\lambda$ be a given element of $\bfk$.
	\begin{enumerate}
		\item An {\bf infinitesimal bialgebra} (abbreviated {\bf $\epsilon$-bialgebra}) {\bf of weight $\lambda$} is a triple $(A,m,\Delta)$, where
		\begin{enumerate}
			\item $(A,m)$ is an algebra (possibly without unit),
			\item $(A,\Delta)$ is a coalgebra (possibly without counit),
		\end{enumerate}
		and the coproduct $\Delta$ satisfies the  weighted derivation rule on $A$ in the sense that
		\begin{equation}
			\Delta (ab)=a\cdot \Delta(b)+\Delta(a) \cdot b+\lambda (a\ot b),\quad \forall a, b\in A.
			\mlabel{eq:cocycle}
		\end{equation}
		
		\item If further $(A,m,1)$ is a unitary algebra and $(A,\Delta ,\varepsilon)$ is a counitary coalgebra , then the quintuple $(A, m, 1,\Delta, \varepsilon)$ is called an {\bf infinitesimal unitary counitary bialgebra} (abbreviated {\bf $\epsilon$-unitary counitary bialgebra}) {\bf of weight $\lambda$}.
		\mlabel{def:iub}
		\item Let $A$ and  $B$ be two $\epsilon$-bialgebras of weight $\lambda$.
		A map $\phi : A\rightarrow B$ is called an {\bf infinitesimal bialgebra morphism} (abbreviated $\epsilon$-bialgebra morphism) if $\phi$ is an algebra morphism and a coalgebra morphism. The concept of {\bf infinitesimal unitary bialgebra morphism} can be defined in the same way.
	\end{enumerate}
\end{defn}
We shall use the infix notation $\epsilon$- interchangeably with the adjective ``infinitesimal" throughout the rest of this paper.
\begin{remark}\label{remk:4rem}
	\begin{enumerate}
		\item \label{remk:units}Let $(A,m,1, \Delta)$ be an $\epsilon$-unitary bialgebra of weight $\lambda$. Then $\Delta(1)=-\lambda(1\ot1)$ by taking $a=b=1$ in Eq.~(\mref{eq:cocycle}).
		\item  $\epsilon$-bialgebras introduced by Joni and Rota~\mcite{JR} are  $\epsilon$-bialgebra of weight 0,
		and $\epsilon$-bialgebras originated from Loday and Ronco~\mcite{LR06} are $\epsilon$-bialgebra of weight -1.
		
		
	\end{enumerate}
\end{remark}


\begin{exam}\label{exam:bialgebras}
	Here are some examples of weighted $\epsilon$-unitary bialgebras.
	\begin{enumerate}
		\item Any unitary algebra $(A, \mu,1)$ is an $\epsilon$-unitary bialgebra of weight $\lambda$ by taking $$\Delta(a)=-\lambda(a\ot 1) \, \text{ for }\,  a\in A.$$
		
		\item \cite[Example~2.4]{CLPZ} The  polynomial algebra $\bfk [x]$ is an $\epsilon$-unitary bialgebra of weight $\lambda$ with the coproduct defined by
		\begin{align*}
			\Delta(1)&=-\lambda (1\ot 1) &&\text{ and }&
			\Delta(x^n)&=\sum_{i=0}^{n-1}x^{i}\ot x^{n-1-i}+\lambda \sum_{i=1}^{n-1}x^{i}\ot x^{n-i} \text { for } n\geq 1.
		\end{align*}
		
		\item \cite{ZZL19} The matrix algebra $M_n(\bfk)$ is an $\epsilon$-unitary bialgebra of weight zero with the coproduct defined by
		\begin{align*}
			\col (E_{ij}):=\begin{cases}\displaystyle\sum_{s=i}^{j-1}E_{is}\otimes E_{(s+1)j} &\text{ if } i< j,\\
				0 &\text{ if }i= j,\\
				\displaystyle-\sum_{s=j}^{i-1}E_{is}\otimes E_{(s+1)j} &\text{ if } i> j.
			\end{cases}
		\end{align*}
		
		\item \cite[Section~2.3]{LR06}\label{exam:tensor}
		Let $V$ denote a vector space. Recall that the tensor algebra $T(V)$ over $V$ is the tensor module,
		\begin{align*}
			T(V)=\bfk \oplus V\oplus V^{\ot 2}\oplus \cdots \oplus V^{\ot n}\oplus \cdots,
		\end{align*}
		equipped with the associative multiplication called concatenation defined by
		\begin{align*}
			v_1\cdots v_i\ot v_{i+1}\cdots v_n \mapsto v_1\cdots v_i v_{i+1}\cdots v_n \quad \text{ for } 0 \leq i\leq n,
		\end{align*}
		and with the convention that $v_1v_0=1$ and $v_{n+1}v_{n}=1$. It is a well-known free associative algebra. The
		tensor algebra $T(V)$ is an $\epsilon$-unitary bialgebra of weight $-1$ with the coassociative coproduct defined by
		\begin{align*}
			\Delta(v_1\cdots v_n):=\sum_{i=0}^{n}v_1\cdots v_i\ot v_{i+1}\cdots v_n.
		\end{align*}
		
		\item \cite[Example~2.4]{CLPZ}
		The free algebra $\bfk \langle X\rangle$ generated by a set $X$ (this algebra is isomorphic to tensor algebra $T(\bfk X)$) can be turned into an $\epsilon$-unitary bialgebra of weight $\lambda$ with the  coproduct defined by
		\begin{align*}
			\Delta(x_{1}x_{2}\cdots {x_{n}}):&= \sum_{i=1}^n{x_{1}}\cdots{x_{i-1}}\otimes{x_{i+1}}\cdots{x_{n}}+
			\lambda\sum_{i=1}^{n-1}{x_{1}}\cdots{x_{i}}\otimes{x_{i+1}}\cdots{x_{n}}.
		\end{align*}
	\end{enumerate}
\end{exam}

\section{Weighted infinitesimal unitary bialgebras of  rooted forests}
\label{sec:infbi}
In this section, we first recall the concepts of planar rooted forests~\mcite{Sta97} and decorated planar rooted forests~\mcite{Foi02,ZCGL18, Guo09}. We then equip the space of decorated planar rooted forests with a coproduct which makes it a weighted infinitesimal bialgebra. We give a combinatorial description of the coproduct by the concept of forests biideals. By construction of a bilinear symmetric form and a new grafting on rooted forests, we give a description for dual products.

\subsection{Decorated planar rooted forests}\mlabel{sucsec:deco}

A $\mathbf{rooted\ tree}$ is a finite, connected, and acyclic graph that contains a distinguished vertex referred to as the $\mathbf{root}$. A $\mathbf{planar\ rooted\ tree}$ is a rooted tree equipped with a fixed planar embedding, meaning its structure is preserved under a specific drawing in the plane.
The initial collection of planar rooted trees includes the following examples:
$$\tun,\ \tdeux,\ \ttroisun,\ \ttroisdeux,\ \tquatreun,\  \tquatretrois,\ \tquatredeux,\ \tquatrequatre,\ \tquatrecinq, \, \tcinqdeux,\tcinqtrois,\tcinqquatre$$
In each case, the root is positioned at the bottom of the tree.

We now recall several foundational definitions and properties of decorated planar rooted trees and forests that will be employed throughout this paper. For further details, see~\mcite{Foi02, ZCGL18, Guo09, PZGL}.

\begin{enumerate}
	\item Let $\calt$ denote the collection of planar rooted trees, and let $M(\calt)$ be the free monoid generated by $\calt$ under concatenation, denoted by $m_{\mathrm{RT}}$, which is usually omitted for simplicity. The identity element (also called the {\bf empty tree}) in $M(\calt)$ is denoted by $1$.
	
	\item Any element in $M(\calt)$, referred to as a {\bf planar rooted forest}, is a noncommutative concatenation of planar rooted trees. It is typically written in the form $F = T_1 \cdots T_n$ with $T_1, \ldots, T_n \in \calt$. We adopt the convention that $F = 1$ when $n = 0$.
	
	\item Let $\Omega$ be a nonempty set, and let $X$ be another set disjoint from $\Omega$. Define $\rts$ (resp.~$\rfs$) to be the set of planar rooted trees (resp.~forests) in which {\bf internal vertices} (i.e., non-leaf vertices) are decorated exclusively by elements of $\Omega$, while {\bf leaf vertices} may be decorated by elements from $X \sqcup \Omega$. The trivial tree $\bullet$, which contains only a single vertex, is regarded as having a leaf vertex. Elements in $\rfs$ are called {\bf decorated planar rooted forests}.
	
	\item Define
	\begin{align*}
		\hrts:= \bfk \rfs=\bfk M(\rts)
	\end{align*}
	as the {\bf free $\bfk$-module} with basis given by $\rfs$.
	
	\item For each $\omega \in \Omega$, define the linear operator
	$$B^+_\omega:\hrts\to \hrts$$
	to be the {\bf grafting operation}, which sends the empty forest $1$ to the tree $\bullet_\omega$, and for a general forest, returns the tree formed by attaching all roots to a new root decorated by $\omega$.
	
	\item Given $F = T_1 \cdots T_n \in \rfs$ with $n \geq 0$ and $T_1, \ldots, T_n \in \rts$, define $\bre(F) := n$ to be the {\bf breadth} of $F$. In particular, we set $\bre(\etree) = 0$ when $F$ is the empty forest.
	
	\item Let $\bullet_X := \{ \bullet_x \mid x \in X \}$ and define
	\begin{align*}
		\calf_0 := M(\bullet_X) = S(\bullet_X) \sqcup \{\etree\},
	\end{align*}
	where $M(\bullet_X)$ (resp.~$S(\bullet_X)$) denotes the submonoid (resp.~subsemigroup) of $\rfs$ generated by $\bullet_X$.
	Suppose that $\calf_n$ has been constructed for some $n \geq 0$. Then define
	\begin{align*}
		\calf_{n+1} := M\left(\bullet_X \sqcup \left( \bigsqcup_{\omega \in \Omega} B^+\omega(\calf_n) \right) \right).
	\end{align*}
	This gives rise to an ascending chain $\calf_n \subseteq \calf{n+1}$, and we obtain
	\begin{align*}
		\rfs = \lim_{\longrightarrow} \calf_n = \bigcup_{n=0}^{\infty} \calf_n.
	\end{align*}
	An element $F \in \calf_n \setminus \calf_{n-1}$ is said to have {\bf depth} $n$, which we denote by $\dep(F) = n$.
\end{enumerate}

\begin{exam}
	The following are some examples in $\rts$:
	$$\tdun{$\alpha$},\ \, \tdun{$x$},\ \, \tddeux{$\alpha$}{$\beta$},\ \,  \tddeux{$\alpha$}{$x$}, \ \, \tdtroisun{$\alpha$}{$\beta$}{$\gamma$},\ \,\tdtroisun{$\alpha$}{$x$}{$\gamma$}, \ \,\tdtroisun{$\alpha$}{$x$}{$y$}, \ \, \tdquatretrois{$\alpha$}{$\beta$}{$\gamma$}{$\beta$},\ \, \tdquatretrois{$\alpha$}{$\beta$}{$\gamma$}{$x$}, \ \, \tdquatretrois{$\alpha$}{$\beta$}{$x$}{$y$},$$
	with $\alpha,\beta,\gamma\in \Omega$ and $x, y \in X$.
\end{exam}

\begin{exam}
	The following are some grafting operations:
	\begin{align*}
		B_{\omega}^{+}(\etree)&=\tdun{$\omega$}\:,& B_{\omega}^{+}(\tdun{$x$}\tddeux{$\alpha$}{$y$})&=\tdquatretrois{$\omega$}{$\alpha$}{$y$}{$x$},&  B_{\omega}^{+}(\tddeux{$\beta$}{$\alpha$}\tdun{$x$})&=
		\tdquatredeux{$\omega$}{$x$}{$\beta$}{$\alpha$},
	\end{align*}
	where $\alpha, \beta, \omega\in \Omega$ and $x, y \in X$.
\end{exam}

\begin{exam}
	Here are some examples about the depths of some decorated planar rooted forests.
	\begin{align*}
		\dep(\etree) =\ \dep(\bullet_x) &=0,& \dep(\bullet_\omega)=\dep(B^+_{\omega}(\etree)) &= 1,\\
		\dep(\tddeux{$\omega$}{$\alpha$})= \dep(B^+_{\omega}(B^+_{\alpha}(\etree))) &=2, &
		\dep(\tdun{$x$}\tddeux{$\omega$}{$y$}\tdun{$y$}) =\ \dep(\tddeux{$\omega$}{$y$})=
		\dep(B^+_{\omega}(\bullet_y)) &=1,\\ \dep(\tdtroisun{$\omega$}{$x$}{$\alpha$}) = \dep(B^+_{\omega}(B^+_{\alpha}(\etree) \bullet_x)) &= 2,
	\end{align*}
	where $\alpha,\omega\in \Omega$ and $x, y \in X$.
\end{exam}

\begin{remark}\mcite{ZCGL18}\label{re:3ex}
	We now present several special cases of the decorated planar rooted forests in $\mathcal{F}(X, \Omega)$:
	
	\begin{enumerate}
		\item If $X = \emptyset$ and $\Omega$ is a singleton set, then all decorated planar rooted forests in $\mathcal{F}(X, \Omega)$ share the same decoration. Hence, the decorations can be omitted, and the forests can be identified with planar rooted forests without decorations. This case corresponds to the well-known Foissy-Holtkamp Hopf algebra, which is a noncommutative analogue of the Connes-Kreimer Hopf algebra~\mcite{Foi02, Hol03}.
		
		\item If $X = \emptyset$, then $\mathcal{F}(X, \Omega)$ reduces to $\Omega$-decorated planar rooted forests with no leaf decorations. This setting was studied by the first author~\mcite{Foi02}, where a decorated noncommutative version of the Connes-Kreimer Hopf algebra was constructed. \mlabel{it:2ex}
		
		\item If $\Omega$ is a singleton set, then $\rfs$ consists of planar rooted forests decorated only at the leaves. This case was introduced and studied in~\mcite{ZGG16} to construct a cocycle Hopf algebra structure on decorated planar rooted forests.
	\end{enumerate}
\end{remark}

\subsection{Weighted infinitesimal unitary bialgebras on decorated planar rooted forests}
In this subsection, we shall define a weighted infinitesimal unitary bialgebraic structure on decorated planar rooted forests.

Let $\lambda, \mu $ be  given elements of $\bfk$. We now define a new coproduct $\col$ on $\hrts$ by induction on depth.
By linearity, we only need to define $\col(F)$ for basis elements $F\in \rfs$.
For the initial step of $\dep(F)=0$, we  define
%
\begin{equation}
	\col(F) :=
	\left\{
	\begin{array}{ll}
		-\lambda (\etree \ot \etree) & \text{ if } F = \etree, \\
		\mu (\etree \ot \etree)-\lambda(\bullet_{x}\ot \etree+\etree \ot \bullet_{x}) & \text{ if } F = \bullet_x \text{ for some } x \in X,\\
		\bullet_{x_{1}}\cdot \col(\bullet_{x_{2}}\cdots\bullet_{x_{m}})+ \col(\bullet_{x_{1}}) \cdot (\bullet_{x_{2}}\cdots\bullet_{x_{m}})\\
		+\lambda\bullet_{x_{1}}\ot \bullet_{x_{2}}\cdots\bullet_{x_{m}}
		& \text{ if }  F=\bullet_{x_{1}}\cdots \bullet_{x_{m}} \text{ with } m\geq 2.
	\end{array}
	\right.
	\mlabel{eq:dele}
\end{equation}

For the induction step of $\dep(F)\geq 1$, we reduce the definition to induction on breadth.
If $\bre(F) = 1$, we
write $F=B_{\omega}^{+}(\lbar{F})$ for some $\omega\in \Omega$ and $\lbar{F}\in \rfs$, and define
\begin{equation}
	\col(F)=\col B_{\omega}^{+}(\lbar{F}) :=-\lambda B_{\omega}^{+}(\lbar{F}) \otimes \etree+ \mu\lbar{F} \otimes \etree + (\id\otimes B_{\omega}^{+})\col(\lbar{F}).
	\mlabel{eq:dbp}
\end{equation}
In other words
\begin{align}
	\col B_{\omega}^{+}=-\lambda B_{\omega}^{+} \otimes \etree+\mu  \id \otimes \etree + (\id\otimes B_{\omega}^{+})\col.
	\mlabel{eq:cdbp}
\end{align}
we call Eq.~(\mref{eq:cdbp}) the {\bf pair weight $(\lambda, \mu)$ of 1-cocycle condition}.
If $\bre(F) \geq 2$, we write $F=T_{1}T_{2}\cdots T_{m}$ with $m\geq 2$ and $T_1, \ldots, T_m \in \rts$, and define
\begin{equation}
	\col(F)=T_{1}\cdot \col(T_{2}\cdots T_{m})+\col(T_{1})\cdot (T_{2}\cdots T_{m})+\lambda T_{1} \ot T_{2}\cdots T_{m},
	\mlabel{eq:delee1}
\end{equation}
where the left and right actions are given in Eq.~(\mref{eq:dota}).

\begin{exam}\mlabel{exam:cop}
	Let $x,y \in X$ and $\alpha, \beta \in \Omega$. Then
	\begin{align*}
		\col(\tdun{$x$})=&\ \mu(\etree \ot \etree)-\lambda(\tdun{$x$}\ot \etree+\etree \ot \tdun{$x$}),\\
		\col(\tdun{$\alpha$})=&\ \mu(\etree \ot \etree)-\lambda(\tdun{$\alpha$}\ot \etree+\etree \ot \tdun{$\alpha$}),\\
		\col(\tddeux{$\alpha$}{$x$})=&\ \mu (\tdun{$x$} \ot \etree + \etree \ot \tdun{$\alpha$})
		-\lambda(\tddeux{$\alpha$}{$x$}\ot \etree +\tdun{$x$}\ot \tdun{$\alpha$}+ \etree \ot \tddeux{$\alpha$}{$x$}),\\
		\col(\tdun{$y$}\tddeux{$\alpha$}{$x$})=&\ \mu (\tdun{$y$}\tdun{$x$} \ot \etree + \tdun{$y$}\ot\tdun{$\alpha$}+\etree \ot \tddeux{$\alpha$}{$x$})\\
		&-\lambda (\tdun{$y$}\tddeux{$\alpha$}{$x$}\ot \etree + \tdun{$y$}\tdun{$x$}\ot \tdun{$\alpha$}
		+\tdun{$y$}\ot \tddeux{$\alpha$}{$x$}+1\ot \tdun{$y$}\tddeux{$\alpha$}{$x$}),\\
		\col(\tdquatretrois{$\alpha$}{$\beta$}{$x$}{$y$})=&\mu(\etree\ot \tdtroisdeux{$\alpha$}{$\beta$}{$x$}+\tdun{$y$}\ot \tddeux{$\alpha$}{$\beta$}+\tdun{$y$}\tdun{$x$}\ot \tdun{$\alpha$}+ \tdun{$y$}\tddeux{$\beta$}{$x$}\ot \etree)\\
		&-\lambda(\tdquatretrois{$\alpha$}{$\beta$}{$x$}{$y$}\ot \etree +\tdun{$y$}\tddeux{$\beta$}{$x$}\ot \tdun{$\alpha$}+\tdun{$y$}\tdun{$x$}\ot\tddeux{$\alpha$}{$\beta$}+\tdun{$y$}\ot\tdtroisdeux{$\alpha$}{$\beta$}{$x$}+\etree \otimes \tdquatretrois{$\alpha$}{$\beta$}{$x$}{$y$} ).
		%
	\end{align*}
\end{exam}

To show $(\hrts, \col)$ is a coalgebra, we record the following two lemmas as a preparation.

\begin{lemma}\label{lem:rt11}
	Let $\bullet_{x_{1}}\cdots \bullet_{x_{m}}\in \hrts$ with $m\geq 1$ and $x_1, \ldots, x_m\in X$. Then
	$$\col(\bullet_{x_{1}}\cdots \bullet_{x_{m}})=\mu \sum_{i=1}^{m}\bullet_{x_{1}}\cdots\bullet_{x_{i-1}}\otimes
	\bullet_{x_{i+1}}\cdots\bullet_{x_{m}}
	-\lambda \sum_{i=0}^{m}\bullet_{x_{1}}\cdots\bullet_{x_{i}}\otimes
	\bullet_{x_{i+1}}\cdots\bullet_{x_{m}},$$
	with the convention that $\bullet_{x_{1}}\bullet_{x_{0}}=1$ and $\bullet_{x_{m+1}}\bullet_{x_{m}}=1$.
\end{lemma}

\begin{proof}
	We prove the result by induction on $m\geq 1$. For the initial step of $m=1$, we have
	$$\col(\bullet_{x_{1}})=\mu (\etree \ot \etree)-\lambda(\bullet_{x_{1}} \ot \etree+\etree\ot\bullet_{x_{1}}) ,$$
	and the result is true trivially. For the induction step of $m\geq 2$, we get
	\allowdisplaybreaks{
		\begin{align*}
			\col(\bullet_{x_{1}}\cdots \bullet_{x_{m}})
			=&\ \bullet_{x_{1}} \cdot \col(\bullet_{x_{2}}\cdots\bullet_{x_{m}})+ \col(\bullet_{x_{1}})\cdot (\bullet_{x_{2}}\cdots\bullet_{x_{m}})
			+\lambda(\bullet_{x_{1}}\ot \bullet_{x_{2}}\cdots\bullet_{x_{m}})  \quad (\text{by Eq.~(\ref{eq:dele})})\\
			=&\ \bullet_{x_{1}} \cdot \col(\bullet_{x_{2}}\cdots\bullet_{x_{m}})+\Big(\mu (\etree \ot \etree)-\lambda (\bullet_{x_{1}}\ot \etree+\etree\ot \bullet_{x_{1}})\Big)\cdot (\bullet_{x_{2}}\cdots\bullet_{x_{m}})\\
			& +\lambda(\bullet_{x_{1}}\ot \bullet_{x_{2}}\cdots\bullet_{x_{m}}) \quad (\text{by Eq.~(\mref{eq:dele})})\\
			=&\ \bullet_{x_{1}} \cdot \col(\bullet_{x_{2}}\cdots\bullet_{x_{m}})
			+\mu(\etree \ot \bullet_{x_{2}} \cdots\bullet_{x_{m}})
			-\lambda(\bullet_{x_{1}}\ot \bullet_{x_{2}} \cdots\bullet_{x_{m}})
			\\
			&-\lambda(\etree \otimes\bullet_{x_{1}}\bullet_{x_{2}} \cdots\bullet_{x_{m}}) +\lambda(\bullet_{x_{1}}\ot \bullet_{x_{2}}\cdots\bullet_{x_{m}}) \quad (\text{by Eq.~(\mref{eq:dota})})\\
			=&\ \bullet_{x_{1}} \cdot \col(\bullet_{x_{2}}\cdots\bullet_{x_{m}})
			+\mu(\etree \ot \bullet_{x_{2}} \cdots\bullet_{x_{m}})
			-\lambda(\etree \otimes\bullet_{x_{1}}\bullet_{x_{2}} \cdots\bullet_{x_{m}})\\
			=&\ \bullet_{x_{1}} \cdot \left( \mu \sum_{i=2}^{m}\bullet_{x_{2}}\cdots\bullet_{x_{i-1}}\otimes\bullet_{x_{i+1}}\cdots\bullet_{x_{m}}
			-\lambda
			\sum_{i=1}^{m}\bullet_{x_{2}}\cdots\bullet_{x_{i}}\otimes\bullet_{x_{i+1}}
			\cdots\bullet_{x_{m}}\right)\\
			&\ +\mu(\etree \ot \bullet_{x_{2}} \cdots\bullet_{x_{m}})
			-\lambda(\etree \otimes\bullet_{x_{1}} \cdots\bullet_{x_{m}})
			\quad (\text{by the induction hypothesis})\\
			=&\ \mu \sum_{i=2}^{m}\bullet_{x_{1}} \bullet_{x_{2}}\cdots\bullet_{x_{i-1}}\otimes\bullet_{x_{i+1}}\cdots\bullet_{x_{m}}
			-\lambda
			\sum_{i=1}^{m}\bullet_{x_{1}} \bullet_{x_{2}}\cdots\bullet_{x_{i}}\otimes\bullet_{x_{i+1}}
			\cdots\bullet_{x_{m}}\\
			&\ +\mu(\etree \ot \bullet_{x_{2}} \cdots\bullet_{x_{m}})
			-\lambda(\etree \otimes\bullet_{x_{1}} \cdots\bullet_{x_{m}})\\
			=&\ \mu \sum_{i=1}^{m}\bullet_{x_{1}}\cdots\bullet_{x_{i-1}}\otimes\bullet_{x_{i+1}}\cdots\bullet_{x_{m}}
			-\lambda
			\sum_{i=0}^{m}\bullet_{x_{1}}\cdots\bullet_{x_{i}}\otimes\bullet_{x_{i+1}}\cdots\bullet_{x_{m}},
		\end{align*}
	}
	as required.
\end{proof}

\begin{lemma}\label{lem:colff}
	Let $F_1, F_2\in \hrts$. Then
	$$\col(F_1 F_2) = F_1 \cdot \col(F_2) + \col(F_1) \cdot F_2+\lambda (F_1\ot F_2).$$
\end{lemma}

\begin{proof}
	It suffices to consider basis elements $F_1, F_2\in \rfs$ by linearity, and we shall use this technique tacitly in the later proofs of this paper.
	If $\bre(F_1)=0$ or $\bre(F_2)=0$, without loss of generality, letting $\bre(F_1)=0$, then $F_1=1$ and by Eq.~(\mref{eq:dele}),
	\begin{align*}
		\col(F_1 F_2)=\col(1 F_2)&=\col(F_2)-\lambda(1\ot F_2)+\lambda(1\ot F_2)\\
		&=\col(F_2)-\lambda(1\ot 1)\cdot F_2+\lambda(1\ot F_2)\\
		&=\col(F_2)+\col(1)\cdot F_2+\lambda(1\ot F_2)\\
		&=1\cdot\col(F_2)+\col(1)\cdot F_2+\lambda(1\ot F_2).
	\end{align*}
	If $\bre(F_1)\geq1$ and $\bre(F_2)\geq1$, we proceed to prove the result by induction on the sum of breadths $\bre(F_1)+ \bre(F_2)\geq 2$. For the initial step of $\bre(F_1)+ \bre(F_2)= 2$, we have $F_1=T_1$ and $F_2=T_2$ for some decorated planar rooted trees $T_1, T_2\in \rts$. Using Eq.~(\mref{eq:delee1}), we have
	\begin{align*}
		\col(F_1 F_2) =\col(T_1 T_2)&= T_1 \cdot \col(T_2) + \col(T_1) \cdot T_2+\lambda (T_1\ot T_2)\\
		&=F_1 \cdot \col(F_2) + \col(F_1) \cdot F_2+\lambda (F_1\ot F_2).
	\end{align*}
	For the induction step of $\bre(F_1)+ \bre(F_2)\geq 3$, without loss of generality, we may suppose $\bre(F_2)\geq \bre(F_1)\geq 1$. If $\bre(F_1)=1$ and $\bre(F_2)\geq2$ , we may write $F_1=T_1$ for some decorated planar rooted trees $T_1\in \rts$. Applying Eq.~(\mref{eq:delee1}), we have
	\begin{align*}
		\col(F_1 F_2) =\col(T_1 F_2)&= T_1 \cdot \col(F_2) + \col(T_1) \cdot F_2+\lambda (T_1\ot F_2)\\
		&=F_1 \cdot \col(F_2) + \col(F_1) \cdot F_2+\lambda (F_1\ot F_2).
	\end{align*}
	If $\bre(F_1)\geq2$, we can
	write $F_1=T_{1}{F_{1}}'$ with $\bre(T_{1})=1$ and $\bre({F_{1}}') = \bre(F_1) -1 $. Then
	\begin{align*}
		&\col(F_1 F_2)=\col(T_{1}{F_{1}}' F_2)\\
		&=T_{1}\cdot\col({F_{1}}'F_2)+\col(T_{1})\cdot({F_{1}}'F_2)+\lambda(T_1\ot {F_{1}}'F_2)\quad (\text{by Eq.~(\mref{eq:delee1})})\\
		&=T_{1}\cdot\Big({F_{1}}' \cdot \col(F_2) + \col({F_{1}}') \cdot F_2+\lambda({F_{1}}'\ot F_2)\Big)+\col(T_{1})\cdot({F_{1}}'F_2)+\lambda(T_1\ot {F_{1}}'F_2) \\
		&\hspace{8cm} (\text{by\ the\ induction\ hypothesis})\\
		&=(T_{1}{F_{1}}')\cdot\col(F_2)+T_{1}\cdot\col({F_{1}}')\cdot F_2+\lambda(T_1{F_{1}}'\ot F_2)+\col(T_{1})\cdot ({F_{1}}' F_2)+\lambda(T_1\ot {F_{1}}'F_2) \\
		&=(T_{1}{F_{1}}')\cdot \col(F_2)+\Big(T_{1}\cdot \col({F_{1}}') +\col(T_{1}) \cdot {F_{1}}'+\lambda (T_1\ot {F_{1}}' )\Big)\cdot F_2+\lambda(T_1{F_{1}}'\ot F_2)\\
		&=(T_{1}{F_{1}}')\cdot \col(F_2)+\col(T_1{F_{1}}')\cdot F_2+\lambda(T_1{F_{1}}'\ot F_2)\\
		&=F_1\cdot \col(F_2)+\col(F_1)\cdot F_2 +\lambda(F_1\ot F_2)\quad \quad \ \ \text{(by\ the\ induction\ hypothesis)}.
	\end{align*}
	This completes the proof.
\end{proof}

The following lemma shows that $\hrts$ is closed under the coproduct $\col$.
\begin{lemma} \label{lem:cclosed}
	For $F\in \hrts$,
	\begin{equation}
		\col(F)\in \hrts\ot\hrts.
		\mlabel{fact:closed}
	\end{equation}
\end{lemma}



\begin{proof}
	We prove the result by induction on $\dep(F)\geq 0$ for basis elements $F\in \rfs$.
	For the initial step of $\dep(F)=0$,
	we have $F=\bullet_{x_1}\ldots\bullet_{x_{m}}$ for some $m\geq 0$, with the convention that $F=1$ when $m=0$.
	When $m=0$, we have
	\begin{align*}
		\col(F)=\col(1)=-\lambda(1\ot 1)\in  \hrts \ot \hrts.
	\end{align*}
	When $m\geq 1$, by Lemma~\mref{lem:rt11}, we obtain
	\begin{align*}
		\col(F)=\col(\bullet_{x_1}\ldots\bullet_{x_{m}})
		=& \mu \sum_{i=1}^{m}\bullet_{x_{1}}\cdots\bullet_{x_{i-1}}\otimes\bullet_{x_{i+1}}\cdots\bullet_{x_{m}}
		-\lambda
		\sum_{i=0}^{m}\bullet_{x_{1}}\cdots\bullet_{x_{i}}\otimes\bullet_{x_{i+1}}\cdots\bullet_{x_{m}}\\
		&\in\hrts \ot \hrts.
	\end{align*}
	
	Suppose that Eq.~(\mref{fact:closed}) holds for $\dep(F) \leq n$ for an $n\geq 0$ and consider the case of $\dep(F) =n+1$.
	For this case, we apply the induction on breadth $\bre(F)$. Since $\dep(F)=n+1 \geq 1$, we get $F\neq 1$ and $\bre(F)\geq 1$.
	If $\bre(F)=1$, since $\dep(F)\geq 1$, we have $F = B_{\omega}^{+}(\lbar{F})$ for some $\omega\in \Omega$ and $\lbar{F} \in \hrts$. By Eq.~(\mref{eq:dbp}), we have
	\begin{align*}
		\col(F)= \col\Big(B_{\omega}^{+}(\lbar{F})\Big) =-\lambda (F\ot 1)+\mu(\lbar{F}\ot 1)+(\id\ot B_{\omega}^{+})\col(\lbar{F}).
	\end{align*}
	By the induction hypothesis on $\dep(F)$,
	$$\col(\lbar{F})\in \hrts \ot \hrts \,\text{ and so }\, (\id\ot B_{\omega}^{+})\col(\lbar{F})\in \hrts\ot\hrts.$$
	Moreover,  $-\lambda(F\ot 1)\in \hrts \ot \hrts$ and $\mu(\lbar{F}\ot 1)\in \hrts \ot \hrts$ follows from $F\in \hrts$ and $\lbar{F}\in \hrts$.
	Hence
	\begin{align*}
		-\lambda({F}\ot 1)+\mu(\lbar{F}\ot 1)+(\id\ot B_{\omega}^{+})\col(\lbar{F})\in \hrts\ot\hrts.
	\end{align*}
	
	Assume that Eq.~(\mref{fact:closed}) holds for  $\dep(F) = n+1$ and $\bre(F)\leq m$, in addition to $\dep(F)\leq n$ by the first induction hypothesis, and consider the case of
	$\dep(F) = n+1$ and $\bre(F) =m+1\geq 2$.
	Then we may write $F=T_{1}T_{2}\cdots T_{m+1}$ for some $T_1,\ldots , T_{m+1} \in \rts$ and so
	\begin{align*}
		\col(F)=&\col(T_{1}T_{2}\cdots T_{m+1})\\
		=&\ T_{1}\cdot \col(T_{2}\cdots T_{m+1})+\col(T_{1})\cdot (T_{2}\cdots T_{m+1})+\lambda(T_1\ot T_2\cdots T_{m+1}) \quad (\text{by Eq.~(\mref{eq:delee1})})\\
		=&\ \cdots=\sum_{i =1}^{m+1} (T_{1} \cdots T_{i-1}) \cdot\col(T_{i}) \cdot (T_{i+1}\cdots T_{m+1})+\lambda \sum_{i =1}^{m} T_{1} \cdots T_{i}\ot T_{i+1}\cdots T_{m+1},
	\end{align*}
	with the convention that $T_1T_0=1$ and $T_{m+2}T_{m+1}=1$.
	By the induction hypothesis on breadth, we have
	$$ \col(T_i) \in \hrts\ot \hrts,$$
	hence by Eq.~(\mref{eq:dota}),
	$$\sum_{i =1}^{m+1} (T_{1} \cdots T_{i-1}) \cdot\col(T_{i}) \cdot (T_{i+1}\cdots T_{m+1})\in \hrts\ot \hrts.$$
	Thus
	$$\col(F)=\col(T_{1}T_{2}\cdots T_{m+1})\in \hrts\ot \hrts.$$
	This completes the induction on the breadth and hence the induction on the depth.
\end{proof}

We now state our first main result in this subsection.

\begin{theorem}
	The pair $(\hrts, \col)$ is a coalgebra (maybe without counit).
	\mlabel{thm:rt1}
\end{theorem}

\begin{proof}
	By Lemma~\mref{lem:cclosed}, we only need to verify the the  coassociative law
	\begin{equation}
		(\id\otimes \col)\col(F)=(\col\otimes \id)\col(F)\,\text{ for } F\in \rfs,
		\mlabel{eq:coass}
	\end{equation}
	which will be proved by induction on $\dep(F)\geq 0$.
	For the initial step of $\dep(F)=0$, we have $F=\bullet_{x_{1}}\bullet_{x_{2}}\cdots\bullet_{x_{m}}$ for some $m\geq 0$, with the convention that $F=\etree$ if $m=0$.
	When $m=0$, we have
	\begin{align*}
		(\id\otimes \col)\col(F)&=(\id\otimes \col)\col(1)=-\lambda1\otimes \col(1)=\lambda^2 (1\ot 1\ot 1)=-\lambda \col(1)\ot 1\\
		&=(\col\otimes \id)\col(1).
	\end{align*}
	When $m\geq 1$, on the one hand,
	\allowdisplaybreaks{
		\begin{align*}
			&(\id\ot\col)\col(\bullet_{x_{1}}\bullet_{x_{2}}\cdots\bullet_{x_{m}})\\
			=&\ (\id\ot\col)\left(\mu \sum_{i=1}^{m}\bullet_{x_{1}}\cdots\bullet_{x_{i-1}}\otimes\bullet_{x_{i+1}}\cdots\bullet_{x_{m}}
			-\lambda
			\sum_{i=0}^{m}\bullet_{x_{1}}\cdots\bullet_{x_{i}}
			\otimes\bullet_{x_{i+1}}\cdots\bullet_{x_{m}}\right)\quad (\text{by Lemma~\ref{lem:rt11}})\\
			=& \mu \sum_{i=1}^{m-1}\bullet_{x_{1}}\cdots\bullet_{x_{i-1}}\otimes \col(\bullet_{x_{i+1}}\cdots\bullet_{x_{m}})
			+\mu\bullet_{x_{1}}\cdots\bullet_{x_{m-1}}\ot (-\lambda (\etree\ot \etree))-\lambda\bullet_{x_{1}}\cdots\bullet_{x_{m}}\ot (-\lambda (\etree\ot \etree))
			\\
			&-\lambda\sum_{i=0}^{m-1}\bullet_{x_{1}}\cdots\bullet_{x_{i}}
			\otimes\col(\bullet_{x_{i+1}}\cdots\bullet_{x_{m}})\\
			=&\mu \sum_{i=1}^{m-1}\bullet_{x_{1}}\cdots\bullet_{x_{i-1}}\otimes\left(\mu \sum_{j=i+1}^m \bullet_{x_{i+1}}\cdots \bullet_{x_{j-1}} \ot \bullet_{x_{j+1}}\cdots \bullet_{x_{m}}
			-\lambda\sum_{j={i}}^m \bullet_{x_{i+1}}\cdots \bullet_{x_{j}} \ot \bullet_{x_{j+1}}\cdots \bullet_{x_{m}}\right)
			\\
			&-\lambda\sum_{i=0}^{m-1}\bullet_{x_{1}}\cdots\bullet_{x_{i}}
			\otimes \left(\mu \sum_{j=i+1}^m \bullet_{x_{i+1}}\cdots \bullet_{x_{j-1}} \ot \bullet_{x_{j+1}}\cdots \bullet_{x_{m}}
			-\lambda\sum_{j={i}}^m \bullet_{x_{i+1}}\cdots \bullet_{x_{j}} \ot \bullet_{x_{j+1}}\cdots \bullet_{x_{m}}\right)\\
			&+\lambda^2 (\bullet_{x_{1}}\cdots\bullet_{x_{m}}\ot \etree\ot \etree)-\lambda\mu(\bullet_{x_{1}}\cdots\bullet_{x_{m-1}}\ot \etree\ot \etree)\\
			=
			&\mu^2\sum_{i=1}^{m-1}\sum_{j=i+1}^{m}\bullet_{x_{1}}\cdots\bullet_{x_{i-1}}
			\ot  \bullet_{x_{i+1}}\cdots \bullet_{x_{j-1}}\ot \bullet_{x_{j+1}}\cdots \bullet_{x_{m}}\\
			&-\lambda\mu\sum_{i=1}^{m}\sum_{j=i}^{m}\bullet_{x_{1}}\cdots\bullet_{x_{i-1}}
			\ot  \bullet_{x_{i+1}}\cdots \bullet_{x_{j}}\ot \bullet_{x_{j+1}}\cdots \bullet_{x_{m}}\\
			&-\lambda\mu\sum_{i=0}^{m-1}\sum_{j=i+1}^{m}\bullet_{x_{1}}\cdots\bullet_{x_{i}}
			\ot  \bullet_{x_{i+1}}\cdots \bullet_{x_{j-1}}\ot \bullet_{x_{j+1}}\cdots \bullet_{x_{m}}\\ &+\lambda^2\sum_{i=0}^{m}\sum_{j=i}^{m}\bullet_{x_{1}}\cdots\bullet_{x_{i}}
			\ot  \bullet_{x_{i+1}}\cdots \bullet_{x_{j}}\ot \bullet_{x_{j+1}}\cdots \bullet_{x_{m}}.
		\end{align*}
	}
	On the other hand,
	\allowdisplaybreaks{
		\begin{align*}
			&(\col \ot \id)\col(\bullet_{x_{1}}\bullet_{x_{2}}\cdots\bullet_{x_{m}})\\
			=&(\col \ot \id)\left(\mu \sum_{i=1}^{m}\bullet_{x_{1}}\cdots\bullet_{x_{i-1}}\otimes\bullet_{x_{i+1}}\cdots\bullet_{x_{m}}
			-\lambda
			\sum_{i=0}^{m}\bullet_{x_{1}}\cdots\bullet_{x_{i}}
			\otimes\bullet_{x_{i+1}}\cdots\bullet_{x_{m}}\right)\quad (\text{by Lemma~\ref{lem:rt11}})\\
			=& \mu \sum_{i=1}^{m}\col(\bullet_{x_{1}}\cdots\bullet_{x_{i-1}})\otimes \bullet_{x_{i+1}}\cdots\bullet_{x_{m}}
			+\lambda^2 (\etree\ot \etree \ot \bullet_{x_{1}}\cdots \bullet_{x_{m}})
			-\lambda\sum_{i=1}^{m}\col(\bullet_{x_{1}}\cdots\bullet_{x_{i}})
			\otimes\bullet_{x_{i+1}}\cdots\bullet_{x_{m}}\\
			=& \mu \sum_{i=1}^{m}\left(\mu \sum_{j=1}^{i-1}  \bullet_{x_{1}}\cdots \bullet_{x_{j-1}} \ot \bullet_{x_{j+1}}\cdots \bullet_{x_{i-1}} -\lambda\sum_{j=0}^{i-1} \bullet_{x_{1}}\cdots \bullet_{x_{j}} \ot \bullet_{x_{j+1}}\cdots \bullet_{x_{i-1}} \right)
			\otimes \bullet_{x_{i+1}}\cdots\bullet_{x_{m}}\\
			& -\lambda\sum_{i=1}^{m}\left(\mu \sum_{j=1}^{i} \bullet_{x_{1}}\cdots \bullet_{x_{j-1}} \ot \bullet_{x_{j+1}}\cdots \bullet_{x_{i}} -\lambda\sum_{j=0}^{i} \bullet_{x_{1}}\cdots \bullet_{x_{j}} \ot \bullet_{x_{j+1}}\cdots \bullet_{x_{i}} \right)
			\otimes\bullet_{x_{i+1}}\cdots\bullet_{x_{m}}\\
			& +\lambda^2 (\etree\ot \etree \ot \bullet_{x_{1}}\cdots \bullet_{x_{m}})
			\\
			=&\mu^2\sum_{i=1}^{m}\sum_{j=1}^{i-1}\bullet_{x_{1}}\cdots\bullet_{x_{j-1}}
			\ot  \bullet_{x_{j+1}}\cdots \bullet_{x_{i-1}}\ot \bullet_{x_{i+1}}\cdots \bullet_{x_{m}}\\
			&-\lambda\mu\sum_{i=1}^{m}\sum_{j=0}^{i-1}\bullet_{x_{1}}\cdots\bullet_{x_{j}}
			\ot  \bullet_{x_{j+1}}\cdots \bullet_{x_{i-1}}\ot \bullet_{x_{i+1}}\cdots \bullet_{x_{m}}\\
			&-\lambda\mu\sum_{i=1}^{m}\sum_{j=1}^{i}\bullet_{x_{1}}\cdots\bullet_{x_{j-1}}
			\ot  \bullet_{x_{j+1}}\cdots \bullet_{x_{i}}\ot \bullet_{x_{i+1}}\cdots \bullet_{x_{m}}\\ &+\lambda^2\sum_{i=0}^{m}\sum_{j=0}^{i}\bullet_{x_{1}}\cdots\bullet_{x_{j}}
			\ot  \bullet_{x_{j+1}}\cdots \bullet_{x_{i}}\ot \bullet_{x_{i+1}}\cdots \bullet_{x_{m}}\\
			=&\mu^2\sum_{j=1}^{m}\sum_{i=1}^{j-1}\bullet_{x_{1}}\cdots\bullet_{x_{i-1}}
			\ot  \bullet_{x_{i+1}}\cdots \bullet_{x_{j-1}}\ot \bullet_{x_{j+1}}\cdots \bullet_{x_{m}}\\
			&-\lambda\mu\sum_{j=1}^{m}\sum_{i=0}^{j-1}\bullet_{x_{1}}\cdots\bullet_{x_{i}}
			\ot  \bullet_{x_{i+1}}\cdots \bullet_{x_{j-1}}\ot \bullet_{x_{j+1}}\cdots \bullet_{x_{m}}\\
			&-\lambda\mu\sum_{j=1}^{m}\sum_{i=1}^{j}\bullet_{x_{1}}\cdots\bullet_{x_{i-1}}
			\ot  \bullet_{x_{i+1}}\cdots \bullet_{x_{j}}\ot \bullet_{x_{j+1}}\cdots \bullet_{x_{m}}\\ &+\lambda^2\sum_{j=0}^{m}\sum_{i=0}^{j}\bullet_{x_{1}}\cdots\bullet_{x_{i}}
			\ot  \bullet_{x_{i+1}}\cdots \bullet_{x_{j}}\ot \bullet_{x_{j+1}}\cdots \bullet_{x_{m}}\\
			&\hspace{5cm} (\text{by exchanging the index of $i$ and $j$ })\\
			=&\mu^2\sum_{i=1}^{m-1}\sum_{j=i+1}^{m}\bullet_{x_{1}}\cdots\bullet_{x_{i-1}}
			\ot  \bullet_{x_{i+1}}\cdots \bullet_{x_{j-1}}\ot \bullet_{x_{j+1}}\cdots \bullet_{x_{m}}\\
			&-\lambda\mu\sum_{i=0}^{m-1}\sum_{j=i+1}^{m}\bullet_{x_{1}}\cdots\bullet_{x_{i}}
			\ot  \bullet_{x_{i+1}}\cdots \bullet_{x_{j-1}}\ot \bullet_{x_{j+1}}\cdots \bullet_{x_{m}}\\
			&-\lambda\mu\sum_{i=1}^{m}\sum_{j=i}^{m}\bullet_{x_{1}}\cdots\bullet_{x_{i-1}}
			\ot  \bullet_{x_{i+1}}\cdots \bullet_{x_{j}}\ot \bullet_{x_{j+1}}\cdots \bullet_{x_{m}}\\ &+\lambda^2\sum_{i=0}^{m}\sum_{j=i}^{m}\bullet_{x_{1}}\cdots\bullet_{x_{i}}
			\ot  \bullet_{x_{i+1}}\cdots \bullet_{x_{j}}\ot \bullet_{x_{j+1}}\cdots \bullet_{x_{m}}
			\\
			&\hspace{5cm} (\text{by exchanging the summands}).
		\end{align*}
	}
	Thus
	\begin{align*}
		(\id\ot\col)\col(\bullet_{x_{1}}\bullet_{x_{2}}\cdots\bullet_{x_{m}})=(\col \ot \id)\col(\bullet_{x_{1}}\bullet_{x_{2}}\cdots\bullet_{x_{m}}).
	\end{align*}
	
	Suppose that Eq.~(\mref{eq:coass}) holds for $\dep(F)\leq n$ for an $n\geq 0$ and consider the case of $\dep(F)=n+1$.
	We now apply the induction on breadth. Since $\dep(F)=n+1\geq 1$, we have $F\neq 1$ and $\bre(F)\geq 1$.
	If $\bre(F)=1$, then we may write $F=B_{\omega}^{+}(\lbar{F})$ for some $\lbar{F}\in \rfs$ and $\omega\in \Omega$. Hence
	\allowdisplaybreaks{
		\begin{align*}
			&\ (\id\otimes\col)\col(F)\\
			=&\ (\id\otimes \col)\col(B_{\omega}^{+}(\lbar{F}))\\
			=&\ (\id\otimes \col)\Big(-\lambda B_{\omega}^{+}(\lbar{F}) \otimes \etree+ \mu\lbar{F} \otimes \etree + (\id\otimes B_{\omega}^{+})\col(\lbar{F})\Big)\quad(\text{by Eq.~(\mref{eq:dbp})})\\
			=&\ \lambda^2(F\otimes \etree\otimes \etree)-\lambda\mu(\lbar{F}\otimes \etree\otimes \etree)+(\id\otimes \col B_{\omega}^{+})\col (\lbar{F})\\
			%
			%
			=&\  \lambda^2(F\otimes \etree\otimes \etree)-\lambda\mu(\lbar{F}\otimes \etree\otimes \etree)+\big(\id\otimes (-\lambda B_{\omega}^{+} \otimes \etree )+\mu(\id\ot \etree )+(\id \ot B_{\omega}^{+} )\col)\col (\lbar{F})\\ &\hspace{10cm} (\text{by Eq.~(\mref{eq:cdbp})})\\
			=&\   \lambda^2(F\otimes \etree\otimes \etree)-\lambda\mu(\lbar{F}\otimes \etree\otimes \etree)-\lambda\big((\id \ot B_{\omega}^{+})\col (\lbar{F})\ot \etree \big)+\mu(\id \ot \id \ot \etree)\col(\lbar{F})\\
			&+(\id \ot \id \ot B_{\omega}^{+})(\id\ot \col) \col(\lbar{F})\\
			=&\  \lambda^2(F\otimes \etree\otimes \etree)-\lambda\mu(\lbar{F}\otimes \etree\otimes \etree)-\lambda\big((\id \ot B_{\omega}^{+})\col (\lbar{F})\ot \etree )+\mu(\id \ot \id \ot \etree)\col(\lbar{F})\\
			&+(\id \ot \id \ot B_{\omega}^{+})(\col\ot \id) \col(\lbar{F})
			\quad (\text{by\ the\ induction\ hypothesis})\\
			=&\  -\lambda\big(-\lambda(F\ot \etree)+\mu(\lbar{F}\ot  \etree)+(\id \ot B_{\omega}^{+})\col(\lbar{F})\big)\ot \etree +\mu (\col(\lbar{F})\ot \etree)+(\col \ot B_{\omega}^{+}) \col(\lbar{F})\\
			=&\  -\lambda(\col B_{\omega}^{+}(\lbar{F}) \ot \etree)+\mu (\col(\lbar{F})\ot \etree) +(\col \ot B_{\omega}^{+}) \col(\lbar{F})\\
			=&\ (\col \ot \id  )\Big(-\lambda B_{\omega}^{+}(\lbar{F}) \otimes \etree+ \mu\lbar{F} \otimes \etree + (\id\otimes B_{\omega}^{+})\col(\lbar{F})\Big)\\
			=&\ (\col\otimes \id)\col(F) \quad(\text{by Eq.~(\mref{eq:dbp})}).
		\end{align*}
	}
	Assume that Eq.~(\mref{eq:coass}) holds for $\dep(F)=n+1$ and $\bre(F)\leq m$, in addition to $\dep(F)\leq n$ by the first induction hypothesis.
	Consider the case when $\dep(F)=n+1$ and $\bre(F)=m+1 \geq 2$.
	Then $F=F_{1}F_{2}$ for some $F_1, F_2\in \rfs$ with $0< \bre(F_1), \bre(F_2)< \bre(F)$.
	Using Sweedler notation, we write
	\begin{align*}
		\col(F_1)=\sum_{(F_1)}F_{1(1)}\otimes F_{1(2)} \text{\ and \ } \col(F_2)=\sum_{(F_2)}F_{2(1)}\otimes F_{2(2)}.
	\end{align*}
	Thus
	\allowdisplaybreaks{
		\begin{align*}
			&(\id\ot\col)\col(F_1F_2)\\
			=&\ (\id\ot\col)(F_1 \cdot \col(F_2)+\col(F_1) \cdot F_2+\lambda(F_1\ot F_2))\quad(\text{by Lemma~\mref{lem:colff}})\\
			=&\ (\id\ot\col)\left(\sum_{(F_2)}F_1F_{2(1)}\ot F_{2(2)}+\sum_{(F_1)}F_{1(1)}\ot F_{1(2)}F_2+\lambda(F_1\ot F_2) \right) \ \ (\text{by Eq.~(\ref{eq:dota})})\\
			=&\ \sum_{(F_2)}F_1F_{2(1)}\ot\col(F_{2(2)})+\sum_{(F_1)}F_{1(1)}\ot\col(F_{1(2)}F_2)+\lambda(F_1\ot \col (F_2))\\
			=&\ \sum_{(F_2)}F_1F_{2(1)}\ot F_{2(2)}\ot F_{2(3)}+\sum_{(F_1)}F_{1(1)}\ot\sum_{(F_2)}F_{1(2)}F_{2(1)}\ot F_{2(2)}\\
			&+\sum_{(F_1)}F_{1(1)}\ot F_{1(2)}\ot F_{1(3)}F_2 +\lambda\sum_{(F_1)}F_{1(1)}\ot F_{1(2)}\ot F_2+\lambda F_1\ot\col(F_2)\\
			&\hspace{8cm}(\text{by induction on breadth and Eq.~(\ref{eq:dota})})\\
			=&\ \left(\sum_{(F_2)}F_1F_{2(1)}\ot F_{2(2)}\ot F_{2(3)}+\sum_{(F_1)}F_{1(1)}\ot\sum_{(F_2)}F_{1(2)}F_{2(1)}\ot F_{2(2)}+\lambda F_1\ot\col(F_2)\right)\\
			&+\sum_{(F_1)}F_{1(1)}\ot F_{1(2)}\ot F_{1(3)}F_2 +\lambda\sum_{(F_1)}F_{1(1)}\ot F_{1(2)}\ot F_2\\
			=&\ \left(\sum_{(F_2)}F_1F_{2(1)}\ot F_{2(2)}\ot F_{2(3)}+\sum_{(F_2)}\sum_{(F_1)}F_{1(1)}\ot F_{1(2)}F_{2(1)}\ot F_{2(2)}+\lambda F_1\ot\col(F_2)\right)\\
			&+\sum_{(F_1)}F_{1(1)}\ot F_{1(2)}\ot F_{1(3)}F_2 +\lambda \col(F_1)\ot F_2\\
			=&\ \left(\sum_{(F_2)}F_1F_{2(1)}\ot F_{2(2)}\ot F_{2(3)}+\sum_{(F_2)}\sum_{(F_1)}F_{1(1)}\ot F_{1(2)}F_{2(1)}\ot F_{2(2)}
			+\lambda\sum_{(F_2)} F_1\ot F_{2(1)}\ot F_{2(2)}\right)\\
			&+\sum_{(F_1)}F_{1(1)}\ot F_{1(2)}\ot F_{1(3)}F_2 +\lambda \col(F_1)\ot F_2\\
			=&\ \sum_{(F_2)}\col(F_1F_{2(1)})\ot F_{2(2)}+\sum_{(F_1)}\col(F_{1(1)})\ot F_{1(2)}F_2 +\lambda \col(F_1)\ot F_2\\
			&\hspace{8cm}(\text{by induction on breadth})\\
			=&\ (\col\ot\id)\left(\sum_{(F_{2})}F_1F_{2(1)}\ot F_{2(2)}+\sum_{(F_1)}F_{1(1)}\ot F_{1(2)}F_2+\lambda F_1\ot F_2\right)\\
			=&\ (\col\ot\id)(F_1 \cdot \col(F_2)+\col(F_1) \cdot F_2+\lambda(F_1\ot F_2))  \quad(\text{by Eq.~(\mref{eq:dota})})\\
			=&\ (\col\ot\id)\col(F_1F_2) \quad(\text{by Lemma~\mref{lem:colff}}).
		\end{align*}
	}
	This completes the induction on the breadth and hence the induction on the depth.
\end{proof}

Now we arrive at our second main result in this subsection.
\begin{theorem}
	The quadruple $(\hrts,\, \conc, \etree, \,\col)$ is an  $\epsilon$-unitary bialgebra of weight $\lambda$.
	\mlabel{thm:rt2}
\end{theorem}

\begin{proof}
	Note that the triple $(\hrts,\, \conc, \etree)$ is a unitary algebra. Then the result follows from Lemma~\mref{lem:colff} and Theorem~\mref{thm:rt1}.
\end{proof}

\begin{remark}
	\begin{enumerate}
		\item
		When we consider an $\epsilon$-unitary bialgebra of weight $\lambda$ with a counit $\varepsilon$, the condition $\lambda$ is invertible in $\bfk$ is needed. In fact, by $\Delta(1)=-\lambda(1\ot1)$, we have $\varepsilon(1)=-\lambda^{-1}$.
		\item We shall prove the converse result and give a formula for the counit in Corollary \mref{cor:counit}.
	\end{enumerate}
\end{remark}

\subsection{A combinatorial description of the coproduct $\col$}
In this subsection, we give a combinatorial description of the coproduct $\col$. Let us first recall several order relations on the set of the vertices of a decorated planar rooted forest. See~\mcite{Foi02, Foi09} for more details.

\begin{defn}\label{def:order}\cite[Sec.~2.1]{Foi09}
	Let $F = T_1 \cdots T_n\in \rfs \setminus\{\etree\}$ with $T_1, \cdots, T_n\in \rts$ and $n\geq 1$,
	and let $u,v$ be two vertices of $F$. Then
	\begin{enumerate}
		\item $v\leq_{h}u$ ({\bf being higher}) if there exists a (directed) path from $v$ to $u$ in $F$ and the edges of $F$ being oriented from roots to leaves;
		\item $v\leq_{l}u$ ({\bf being more on the left}) if $u$ and $v$ are not comparable for $\leq_{h}$ and one of the following assertions is satisfied:
		\begin{enumerate}
			\item  $u$ is a vertex of $T_i$ and $v$ is a vertex of $T_j$, with $1\leq i< j \leq n$.
			\item $u$ and $v$ are vertices of the same $T_i$, and $v\leq_{l}u$ in the forest obtained from $T_i$ by deleting its root;
		\end{enumerate}
		\item $v\lhl u$ ({\bf being higher or more on the left}) if $v\leq_{h}u$ or $v\leq_{l}u$.
	\end{enumerate}
	
\end{defn}

\begin{exam}
	Let $F=\tdtroisun{$\alpha$}{$\gamma$}{$\beta$}$. The following arrays give the order relations $\leq_{h}$ and $\leq_{l}$ for the vertices of $F$, respectively.
	The symbol $\times$ means that the vertices are not comparable for the order.
	\begin{table}[H]
		\begin{tabular}{|c|c|c|c|}
			\hline
			$\leq_h$ & $\bullet_\alpha$& $\bullet_\beta$ & $\bullet_\gamma$ \\
			\hline
			$\bullet_\alpha$ & $=$ & $\leq_h$ & $\leq_h$\\
			\hline
			$\bullet_\beta$ & $\geq_{h}$ & = & $\times$  \\
			\hline
			$\bullet_\gamma$ & $\geq_{h}$ & $\times$  & $=$  \\
			\hline
		\end{tabular}\quad \quad \quad \quad
		\begin{tabular}{|c|c|c|c|}
			\hline
			$\leq_l$ & $\bullet_\alpha$& $\bullet_\beta$ & $\bullet_\gamma$ \\
			\hline
			$\bullet_\alpha$ & $=$ & $\times$ & $\times$\\
			\hline
			$\bullet_\beta$ & $\times$ & = & $\geq_{l}$  \\
			\hline
			$\bullet_\gamma$ & $\times$ & $\leq_l$  & $=$  \\
			\hline
		\end{tabular}
	\end{table}
	\noindent Then we have $\bullet_\alpha\lhl \bullet_\gamma\lhl \bullet_\beta$.
\end{exam}

\begin{remark}\cite[Sec.~2.1]{Foi09}
	Let $F\in \rfs$. Then
	\begin{enumerate}
		\item the order ``being higher" $\leq_{h}$ is a partial order, whose Hasse graph is the decorated planar rooted forest $F$;
		\item the order ``being more on the left'' is a  partial order for $F$;
		\item the order ``being higher or more on the left'' is a total order on the vertices of $F$.
	\end{enumerate}
\end{remark}

Using these order relations on the set of vertices of $F$, the first author~\mcite{Foi09} introduced the following concepts.

\begin{defn}\mlabel{def:fid}
	Let $F$ be a decorated planar rooted forest and $V(F)$ the set of vertices of $F$.
	A set $I\subseteq V(F)$ is called a {\bf forest biideal} of $F$, denoted by $I \fid F$, if
	$$ \forall u\in I, \, \forall v\in V(F), \,\, u\lhl v \Rightarrow v\in I.$$
	A forest biideal $I$ is called a {\bf principal forest biideal} if it is generated by a single element of $V(F)$ under the total order $\lhl $, that is, if $u$ is the generator, then any other element $v\in I$ satisfies $u\lhl v$.
	We denote such principal forest biideal by $\langle u\rangle_{\lhl }$.
\end{defn}

The forest biideals of $F$ can be characterized as:

\begin{lemma}
	Let $F\in \rfs$. The number of vertices of $F$ is denoted by $n_F$. Let $u_{n_F}\lhl \cdots \lhl u_1 $ be its vertices.
	Then $F$ has exactly $n_F+1$ forest biideals:
	\begin{align*}
		I_k=\{u_1, \ldots, u_k\} \text { for } k = 0, \ldots, n_F,
	\end{align*}
	with the convention that $I_0 = \emptyset$.
	\mlabel{lem:comid1}
\end{lemma}

\begin{proof}
	It is the same as the proof of~\cite[Lemma~13]{Foi09}.
\end{proof}

As a direct consequence of Lemma~\ref{lem:comid1}, we obtain the following.

\begin{coro}
	Let $F\in \rfs$ and let $u_{n_F}\lhl \cdots \lhl u_1 $ be its vertices.  Then every nonempty forest biideal $I \fid F$ is principal:
	\begin{align*}
		I=I_k=\langle u_k\rangle_{\lhl } \text { for some } k\in \{1, \ldots, n_F\}.
	\end{align*}
\end{coro}

\begin{proof}
	By Lemma~\ref{lem:comid1}, $I=I_k$ for some $k\in \{1, \ldots, n_F\}$. Since $u_1\lhl \cdots \lhl u_{n_F} $, it follows that $I_k = \langle u_k\rangle_{\lhl }.$
\end{proof}

So any decorated planar rooted forest $F$ has exactly $n_F+1$ (principal) biideals, which form
a chain:
\begin{align*}
	\begin{array}{ccccccccccc}
		\emptyset &\subseteq &\langle u_1\rangle_{\lhl }&\subseteq &\cdots &\subseteq &\langle u_k\rangle_{\lhl }&\subseteq &\cdots &\subseteq& \langle u_{n_F}\rangle_{\lhl }\\
		\shortparallel&&\shortparallel&&&&\shortparallel&&&&\shortparallel\\
		I_0 &\subseteq &I_1&\subseteq &\cdots &\subseteq &I_k&\subseteq &\cdots &\subseteq& I_{n_F}
	\end{array}
\end{align*}
Note that for any $k$, $|I_k|=k$.

For any set $I\subseteq V(F)$, let $F_{|I}$ be the derived decorated planar rooted subforest whose vertices are $I$ and edges are the edges of $F$ between elements of $I$. In particular,
\begin{align}
	F_{|\emptyset} := \etree \,\text{ and }\, \etree_{|I}:=\etree.
	\mlabel{eq:efoz1}
\end{align}

\begin{lemma}
	Let $F=B_{\omega}^{+}(\lbar F)$ for some $\omega\in \Omega$ and $\lbar{F}\in \rfs$ and $I \subseteq V(\lbar F)$. Then
	$I\fid \lbar F$ if and only if $I\fid F$.
	\mlabel{lem:bfid1}
\end{lemma}

\begin{proof}
	It follows from Lemma~\mref{lem:comid1} and the fact that the root $\bullet_\omega$ is the minimum of $V(F)$ with respect to $\lhl$.
\end{proof}

Let $F\in \rfs$.
For any set $I\subseteq V(F)$, let $F_{|I}$ be the derived decorated planar rooted subforest in $\rfs$ whose vertices are $I$
and whose partial orders $\leq_h$ and $\leq_l$ are the restriction of the partial orders $\leq_h$ and $\leq_l$ of $F$. In particular,
\begin{align*}
	F_{|\emptyset}& = \etree,&F_{\mid V(F)}&=F.
	\mlabel{eq:efoz}
\end{align*}
\begin{exam} If $F=\tdquatredeux{$\delta$}{$\gamma$}{$\beta$}{$\alpha$}$, indexing the vertices of $F$ according to $\lhl$, as $\tdquatredeux{4}{3}{2}{1}$,
	we obtain
	\begin{align*}
		F_{\mid \emptyset}&=1,&F_{\mid \{1\}}&=\tdun{$\alpha$},&F_{\mid \{2\}}&=\tdun{$\beta$},&F_{\mid \{3\}}&=\tdun{$\gamma$},\\
		F_{\mid \{4\}}&=\tdun{$\delta$},&
		F_{\mid \{1,2\}}&=\tddeux{$\beta$}{$\alpha$},&
		F_{\mid \{1,3\}}&=\tdun{$\alpha$}\tdun{$\gamma$},&
		F_{\mid \{1,4\}}&=\tddeux{$\delta$}{$\alpha$},\\
		F_{\mid \{2,3\}}&=\tdun{$\beta$}\tdun{$\gamma$},&
		F_{\mid \{2,4\}}&=\tddeux{$\delta$}{$\beta$},&
		F_{\mid \{3,4\}}&=\tddeux{$\delta$}{$\gamma$},&
		F_{\mid \{1,2,3\}}&=\tddeux{$\beta$}{$\alpha$}\tdun{$\gamma$},\\
		F_{\mid \{1,2,4\}}&=\tdtroisdeux{$\delta$}{$\beta$}{$\alpha$},&
		F_{\mid \{1,3,4\}}&=\tdtroisun{$\delta$}{$\gamma$}{$\alpha$},&
		F_{\mid \{2,3,4\}}&=\tdtroisun{$\delta$}{$\gamma$}{$\beta$},&
		F_{\mid \{1,2,3,4\}}&=\tdquatredeux{$\alpha$}{$\beta$}{$\gamma$}{$\delta$}.
\end{align*}\end{exam}

\begin{remark}\label{rk:cut}
	Let $F=B_{\omega}^{+}(\lbar F)$ for some $\omega\in \Omega$ and $\lbar{F}\in \rfs$ and $I,H\subseteq V(\lbar F)$. Then
	\begin{align*}
		B_{\omega}^{+}({\lbar F}_{|I})&= B_{\omega}^{+}(\lbar F)_{| (I\sqcup \{\bullet_\omega\})}
	\end{align*}
	illustrated graphically as below:
	\begin{align*}
		\lbar {F}=\begin{tikzpicture}[baseline,scale=0.8]
			\draw (0,0) circle (0.5);
			\node at (0,0) {$\overline F$};
		\end{tikzpicture},\ \ \ \
		{\lbar F}_{|I}=\begin{tikzpicture}[baseline,scale=0.8]
			\draw (3,-0.5) arc (270:90:0.5)--+(0,-1);
			\node at (2.75,0) {$I$};
		\end{tikzpicture},\ \ \ \
		B_{\omega}^{+}({\lbar F}_{|I})=\begin{tikzpicture}[baseline, scale=0.8]
			\draw (6.5,-0.5) arc (270:90:0.5)--+(0,-2);
			\node at (6.25,0) {$I$};
			\fill (6.5,-1.5) circle (2pt) node[right]{$\omega$};
		\end{tikzpicture},\ \
		B_{\omega}^{+}(\lbar F)=\begin{tikzpicture}[baseline,scale=0.8]
			\draw (10,0) circle (0.5);
			\draw (10,0.5)--+(0,-2);
			\node at (9.75,0) {$I$};
			\fill (9.75,-1.5) circle (2pt) node[right]{$\omega$};
		\end{tikzpicture},\ \
		B_{\omega}^{+}(\lbar F)_{| (I\sqcup \{\bullet_\omega\})}=\begin{tikzpicture}[baseline,scale=0.8]
			\draw (15,-0.5) arc (270:90:0.5)--+(0,-2);
			\node at (14.75,0) {$I$};
			\fill (15,-1.5) circle (2pt) node[right]{$\omega$};
		\end{tikzpicture}.
	\end{align*}
\end{remark}


\begin{proof}
	It follows from Lemma~\mref{lem:comid1} and the fact that the root $\bullet_\omega$ is minimum in $V(F)$ with respect to $\lhl$.
\end{proof}

\begin{theorem}
	Let $F\in \hrts$. We order its vertices according to $\lhl$:
	\[u_{n_F}\lhl \cdots \lhl u_1.\]
	For any $k\in \{0,\ldots,n_F\}$, we put
	$I_k=\{u_1,\ldots,u_k\}$, and $J_k=\{u_{k+1},\ldots,n_F\}$.
	Note that the $I_k$'s are the forest biideals of Lemma \mref{lem:comid1}, and the $J_k$'s are their complements.
	Then
	\begin{align}
		\col(F)&=-\lambda \sum_{k=0}^{n_F} F_{\mid I_k}\otimes F_{\mid J_k}+\mu \sum_{k=1}^{n_F} F_{\mid I_{k-1}}\otimes F_{\mid J_k}.
		\mlabel{eq:comb}
	\end{align}
	\mlabel{thm:comb}
\end{theorem}
\begin{proof}
	We prove this equality for basis elements $F\in \rfs$ by induction on $n_F= |V(F)|$.
	If $n=0$, then $F=\etree$ and
	\[\col(F)=-\lambda 1\otimes 1=-\lambda F_{\mid I_0}\otimes F_{\mid J_0}.\]
	Assume that the result holds for any $G\in \rfs$ such that $n_G<n_F$.
	If $\bre(F)>1$, we can write $F=F_1F_2$, with $n_{F_1}, n_{F_2}<n_F$. We obtain
	\begin{align*}
		\col(F)=&\ \col(F_1)F_2+F_1\col(F_2)+\lambda F_1\otimes F_2\\
		=&\ -\lambda \sum_{k=0}^{n_{F_1}}(F_1)_{\mid I_k}\otimes (F_1)_{\mid J_k}F_2
		-\lambda \sum_{k=0}^{n_{F_2}}F_1(F_2)_{\mid I_k}\otimes (F_2)_{\mid J_k}
		+\lambda F_1\otimes F_2\\
		&\ +\mu \sum_{k=1}^{n_{F_1}}(F_1)_{\mid I_{k-1}}\otimes (F_1)_{\mid J_k}F_2
		+\mu \sum_{k=1}^{n_{F_2}}F_1(F_2)_{\mid I_{k-1}}\otimes (F_2)_{\mid J_k}\\
		=&\ -\lambda \sum_{k=0}^{n_{F_1}}F_{\mid I_k}\otimes F_{\mid J_k}
		-\lambda \sum_{k=0}^{n_{F_2}}F_{\mid I_{k+n_{F_1}}}\otimes F_{\mid J_{k+n_{F_1}}}
		+\lambda F_1\otimes F_2\\
		&\ +\mu \sum_{k=1}^{n_{F_1}}F_{\mid I_{k-1}}\otimes F_{\mid J_k}
		+\mu \sum_{k=1}^{n_{F_2}}F_{\mid I_{k+n_{F_1}-1}}\otimes F_{\mid J_{k+n_{F_1}}}\\
		=&\ -\lambda \sum_{k=0}^{n_{F_1}}F_{\mid I_k}\otimes F_{\mid J_k}
		-\lambda \sum_{k=n_{F_1}}^{n_{F_1}+n_{F_2}}F_{\mid I_k}\otimes F_{\mid J_k}
		+\lambda F_1\otimes F_2\\
		&\ +\mu \sum_{k=1}^{n_{F_1}}F_{\mid I_{k-1}}\otimes F_{\mid J_k}
		+\mu \sum_{k=n_{F_1}}^{n_{F_1}+n_{F_2}}F_{\mid I_{k-1}}\otimes F_{\mid J_k}\\
		=&\ -\lambda \sum_{k=0}^{n_{F_1}+n_{F_2}}F_{\mid I_k}\otimes F_{\mid J_k}+\mu \sum_{k=1}^{n_{F_1}+n_{F_2}}F_{\mid I_{k-1}}\otimes F_{\mid J_k}\\
		=&\ -\lambda \sum_{k=0}^{n_F}F_{\mid I_k}\otimes F_{\mid J_k}+\mu \sum_{k=1}^{n_F}F_{\mid I_{k-1}}\otimes F_{\mid J_k}.
	\end{align*}
	
	If $\bre(F)=1$, we have two cases to consider.
	
	\noindent{\bf Case 1.} $F=\bullet_{x}$ for some $x\in X$.
	Then
	\begin{align*}
		\col(F)&=-\lambda(F\otimes 1+1\otimes F)+\mu 1\otimes 1\\
		&=-\lambda (F_{\mid I_1}\otimes F_{\mid J_1}+F_{\mid I_0}\otimes F_{\mid J_0})+\mu F_{\mid I_0}\otimes F_{\mid J_1}.
	\end{align*}		
	
	\newcommand{\barF}{\lbar{F}}
	
	\noindent{\bf Case 2.} $F = B_{\omega}^{+}(\lbar{F})$ for some $\lbar{F} \in \rfs$ and $\omega\in \Omega$. Applying the induction hypothesis on $\barF$, we obtain
	\begin{align*}
		\col(F)&=-\lambda F\otimes 1+\mu \barF\otimes 1+(\id \otimes B_\omega^+)\left(-\lambda
		\sum_{k=0}^{n_{\barF}} \barF_{\mid I_k} \otimes \barF_{\mid J_k}
		+\mu \sum_{k=1}^{n_{\barF}} \barF_{\mid I_{k-1}} \otimes \barF_{\mid J_k}\right)\\
		&=-\lambda F_{\mid I_{n_F}}\otimes F_{\mid J_{n_F}}+\mu F_{\mid I_{n_F}-1}\otimes F_{\mid J_{n_F}}
		-\lambda \sum_{k=0}^{n_{\barF}} \barF_{\mid I_k} \otimes B_\omega^+(\barF_{\mid J_k})
		+\mu \sum_{k=1}^{n_{\barF}} \barF_{\mid I_{k-1}} \otimes B_\omega^+( \barF_{\mid J_k})\\
		&=-\lambda F_{\mid I_{n_F}}\otimes F_{\mid J_{n_F}}+\mu F_{\mid I_{n_F}-1}\otimes F_{\mid J_{n_F}}
		-\lambda \sum_{k=0}^{n_F-1} F_{\mid I_k} \otimes F_{\mid J_k}
		+\mu \sum_{k=1}^{n_F-1} F_{\mid I_{k-1}} \otimes F_{\mid J_k}\\
		&=-\lambda \sum_{k=0}^{n_F} F_{\mid I_k} \otimes F_{\mid J_k}
		+\mu \sum_{k=1}^{n_F} F_{\mid I_{k-1}} \otimes F_{\mid J_k}.
	\end{align*}
	Hence Eq.~(\mref{eq:comb}) holds. This completes the induction on $n_F$.
\end{proof}
\begin{exam}
	\begin{enumerate}
		\item \label{exam:a}
		Consider $F=\tdun{$\beta$}\tddeux{$\alpha$}{$x$}$. Then $\bullet_\alpha \lhl \bullet_x\lhl \bullet_\beta$.
		By  Lemma~\mref{lem:comid1}, $F$ has four forests biideals $\emptyset$, $\{\bullet_\beta\}$,$\{\bullet_\beta, \bullet_x\}$,$\{\bullet_\beta, \bullet_x,\bullet_\alpha\}$. So by Theorem~\mref{thm:comb},
		\begin{align*}
			\col(\tdun{$\beta$}\tddeux{$\alpha$}{$x$})=&\mu(F_{|\emptyset}\ot F_{| \{\bullet_x, \bullet_\alpha\}}
			+F_{|\{\bullet_\beta\}}\ot F_{| \{\bullet_\alpha\}}
			+F_{|\{\bullet_\beta, \bullet_x\}}\ot F_{|\emptyset})\\
			&-\lambda(F_{|\emptyset}\ot F_{|\{\bullet_\beta, \bullet_x, \bullet_\alpha\}}+F_{|\{\bullet_\beta\}}\ot F_{|\{\bullet_x, \bullet_\alpha\}}+F_{|\{\bullet_\beta, \bullet_x\}}\ot F_{|\{ \bullet_\alpha\}}+F_{|\{\bullet_\beta, \bullet_x, \bullet_\alpha\}}\ot F_{|\emptyset})\\
			=&\mu(\etree\ot \tddeux{$\alpha$}{$x$}+\tdun{$\beta$}\ot \tdun{$\alpha$}+ \tdun{$\beta$}\tdun{$x$}\ot \etree)\\
			&-\lambda(\etree \ot \tdun{$\beta$}\tddeux{$\alpha$}{$x$}+\tdun{$\beta$}\ot \tddeux{$\alpha$}{$x$}+\tdun{$\beta$}\tdun{$x$}\ot \tdun{$\alpha$}+ \tdun{$\beta$}\tddeux{$\alpha$}{$x$}\ot \etree).
		\end{align*}
		
		\item \label{exam:b}
		Let $F=\tdquatretrois{$\omega$}{$\beta$}{$x$}{$\alpha$}$. Then $\bullet_\omega \lhl \bullet_\beta\lhl \bullet_{x}\lhl \bullet_{\alpha}$ and so $F$ has five forests biideals $\emptyset$, $\{\bullet_\alpha\}$, $\{\bullet_\alpha, \bullet_x\}$,$\{\bullet_\alpha, \bullet_x, \bullet_\beta\}$,$\{\bullet_\alpha, \bullet_x, \bullet_\beta,\bullet_\omega\}$ by Lemma~\mref{lem:comid1}. It follows from Theorem~\mref{thm:comb} that
		\begin{align*}
			\col( \tdquatretrois{$\omega$}{$\beta$}{$x$}{$\alpha$})
			=&\mu (F_{|\emptyset}\ot F_{|\{\bullet_x,\bullet_\beta,\bullet_\omega\}}+F_{|\{\bullet_\alpha\}}\ot F_{| \{\bullet_\beta, \bullet_\omega\}}+F_{|\{\bullet_\alpha, \bullet_x\}}\ot F_{|\{\bullet_\omega\}}
			+F_{|\{\bullet_\alpha, \bullet_x,\bullet_\beta \}}\ot F_{|\emptyset})\\
			&-\lambda(F_{|\emptyset}\ot F_{|\{\bullet_x,\bullet_\beta,\bullet_\alpha,\bullet_\omega\}}+F_{|\{\bullet_\alpha\}}\ot F_{| \{\bullet_x,\bullet_\beta, \bullet_\omega\}}+F_{|\{\bullet_\alpha, \bullet_x\}}\ot F_{|\{\bullet_\omega,\bullet_\beta\}}\\
			&+F_{|\{\bullet_\alpha, \bullet_x,\bullet_\beta \}}\ot F_{|\{\bullet_\omega\}}+F_{|\{\bullet_\alpha, \bullet_x,\bullet_\beta,\bullet_\omega\} \}}\ot F_{|\emptyset})\\
			=&\mu(\etree\ot \tdtroisdeux{$\omega$}{$\beta$}{$x$}+\tdun{$\alpha$}\ot \tddeux{$\omega$}{$\beta$}+\tdun{$\alpha$}\tdun{$x$}\ot \tdun{$\omega$}+ \tdun{$\alpha$}\tddeux{$\beta$}{$x$}\ot \etree)\\
			&-\lambda(\tdquatretrois{$\omega$}{$\beta$}{$x$}{$\alpha$}\ot \etree +\tdun{$\alpha$}\tddeux{$\beta$}{$x$}\ot \tdun{$\omega$}+\tdun{$\alpha$}\tdun{$x$}\ot\tddeux{$\omega$}{$\beta$}+\tdun{$\alpha$}\ot\tdtroisdeux{$\omega$}{$\beta$}{$x$}+\etree \otimes \tdquatretrois{$\omega$}{$\beta$}{$x$}{$\alpha$}).
		\end{align*}
	\end{enumerate}
	Observe that the results in~(\mref{exam:a}) and (\mref{exam:b}) are consistent with the ones in Example~\mref{exam:cop}.
\end{exam}

\begin{coro}\mlabel{cor:counit}
	Let us assume that $\lambda$ is an invertible element of $\bfk$. Then $\col$ has a counit, defined by
	\begin{align}
		\mlabel{eqcounit}
		\varepsilon_{\lambda,\mu}(F)&=-\frac{\mu^{n_F}}{\lambda^{n_F+1}}\mbox{ for }F\in \rfs.
	\end{align}
	It is an algebra morphism if, and only if, $\lambda=-1$.
\end{coro}

\begin{proof}
	Let us prove that the map defined by (\mref{eqcounit}) is indeed a counit. Let $F\in \rfs$. From Theorem \ref{thm:comb},
	\begin{align*}
		(\varepsilon_{\lambda,\mu}\otimes \id)\circ \col(F)&=\lambda \sum_{k=0}^{n_F}
		\frac{\mu^k}{\lambda^{k+1}}F_{\mid J_k}-\mu \sum_{k=1}^{n_F} \frac{\mu^{k-1}}{\lambda^k}F_{\mid J_k}\\
		&=\sum_{k=0}^{n_F}
		\frac{\mu^k}{\lambda^k}F_{\mid J_k}-\sum_{k=1}^{n_F} \frac{\mu^k}{\lambda^k}F_{\mid J_k}\\
		&=\frac{\mu^0}{\lambda^0}F_{\mid J_0}\\
		&=F,
	\end{align*}
	as $J_0=V(F)$. So $\varepsilon_{\lambda,\mu}$ is a left counit for $\col$. One can show it is a right counit
	in a very similar way.
	
	If $\lambda=-1$, then $\varepsilon_{-1,\mu}(F)=(-\mu)^{n_F}$,
	which obviously defines an algebra morphism.
	If $\varepsilon_{\lambda,\mu}$ is an algebra morphism, then
	\begin{align*}
		-\frac{1}{\lambda}=\varepsilon_{\lambda,\mu}(1)
		=\varepsilon_{\lambda,\mu}(1)\varepsilon_{\lambda,\mu}(1)=
		\frac{1}{\lambda^2}.
	\end{align*}
	So $\lambda=-1$.
\end{proof}

\subsection{Dual products}

We assume in this paragraph that $\Omega$ and $X$ are finite. We identify $\hrts$ with its graded dual, through the bilinear symmetric form given by
\begin{align*}
	&\forall F,G\in \rfs,&\langle F,G\rangle&=\delta_{F,G}.
\end{align*}
If $F\in \rfs$, $\Delta_{\lambda,\mu}(F)$ is a span of tensors of two forests $F'\otimes F''$, with $n_{F'}+n_{F''}\in \{n_F,n_F-1\}$. Therefore, there exists a unique product $\star_{\lambda,\mu}$ on $\hrts$, such that if $x,y,z\in \hrts$,
\[\langle\Delta_{\lambda,\mu}(x),y\otimes z\rangle=\langle x,y\star_{\lambda,\mu}z\rangle.\]
The coassociativity of $\Delta_{\lambda,\mu}$ implies the associativity of $\star_{\lambda,\mu}$.
Moreover, if $F,G\in \rfs$, then $F\star_{\lambda,\mu} G$
is a span of forests with $n_F+n_G$ or $n_F+n_G+1$ vertices.
We put $\star=\star_{-1,0}$ and we first describe $\star$. By definition of $\Delta_{-1,0}$,
if $F,G\in \rfs$, then $F\star G$ is the sum of all forests $H\in \rfs$, such that
\[\begin{cases}
	n_H=n_F+n_G,\\
	H_{\mid I_{n_F}}=F,\\
	H_{\mid J_{n_F}}=G.
\end{cases}\]
Let us describe this set of forests.

\begin{prop}
	Let $F=T_1\ldots T_m\in \rfs$ be a forest formed by $m$ trees and $G\in \rfs$.
	We denote by $v_1,\ldots,v_n$ the vertices of $G$ on the path from the root at most on the left of $G$
	the leaf at most on the left of $G$, except the leaf itself if it is decorated by an element of $X$.
	Let $\sigma:\{1,\ldots,m\}\longrightarrow\{0,\ldots,n\}$ be an increasing map: $\sigma(1)\leq \ldots \leq \sigma(m)$.
	We denote by $F\curvearrowright_\sigma G$ the forest obtained by grafting $T_m,\ldots,T_1$ (in this order) as this:
	\begin{itemize}
		\item If $\sigma(i)=j\geq 1$, $T_i$ is grafted on $v_j$ the most possible on the left.
		\item If $\sigma(i)=0$, $T_i$ is put on the left of the forest obtained so far.
	\end{itemize}
	Then
	\[F\star G=\sum_{\substack {\sigma:\{1,\ldots,m\}\longrightarrow \{0,\ldots,n\},\\\mbox{\scriptsize increasing}}} F\curvearrowright_\sigma G.\]
\end{prop}

\begin{proof}
	We put $k=|V(F)|$ and $l=|V(G)|$.
	Let $\sigma:\{1,\ldots,m\}\longrightarrow \{0,\ldots,n\}$, increasing. We put $H=F\curvearrowright_\sigma G$. As we do not graft on a leaf decorated by an element of $X$ in the process, this belongs to $\rfs$.
	By construction, $V(H)=V(F)\sqcup V(G)$ and:
	\begin{enumerate}
		\item $H_{\mid  V(G)}=G$.
		\item For any $i\in \{1,\ldots,m\}$, $H_{\mid  V(T_i)}=T_i$.
		\item If $x\in T_i$ and $y\in T_j$, with $i<j$, then $x\geq_l y$.
		\item If $x\in V(F)$ and $y\in V(G)$, then $x\geq_l y$ or $x\geq_h y$.
	\end{enumerate}
	As a consequence of the points 2 and 3, $H_{\mid  V(F)}=F$. The last point then implies that
	$H_{\mid I_{n_F}}=F$ and $H_{\mid J_{n_F}}=G$.
	
	Let $H\in \rfs$, with $n_F+n_G$ vertices, such that $H_{\mid I_{n_F}}=F$ and $H_{\mid J_{n_F}}=G$.
	$H$ contains a copy of $T_i$ for any $i$ and a copy of $G$.
	If the root of $T_i$, we put $\sigma(i)=0$. Otherwise, this root is grafted on a vertex $v$ of $G$.
	As $T_i$ is a tree of $H_{\mid I_k}$, necessarily $v$ is one of the $v_j$'s: we denote by $\sigma(i)$ the unique index
	such that the root of $T_i$ is grafted on $v_{\sigma(j)}$. Then, as $T_j \leq_l T_i$ if $i<j$, necessarily $\sigma$ is increasing
	and $H=F\curvearrowright_\sigma G$.
\end{proof}

\begin{remark}
	The product $F\star G$ contains $\displaystyle\binom{m+n}{n}$ terms, all different. Moreover, even if $\Omega$ or $X$ is not finite, this product $\star$ makes sense. It will not be the case of $\star_{\lambda,\mu}$, as it will be seen in Theorem \mref{theoproduitdual}.
\end{remark}

\begin{exam}
	Let $\alpha,\beta,\gamma,\delta \in \Omega$ and let $x\in X$.
	\begin{align*}
		\tdun{$\alpha$}\tdun{$\beta$}\star \tddeux{$\gamma$}{$\delta$}&=\tdun{$\alpha$}\tdun{$\beta$}\tddeux{$\gamma$}{$\delta$}
		+\tdun{$\alpha$}\tdtroisun{$\gamma$}{$\delta$}{$\beta$}+\tdun{$\alpha$}\tdtroisdeux{$\gamma$}{$\delta$}{$\beta$}+
		\tdquatreun{$\gamma$}{$\delta$}{$\beta$}{$\alpha$}+\tdquatretrois{$\gamma$}{$\delta$}{$\beta$}{$\alpha$}
		+\tdquatrequatre{$\gamma$}{$\delta$}{$\beta$}{$\alpha$},\\
		\tdun{$\alpha$}\tdun{$\beta$}\star \tddeux{$\gamma$}{$x$}&=\tdun{$\alpha$}\tdun{$\beta$}\tddeux{$\gamma$}{$x$}
		+\tdun{$\alpha$}\tdtroisun{$\gamma$}{$x$}{$\beta$}+
		\tdquatreun{$\gamma$}{$x$}{$\beta$}{$\alpha$}.
	\end{align*}	
\end{exam}

\begin{theorem}\mlabel{theoproduitdual}
	For any $\lambda,\mu \in \bfk$, for any $x,y\in \hrts$,
	\[x\star_{\lambda,\mu}y=-\lambda x\star y+\mu\sum_{\omega \in X\sqcup \Omega} x\star \tdun{$\omega$} \star y.\]
\end{theorem}

\begin{proof}
	Firstly, for any $G\in \rfs$,
	\begin{align*}
		\sum_{\omega\in X\sqcup \Omega}\langle\tdun{$\omega$},G\rangle&=\begin{cases}
			1\mbox{ if $\exists ~\omega \in X\sqcup \Omega$, $G=\tdun{$\omega$}$},\\
			0\mbox{ otherwise.}
		\end{cases}\\
		&=\begin{cases}
			1\mbox{ if }n_G=1,\\
			0\mbox{ otherwise}.
		\end{cases}
	\end{align*}
	Let $x,y\in \hrts$, and $F\in \rfs$. The vertices of $F$ are totally ordered by $\lhl$:
	\[v_n\lhl \cdots \lhl v_1.\]
	Then we have
	\begin{align*}
		&\sum_{\omega \in X\sqcup \Omega} \langle x\otimes \tdun{$\omega$}\otimes y,(\Delta_{-1,0} \otimes \id)\circ \Delta_{-1,0}(F)\rangle\\
		&= \sum_{\omega \in X\sqcup \Omega}\sum_{0\leq k\leq l\leq n_F} \langle x\otimes \tdun{$\omega$}\otimes y,
		F_{\mid \{v_1,\ldots,v_k\}}\otimes F_{\mid \{v_{k+1},\ldots,v_l\}}\otimes F_{\mid \{v_{l+1},\ldots,v_{n_F}\}}\rangle\\
		&=\sum_{0\leq k <n_F} \langle x\otimes y,
		F_{\mid \{v_1,\ldots,v_k\}}\otimes F_{\mid \{v_{k+2},\ldots,v_{n_F}\}}\rangle+0\\
		&=\sum_{0\leq k <n_F} \langle x\otimes y,F_{\mid I_k}\otimes F_{\mid J_{k+1}}\rangle\\
		&=\sum_{1\leq k\leq n_F} \langle x\otimes y,F_{\mid I_{k-1}}\otimes F_{\mid J_k}\rangle.
	\end{align*}
	As a consequence, by Theorem \ref{thm:comb},
	\begin{align*}
		\langle x\otimes y,\col(F)\rangle&=-\lambda \langle x\otimes y,\Delta_{-1,0}(F)\rangle
		+\mu \sum_{1\leq k\leq n_F} \langle x\otimes y,F_{\mid I_{k-1}}\otimes F_{\mid J_k}\rangle\\
		&=-\lambda \langle x\otimes y,\Delta_{-1,0}(F)\rangle+\mu \sum_{\omega \in X\sqcup \Omega} \langle x\otimes \tdun{$\omega$}\otimes y,(\Delta_{-1,0} \otimes \id)\circ \Delta_{-1,0}(F)\rangle\\
		&=-\lambda \langle x\star y,F\rangle+\mu \sum_{\omega \in X\sqcup \Omega} \langle (x\star\tdun{$\omega$})\star y,F\rangle\\
		&=\langle -\lambda x\star y+\mu\sum_{\omega \in X\sqcup \Omega} x\star \tdun{$\omega$} \star y,F\rangle,
	\end{align*}
	which gives the result.
\end{proof}

\begin{exam}
	Let $\alpha,\beta,\gamma\in \Omega$.
	\begin{align*}
		\tdun{$\alpha$}\star_{\lambda,\mu} \tddeux{$\beta$}{$\gamma$}
		=&-\lambda\left(\tdun{$\alpha$}\tddeux{$\beta$}{$\gamma$}+\tdtroisun{$\beta$}{$\gamma$}{$\alpha$}
		+\tdtroisdeux{$\beta$}{$\gamma$}{$\alpha$}\right)\\
		&+\mu\sum_{\omega \in \Omega}\left(
		\tdun{$\alpha$}\tdun{$\omega$}\tddeux{$\beta$}{$\gamma$}+\tddeux{$\omega$}{$\alpha$}\tddeux{$\beta$}{$\gamma$}
		+\tdun{$\alpha$}\tdtroisun{$\beta$}{$\gamma$}{$\omega$}+\tdquatreun{$\beta$}{$\gamma$}{$\omega$}{$\alpha$}
		+\tdquatredeux{$\beta$}{$\gamma$}{$\omega$}{$\alpha$}
		+\tdun{$\alpha$}\tdtroisdeux{$\beta$}{$\gamma$}{$\omega$}+\tdquatretrois{$\beta$}{$\gamma$}{$\omega$}{$\alpha$}
		+\tdquatrequatre{$\beta$}{$\gamma$}{$\omega$}{$\alpha$}+\tdquatrecinq{$\beta$}{$\gamma$}{$\omega$}{$\alpha$}\right)\\
		&+\mu\sum_{x \in X}\left(
		\tdun{$\alpha$}\tdun{$x$}\tddeux{$\beta$}{$\gamma$}+
		\tdun{$\alpha$}\tdtroisun{$\beta$}{$\gamma$}{$x$}+\tdquatreun{$\beta$}{$\gamma$}{$x$}{$\alpha$}
		+\tdun{$\alpha$}\tdtroisdeux{$\beta$}{$\gamma$}{$x$}+\tdquatretrois{$\beta$}{$\gamma$}{$x$}{$\alpha$}
		+\tdquatrequatre{$\beta$}{$\gamma$}{$x$}{$\alpha$}\right).
	\end{align*}	
\end{exam}	

\section{Free $\Omega$-cocycle infinitesimal unitary bialgebras}\label{uni}
The objective of this section is to investigate weighted infinitesimal unitary bialgebras within the framework of operated algebras. This perspective naturally gives rise to the notion of weighted $\Omega$-cocycle infinitesimal unitary bialgebras. We establish that $\hrts$ serves as a free object in one of these categories.
As an application, we get some isomorphisms between different cocycle infinitesimal unitary bialgebras.
\subsection{Main theorem}
\begin{defn}\cite[Section~1.2]{Guo09}
	\begin{enumerate}
		\item
		An {\bf $\Omega$-operated monoid } is a monoid $M$ together with a set of operators $P_{\omega}: M\to M$, $\omega\in \Omega$.
		\item
		An {\bf $\Omega$-operated unitary algebra } is a unitary algebra $A$ together with a set of linear operators $P_{\omega}: A\to A$, $\omega\in \Omega$.
	\end{enumerate}
\end{defn}

\begin{exam}
	Here are some examples of $\Omega$-operated monoid and $\Omega$-operated unitary algebras.
	\begin{enumerate}
		
		\item The monoid $\rfs$ generated by decorated planar rooted forests constructed in Section~\mref{sucsec:deco} of this paper, together with a family of grafting operators $ \{B_{\omega}^+\mid \omega\in \Omega\}$, forms an {\bf $\Omega$-operated monoid}.
		\item A {\bf differential algebra of weight $\lambda$} \cite{GK08} is defined to be an algebra $A$ together
		with a linear operator $d : A \to A$ such that
		\[
		d(xy) = d(x)y + x d(y) + \lambda d(x)d(y), \quad \forall x, y \in A.
		\]
		\item A {\bf difference algebra} \cite{RM65} is an algebra $A$ with an algebra endomorphism on $A$.
		\item Let $\lambda$ be fixed in the ground ring {\bf k}. A {\bf Rota-Baxter algebra} (of weight $\lambda$) \cite{Gub} is defined to be an associative algebra $A$ with a linear operator $P$ such that
		\begin{equation*}
			P(x)P(y) = P(xP(y)) + P(P(x)y) + \lambda P(xy), \quad \forall x,y \in A.
		\end{equation*}
		\item  A {\bf Rota-Baxter family} \cite{EGP,Guo09} on an algebra $A$ is a collection of linear operators $P_\omega$ on
		$A$ with $\omega$ in a semigroup $\Omega$, such that	
		\[
		P_\alpha(x) P_\beta(y) = P_{\alpha\beta}(P_\alpha(x)y) + P_{\alpha\beta}(x P_\beta(y)) + \lambda P_{\alpha\beta}(xy), \quad \forall x, y \in A,\ \alpha, \beta \in \Omega.
		\]
		Rota-Baxter families have extensive applications in the renormalization of quantum field theory\cite{EGP}.
	\end{enumerate}
\end{exam}

\begin{defn}\cite{ZCGL18}
	Let $\lambda$ be a given element of $\bfk$.
	\begin{enumerate}
		\item  An  {\bf $\Omega$-operated  $\epsilon$-bialgebra} {\bf of weight $\lambda$} is an $\epsilon$-bialgebra $H$ of weight $\lambda$ together with a set of linear operators $P_{\omega}: H\to H$, $\omega\in \Omega$.
		\mlabel{it:def1}
		\item Let $(H,\, \{P_{\omega}\mid \omega \in \Omega\})$ and $(H',\,\{P'_{\omega}\mid \omega\in \Omega\})$ be two $\Omega$-operated $\epsilon$-bialgebras of weight $\lambda$. A linear map $\phi : H\rightarrow H'$ is called an {\bf $\Omega$-operated $\epsilon$-bialgebra morphism} if $\phi$ is a morphism of $\epsilon$-bialgebras of weight $\lambda$ and $\phi \circ P_\omega = P'_\omega \circ\phi$ for $\omega\in \Omega$.
	\end{enumerate}
\end{defn}

By a  pair weight 1-cocycle condition, we then propose

\begin{defn}
	Let $\lambda$ be a given element of $\bfk$.
	\begin{enumerate}

		\item \mlabel{it:def3}
		An {\bf $\Omega$-cocycle $\epsilon$-unitary bialgebra} {\bf of weight $\lambda$} is an $\Omega$-operated  $\epsilon$-unitary bialgebra $(H,\,m,\,1_H, \Delta, \{P_{\omega}\mid \omega \in \Omega\})$ of weight $\lambda$ satisfying the  pair weight $1$-cocycle condition:
		\begin{equation}
			\Delta P_{\omega} = P_{\omega} \otimes (-\lambda1_H) + \mu \id \ot 1_H+ (\id\otimes P_{\omega}) \Delta \quad \text{ for } \omega\in \Omega.
			\mlabel{eq:eqiterated}
		\end{equation}
		
		\item The {\bf free $\Omega$-cocycle $\epsilon$-unitary bialgebra of weight $\lambda$ on a set $X$} is an $\Omega$-cocycle $\epsilon$-unitary bialgebra
		$(H_{X},\,m_{X}, \,1_{H_X}, \Delta_{X}, \,\{P_{\omega}\mid \omega \in \Omega\})$ of weight $\lambda$ together with a set map $j_X: X \to H_{X}$  with the property that,
		for any $\Omega$-cocycle $\epsilon$-unitary bialgebra $(H,\,m,\,1_H, \Delta,\,\{P'_{\omega}\mid \omega \in \Omega\})$ of weight $\lambda$ and any set map
		$f: X\to H$ such that for any $x\in X$,
		$$\Delta (f(x))=\mu (1_H \ot 1_H)-\lambda(f(x)\ot 1_H+ 1_H\ot f(x)),$$
		there is a unique $\Omega$-operated $\epsilon$-unitary bialgebras unique morphism $\free{f}:H_X\to H$
		such that the diagram
		\begin{align*}
			\xymatrix{
				X \ar[dr]_{f} \ar[r]^{j_X  }
				& H_X \ar[d]^{\bar{f}}  \\
				& H            }
		\end{align*}
		commutes.
		\mlabel{it:def4}
	\end{enumerate}
	\mlabel{defn:xcobi}
\end{defn}

When $\Omega$ is a singleton set, we will omit it.
From now on, our discussion takes place on $\hrts= \bfk \rfs$. The following results generalizes the universal properties which were studied in~\cite{CK98, Fo3, Guo09, Moe01, ZGG16}.

\begin{lemma}\cite[Theorem~4.5]{ZCGL18}
	Let $j_{X}: X\hookrightarrow \rfs$, $x \mapsto \bullet_{x}$ be the natural embedding and $m_{\RT}$ be the concatenation product. Then, we have the following.
	\begin{enumerate}
		\item
		The quadruple $(\rfs, \,\mul,\, 1, \, \{B_{\omega}^+\mid \omega\in \Omega\})$ together with the $j_X$ is the free $\Omega$-operated monoid on $X$.
		\mlabel{it:fomonoid}
		\item
		The quadruple $(\hrts, \,\mul,\,1, \, \{B_{\omega}^+\mid \omega\in \Omega\})$ together with the $j_X$ is the free $\Omega$-operated unitary algebra on $X$.
		\mlabel{it:fualg}
	\end{enumerate}
	\mlabel{lemm:free}
\end{lemma}

\begin{theorem}\label{thm:propm}
	Let $j_{X}: X\hookrightarrow \rfs$, $x \mapsto \bullet_{x}$ be the natural embedding and $m_{\RT}$ be the concatenation product. Then
	the quintuple $(\hrts, \,\mul,\,1, \, \col, \, \{B_{\omega}^+\mid \omega\in \Omega\})$ together with the $j_X$ is the free $\Omega$-cocycle $\epsilon$-unitary bialgebra of weight $\lambda$ on $X$.
	\mlabel{it:fubialg}
\end{theorem}

\begin{proof}
	By Theorem~\mref{thm:rt2}, $(\hrts, \,\mul,\,1, \col)$ is an $\epsilon$-unitary bialgebra of weight $\lambda$. Then it follows from Eq.~(\mref{eq:dbp}) that $(\hrts, \,\mul,\,1, \col,\,\{B_{\omega}^+\mid \omega\in \Omega\})$ is an $\Omega$-cocycle $\epsilon$-unitary bialgebra of weight $\lambda$.
	
	We next proceed to show the freeness. Let $(H,\, m,\,1_H, \Delta,\, \{P_{\omega}\mid \omega \in \Omega\})$ be an $\Omega$-cocycle $\epsilon$-bialgebra of weight $\lambda$ and
	$f: X\to H$ a set map such that
	\begin{align*}
		\Delta (f(x))=\mu (1_H \ot 1_H)-\lambda(f(x)\ot 1_H+ 1_H\ot f(x)) \text{ for all } x \in X.
	\end{align*}
	Particularly, $(H,\, m,\, 1_H,\, \{P_{\omega}\mid \omega \in \Omega\})$ is an $\Omega$-operated unitary algebra.
	By Lemma~\mref{lemm:free} (\mref{it:fualg}), there exists a unique $\Omega$-operated unitary algebra morphism $\free{f}:\hrts \to H$ such that $\free{f}\circ j_X={f}$.
	It is sufficient to  check the compatibility of the coproducts $\Delta$ and $\col$ for which we verify
	\begin{equation}
		\Delta \free{f} (F)=(\free{f}\ot \free{f}) \col (F)\quad \text{for all } F\in \rfs,
		\mlabel{eq:copcomp}
	\end{equation}
	by induction on the depth $\dep(F)\geq 0$.
	For the initial step of $\dep(F)=0$,
	we have $F = \bullet_{x_1} \cdots \bullet_{x_m}$ for some $m\geq 0$, with the convention that $F=\etree$ when $m=0$.
	If $m=0$, then by Remark~\mref{remk:4rem}~(\mref{remk:units}) and Eq.~(\mref{eq:dele}),
	\begin{align*}
		\Delta  \bar{f} (F) &= \Delta  \bar{f} (1)=\Delta (1_H)=-\lambda (1_H\ot 1_H) =-\lambda \bar{f} (1)\ot\bar{f} (1) =(\bar{f}\otimes\bar{f})(-\lambda1\ot 1)=(\bar{f}\otimes\bar{f})\col(1).
	\end{align*}
	If $m\geq 1$, then we have
	\begin{align*}
		&\ \Delta\free{f}(\bullet_{x_1} \cdots \bullet_{x_m})
		= \Delta \Big( \free{f}(\bullet_{x_1})\cdots\free{f}(\bullet_{x_{m}}) \Big)\\
		=&\ \cdots=\sum_{i=1}^m \Big(\free{f}(\bullet_{x_{1}})\cdots\free{f}(\bullet_{x_{i-1}})\Big)\cdot
		\Delta\big(\free{f}(\bullet_{x_{i}})\big) \cdot \Big(\free{f}(\bullet_{x_{i+1}})\cdots\free{f}(\bullet_{x_{m}})\Big)\\
		&+\lambda\sum_{i=1}^{m-1}\free{f}(\bullet_{x_{1}})\cdots\free{f}(\bullet_{x_{i}}) \ot \free{f}(\bullet_{x_{i+1}})\cdots \free{f}(\bullet_{x_{m}})
		\quad \text{(by Eq.~(\ref{eq:cocycle}))}\\
		=&\ \sum_{i=1}^m \Big(\free{f}(\bullet_{x_{1}})\cdots\free{f}(\bullet_{x_{i-1}})\Big)\cdot
		\Big(\mu \free{f}(1_H) \ot \free{f}(1_H)-\lambda \free{f}(\bullet_{x_{i}})\ot 1_H-\lambda 1_H\ot \free{f}(\bullet_{x_{i}})\Big)\cdot \Big(\free{f}(\bullet_{x_{i+1}})\cdots\free{f}(\bullet_{x_{m}})\Big)\\
		&+\lambda\sum_{i=1}^{m-1}\free{f}(\bullet_{x_{1}})\cdots\free{f}(\bullet_{x_{i}}) \ot \free{f}(\bullet_{x_{i+1}})\cdots \free{f}(\bullet_{x_{m}})\\
		&\ \  \ \ \ (\text{by}\,\, \Delta(\free{f}(\bullet_{x_{i}}))=\Delta(f(x_i))=\mu (f(1_H) \ot f(1_H))-\lambda (f(x_i)\ot 1_H+1_H\ot f(x_i)))\\
		=&\ \mu \sum_{i=1}^m \free{f}(\bullet_{x_{1}})\cdots
		\free{f}(\bullet_{x_{i-1}}) \ot \free{f}(\bullet_{x_{i+1}})\cdots \free{f}(\bullet_{x_{m}})
		-\lambda \sum_{i=1}^m \free{f}(\bullet_{x_{1}})\cdots\cdot
		\free{f}(\bullet_{x_{i}})\ot \free{f}(\bullet_{x_{i+1}})\cdots\free{f}(\bullet_{x_{m}})\\
		&\ -\lambda\sum_{i=1}^m \free{f}(\bullet_{x_{1}})\cdots\cdot
		\free{f}(\bullet_{x_{i-1}})\ot \free{f}(\bullet_{x_{i}})\cdots\free{f}(\bullet_{x_{m}})
		+\lambda\sum_{i=1}^{m-1}\free{f}(\bullet_{x_{1}})\cdots\free{f}(\bullet_{x_{i}}) \ot \free{f}(\bullet_{x_{i+1}})\cdots \free{f}(\bullet_{x_{m}})\\
		=&\ \mu \sum_{i=1}^m \free{f}(\bullet_{x_{1}})\cdots
		\free{f}(\bullet_{x_{i-1}}) \ot \free{f}(\bullet_{x_{i+1}})\cdots \free{f}(\bullet_{x_{m}})-\lambda \sum_{i=1}^{m-1} \free{f}(\bullet_{x_{1}})\cdots\cdot
		\free{f}(\bullet_{x_{i}})\ot \free{f}(\bullet_{x_{i+1}})\cdots\free{f}(\bullet_{x_{m}})\\
		&\ -\lambda\free{f}(\bullet_{x_{1}})\cdots\cdot
		\free{f}(\bullet_{x_{m}})\ot \free{f}(\etree)
		-\lambda\sum_{i=1}^m \free{f}(\bullet_{x_{1}})\cdots\cdot
		\free{f}(\bullet_{x_{i-1}})\ot \free{f}(\bullet_{x_{i}})\cdots\free{f}(\bullet_{x_{m}})\\
		&\
		+\lambda\sum_{i=1}^{m-1}\free{f}(\bullet_{x_{1}})\cdots\free{f}(\bullet_{x_{i}}) \ot \free{f}(\bullet_{x_{i+1}})\cdots \free{f}(\bullet_{x_{m}})\\
		=&\ \mu \sum_{i=1}^m \free{f}(\bullet_{x_{1}})\cdots
		\free{f}(\bullet_{x_{i-1}}) \ot \free{f}(\bullet_{x_{i+1}})\cdots \free{f}(\bullet_{x_{m}})
		-\lambda\free{f}(\bullet_{x_{1}})\cdots\cdot
		\free{f}(\bullet_{x_{m}})\ot \free{f}(\etree)\\
		&\ -\lambda\sum_{i=1}^m \free{f}(\bullet_{x_{1}})\cdots\cdot
		\free{f}(\bullet_{x_{i-1}})\ot \free{f}(\bullet_{x_{i}})\cdots\free{f}(\bullet_{x_{m}})\\
		=&\ \mu \sum_{i=1}^m \free{f}(\bullet_{x_{1}})\cdots
		\free{f}(\bullet_{x_{i-1}}) \ot \free{f}(\bullet_{x_{i+1}})\cdots \free{f}(\bullet_{x_{m}})
		-\lambda\free{f}(\bullet_{x_{1}})\cdots\cdot
		\free{f}(\bullet_{x_{m}})\ot \free{f}(\etree)\\
		&\ -\lambda\sum_{i=0}^{m-1} \free{f}(\bullet_{x_{1}})\cdots\cdot
		\free{f}(\bullet_{x_{i}})\ot \free{f}(\bullet_{x_{i+1}})\cdots\free{f}(\bullet_{x_{m}})\\
		=&\ (\free{f}\ot\free{f})\left(\mu \sum_{i=1}^{m}\bullet_{x_{1}}\cdots\bullet_{x_{i-1}}\otimes\bullet_{x_{i+1}}\cdots\bullet_{x_{m}}
		-\lambda
		\sum_{i=0}^{m}
		\bullet_{x_{1}}\cdots\bullet_{x_{i}}\otimes\bullet_{x_{i+1}}\cdots\bullet_{x_{m}}\right)\\
		=&\ (\free{f}\ot\free{f})\col(\bullet_{x_1} \cdots \bullet_{x_m})\quad(\text{by Lemma~\ref{lem:rt11}}).
	\end{align*}
	Suppose Eq.~(\mref{eq:copcomp}) holds for $\dep(F)\leq n$ for an $n\geq 0$ and consider the case of $\dep(F)=n+1$.
	For this case we apply the induction on the breadth $\bre(F)$. Since $\dep(F)=n+1\geq1$, we have $F\neq 1$ and $\bre(F)\geq 1$.
	If $\bre(F)=1$, we have $F=B_{\omega}^+(\lbar{F})$ for some $\lbar{F}\in\rfs$ and $\omega\in \Omega$.  Then
	\allowdisplaybreaks{
		\begin{align*}
			\Delta \free{f}(F)&=\Delta \free{f} (B_{\omega}^{+}(\lbar{F}))=\Delta P_\omega(\free{f} (\lbar{F}))\quad(\text{by $\bar{f}$ being an operated algebra morphism})\\
			&=P_{\omega}(\free{f}(\lbar{F})) \ot (-\lambda1_H)+\mu \free{f}(\lbar{F}) \ot 1_H + (\id \ot P_{\omega})\Delta \free{f}(\lbar{F})
			\quad(\text{by Eq.~(\mref{eq:eqiterated})})\\
			&=P_{\omega}(\free{f}(\lbar{F})) \ot (-\lambda1_H)+\mu \free{f}(\lbar{F}) \ot 1_H + (\id\ot P_{\omega})(\free{f}\ot \free{f}) \col (\lbar{F}) \\
			&\hspace{6cm}(\text{by the induction hypothesis on~}\dep(F)) \\
			&=P_{\omega}(\free{f}(\lbar{F})) \ot (-\lambda1_H)+\mu \free{f}(\lbar{F}) \ot 1_H + (\free{f} \ot P_{\omega}\free{f}) \col (\lbar{F})\\
			&=\free{f}(B_{\omega}^+(\lbar{F})) \ot (-\lambda1_H)+\mu \free{f}(\lbar{F}) \ot 1_H + (\free{f} \ot \free{f}B_{\omega}^+) \col (\lbar{F}) \\
			&\hspace{6cm}(\text{by $\bar{f}$ being an operated algebra morphism}) \\
			&=(\free{f}\ot \free{f})\Big(-\lambda B_{\omega}^{+}(\lbar{F}) \otimes \etree+ \mu\lbar{F} \otimes \etree + (\id\otimes B_{\omega}^{+})\col(\lbar{F})\Big) \\
			&=(\free{f}\ot \free{f}) \col (B_{\omega}^+(\lbar{F}))\quad(\text{by Eq.~(\mref{eq:dbp})}) \\
			&=(\free{f}\ot \free{f}) \col (F).
		\end{align*}
	}
	Assume Eq.~(\mref{eq:copcomp}) holds for $\dep(F)=n+1$ and $\bre(F)\leq m$, in addition to $\dep(F)\leq n$ by the first induction hypothesis, and consider the case when $\dep(F)=n+1$ and $\bre(F)=m+1\geq 2$. Then  $F=F_{1}F_{2}$ for some $F_{1},F_{2}\in\rfs$ with $0< \bre(F_{1}), \bre(F_{2}) < m+1$.
	With Sweedler notation, we may write
	\begin{align*}
		\col(F_1)=\sum_{(F_1)}F_{1(1)}\otimes F_{1(2)} \text{\ and \ } \col(F_2)=\sum_{(F_2)}F_{2(1)}\otimes F_{2(2)}.
	\end{align*}
	By the induction hypothesis on the breadth, we have
	\begin{align*}
		&\Delta(\free{f}(F_{1}))=(\free{f} \ot \free{f})\col(F_{1})=\sum_{(F_{1})}\free{f}(F_{1(1)})\ot\free{f} (F_{1(2)}),\\
		&\Delta(\free{f}(F_{2}))=(\free{f} \ot \free{f})\col(F_{2})=
		\sum_{(F_{2})}\free{f}(F_{2(1)})\ot\free{f} (F_{2(2)}).
	\end{align*}
	Thus
	\allowdisplaybreaks{
		\begin{align*}
			\Delta \free{f}(F)=&\ \Delta \free{f} (F_{1}F_{2})=\Delta(\free{f}(F_1)\free{f}(F_2))\\
			=&\ \free{f}(F_{1})\cdot \Delta (\free{f}(F_{2}))+\Delta(\free{f}(F_{1})) \cdot \free{f}(F_{2})+\lambda\free{f}(F_{1})\ot \free{f}(F_{2})\quad (\text{by Eq.~(\mref{eq:cocycle})})\\
			=&\ \free{f}(F_{1})\cdot \bigg(\sum_{(F_{2})}\free{f}(F_{2(1)})\ot \free{f}(F_{2(2)})\bigg)+
			\bigg(\sum_{(F_{1})}\free{f}(F_{1(1)})\ot\free{f} (F_{1(2)})\bigg) \cdot \free{f}(F_{2})+\lambda\free{f}(F_{1})\ot \free{f}(F_{2})\\
			=&\ \sum_{(F_{2})}\free{f}(F_{1})\free{f}(F_{2(1)})\ot\free{f}(F_{2(2)})
			+\sum_{(F_{1})}\free{f}(F_{1(1)})\ot\free{f} (F_{1(2)})\free{f}(F_{2})+\lambda\free{f}(F_{1})\ot \free{f}(F_{2}) \ (\text{by Eq.~(\ref{eq:dota})})\\
			=&\ \sum_{(F_{2})}\free{f}(F_{1}F_{2(1)})\ot\free{f}(F_{2(2)})+\sum_{(F_{1})}\free{f}(F_{1(1)})\ot\free{f} (F_{1(2)}F_{2})+\lambda\free{f}(F_{1})\ot \free{f}(F_{2})\\
			=&\ (\free{f}\ot \free{f})\bigg(\sum_{(F_{2})}F_{1}F_{2(1)}\ot F_{2(2)}\bigg)+(\free{f}\ot \free{f})\bigg( \sum_{(F_{1})}F_{1(1)}\ot F_{1(2)}F_{2}\bigg)+(\free{f}\ot \free{f})(\lambda F_1\ot F_2)\\
			=&\ (\free{f}\ot \free{f})\left(\sum_{(F_{2})}F_{1}F_{2(1)}\ot F_{2(2)}
			+\sum_{(F_{1})}F_{1(1)}\ot F_{1(2)}F_{2} +\lambda F_1\ot F_2\right)\\
			=&\ (\free{f}\ot \free{f})\left(F_{1}\cdot \sum_{(F_{2})}F_{2(1)}\ot F_{2(2)}+ \Big(\sum_{(F_{1})}F_{1(1)}\ot F_{1(2)}\Big) \cdot F_{2} +\lambda F_1\ot F_2\right) \quad (\text{by Eq.~(\ref{eq:dota})})\\
			=&\ (\free{f} \ot \free{f})(F_{1} \cdot \col(F_{2})+\col(F_{1})\cdot F_{2}+\lambda(F_1\ot F_2))\\
			=&\ (\free{f}\ot \free{f})\col(F_{1}F_{2})  \quad(\text{by Lemma~\mref{lem:colff}})\\
			=&\ (\free{f}\ot \free{f})\col(F).
		\end{align*}
	}
	This completes the induction on the breadth and hence the induction on the depth.
\end{proof}


Let $X=\emptyset$. Then we obtain a freeness of $\hck(\emptyset, \Omega)$, which is the infinitesimal version of decorated noncommutative Connes-Kreimer Hopf algebra by Remark~\mref{re:3ex}~(\mref{it:2ex}).
\begin{coro}
	The quintuple $(\hck(\emptyset, \Omega), \,\mul,\,\etree,\, \col,\,\{B_{\omega}^+\mid \omega\in \Omega\})$ is the free $\Omega$-cocycle $\epsilon$-unitary bialgebra of weight $\lambda$ on the empty set, that is, the initial object in the category of $\Omega$-cocycle $\epsilon$-unitary bialgebras of weight $\lambda$.
	\mlabel{cor:rt16}
\end{coro}
\begin{proof}
	It follows from Theorem~\mref{thm:propm}~(\mref{it:fubialg}) by taking $X=\emptyset$.
\end{proof}

Taking $\Omega$ to be singleton in Corollary~\mref{cor:rt16}, then all planar rooted forests are decorated by the same letter. In this case, planar rooted forests have no decorations that are precisely planar rooted forests in the classical Connes-Kreimer Hopf algebra under the noncommutative version which was introduced by the first author~\mcite{Foi02} and Holtkamp~\mcite{Hol03}.

\begin{coro}
	Let $\mathcal{F}$ be the set of planar rooted forests without decorations.
	Then the quintuple $(\bfk\calf, \,\mul,\,\etree,\, \col,\, B^+)$ is the free cocycle $\epsilon$-unitary bialgebra of weight $\lambda$ on the empty set, that is, the initial object in the category of $\Omega$-cocycle $\epsilon$-unitary bialgebras of weight $\lambda$.
	\mlabel{coro:rt16}
\end{coro}

\begin{proof}
	It follows from Corollary~\mref{cor:rt16}  by taking $\Omega$ to be a singleton set.
\end{proof}

\subsection{An application: isomorphisms between different $\col$}
In this subsection, we give an application for the universal property of the free $\Omega$-cocycle $\epsilon$-unitary bialgebra $\hrts$ of weight $\lambda$.

\begin{theorem}
	Let $\nu \in \bfk$. We denote by $\phi_\nu$ the unique algebra endomorphism of $\hrts$ such that:
	\begin{itemize}
		\item For any $\omega \in \Omega$, $\phi_\nu \circ B^+_\omega=(B^+_\omega+\nu \id_{\hrts})\circ \phi_\nu$.
		\item For any $x\in X$, $\phi_\nu(\tdun{$x$})=\tdun{$x$}+\nu 1$.
	\end{itemize}
	Then:
	\begin{enumerate}
		\item For any $F\in \rfs$,
		\[\phi_\nu(F)=\sum_{V_X(F)\subseteq I\subseteq V(F)} \nu^{|V(F)\setminus I|} F_{ I},\]
		where $V_X(F)$ is the set of vertices of $F$ decorated by an element of $X$.
		\item For any $\nu,\nu'\in \bfk$, $\phi_\nu\circ \phi_{\nu'}=\phi_{\nu+\nu'}$. As a consequence, $\phi_\nu$ is an algebra automorphism, of inverse $\phi_{-\nu}$.
		\item For any $\nu\in \bfk$, $\phi_\nu$ is a bialgebra isomorphism from $(\hrts,m,\Delta_{\lambda,\mu})$ to $(\hrts,m,\Delta_{\lambda,\mu-\lambda\nu})$.
\end{enumerate}\end{theorem}

\begin{proof}
	(a). By induction on $n_F=|V(F)|$. If $F=1$, then $\phi_\nu(F)=1$ and the result is obvious. If $n_F>0$, two cases are possible.
	
	\textbf{Case 1.} If $\bre(F)>1$, we put $F=F_1F_2$, with $n_{F_1},n_{F_2}<n_F$ if $i=1,2$. We apply the induction hypothesis on $F_1$ and $F_2$.
	\begin{align*}
		\phi_\nu(F)&=\phi_\nu(F_1)\phi_\nu(F_2)\\
		&=\sum_{\substack{I_1 \subseteq V(F_1),\\ I_2 \subseteq V(F_2)}}\nu^{|V(F_1)\setminus I_1|+|V(F_2)\setminus I_2|}
		{F_1}_{\mid I_1}{F_2}_{\mid I_2}\\
		&=\sum_{\substack{I_1 \subseteq V(F_1),\\ I_2 \subseteq V(F_2)}} \nu^{|V(F)\setminus (I_1 \cup I_2)|}F_{\mid I_1\cup I_2}\\
		&=\sum_{I\subseteq V(F)} \nu^{|V(F)\setminus I|} F_{\mid I}.
	\end{align*}
	\textbf{Case 1.} If $F$ is a tree, two subcases can occur. If $F=\tdun{$x$}$, with $x\in X$, the result is obvious. Otherwise, we put $F=B^+_\omega(G)$ and denote by $r$ the root of $F$. Applying the induction hypothesis on $G$,
	\begin{align*}
		\phi_\nu(F)&=\phi_\nu \circ B^+_\omega(G)\\
		&=(B^+_\omega+\nu \id_{\hrts})\circ \phi_\nu(G)\\
		&=\sum_{I\subseteq V(G)} \nu^{|V(G)\setminus I|} B^+_\omega(G_{\mid I})+
		\sum_{I\subseteq V(G)} \nu^{|V(G)\setminus I|+1} (G_{\mid I})\\
		&=\sum_{I\subseteq V(F),\: r\in I} \nu^{|V(F)\setminus I|} F_{\mid I}+
		\sum_{I\subseteq V(F),\: r\notin I} \nu^{|V(F)\setminus I|} (F_{\mid I})\\
		&=\sum_{I\subseteq V(F)} \nu^{|V(F)\setminus I|} F_{\mid I}.
	\end{align*}
	
	(b). Let us consider the algebra morphism $\psi=\phi_\nu\circ \phi_{\nu'}$. For any $x\in X$,
	\[\psi(\tdun{$x$})=\phi_\nu(\tdun{$x$}+\nu'1)=\tdun{$x$}+\nu 1+\nu'1=\tdun{$x$}+(\nu+\nu')1.\]
	For any $\omega \in \Omega$,
	\begin{align*}
		\psi\circ B^+_\omega&=\phi_\nu \circ (B^+_\omega+\nu'\id_{\hrts})\circ \phi_{\nu'}\\
		&=(B^+_\omega+\nu \id_{\hrts})\circ \phi_\nu\circ \phi_{\nu'}+\nu' \phi_\nu\circ \phi_{\nu'}\\
		&=B^+_\omega \circ \psi+(\nu+\nu')\psi\\
		&=(B^+_\omega+(\nu+\nu')\id_{\hrts})\circ \psi.
	\end{align*}
	So $\psi=\phi_{\nu+\nu'}$. In particular, $\phi_0\circ B^+_\omega=B^+_\omega\circ \phi_0$ for any $\omega\in \Omega$,
	so $\phi_0=\id_{\hrts}$. This implies that for any $\nu\in \bfk$, $\phi_\nu$ is an automorphism,
	of inverse $\phi_{-\nu}$.\\
	
	(c). Let us fix $\nu\in \bfk$ and consider the map $L_\omega=B^+_\omega+\nu \id_{\hrts}$.
	Then, for any $x\in \hrts$, for any $\omega \in \Omega$,
	\begin{align*}
		\Delta_{\lambda,\mu-\lambda\nu}\circ L_\omega(x)&=\Delta_{\lambda,\mu-\lambda\nu}\circ B^+_\omega(x)+\nu \Delta_{\lambda,\mu-\lambda\nu}(x)\\
		&=-\lambda B^+_\omega(x)\otimes 1+(\mu-\lambda\nu) x\otimes 1+(\id_{\hrts} \otimes B^+_\omega)\circ \Delta_{\lambda,\mu-\lambda\nu}(x)+\nu \Delta_{\lambda,\mu-\lambda\nu}(x)\\
		&=-\lambda L_\omega(x)\otimes 1+\mu x\otimes 1+(\id_{\hrts} \otimes L_\omega)\circ \Delta_{\lambda,\mu-\lambda\nu}(x).
	\end{align*}
	For any $x\in X$,
	\begin{align*}
		\Delta_{{\lambda,\mu-\lambda\nu}}(\tdun{$x$}+\nu 1)&=(\mu-\lambda\nu)1\otimes 1-\lambda(\tdun{$x$}\otimes 1+1\otimes \tdun{$x$})-\lambda\nu 1\otimes 1\\
		&=(\mu-2\lambda\nu)1\otimes 1-\lambda(\tdun{$x$}\otimes 1+1\otimes \tdun{$x$})\\
		&=\mu 1\otimes 1-\lambda\left((\tdun{$x$}+\nu 1)\otimes 1+1\otimes (\tdun{$x$}+\nu 1)\right).
	\end{align*}
	The universal property of $(\hrts,m,\Delta_{\lambda\mu},\{B_\omega^+\mid \omega \in\Omega\})$ insures that $\phi_\nu$
	is a bialgebra morphism from $(\hrts,m,\Delta_{\lambda,\mu})$ to $(\hrts,m,\Delta_{\lambda,\mu-\lambda\nu})$.
\end{proof}

\begin{exam}
	Let $\alpha,\beta,\gamma\in \Omega$.
	\begin{align*}
		\phi_\nu(\tdun{$\alpha$})&=\tdun{$\alpha$}+\nu 1,\\
		\phi_\nu(\tddeux{$\alpha$}{$\beta$})&=\tddeux{$\alpha$}{$\beta$}+\nu(\tdun{$\alpha$}+\tdun{$\beta$})+\nu^21,\\
		\phi_\nu(\tdtroisun{$\alpha$}{$\gamma$}{$\beta$})&=\tdtroisun{$\alpha$}{$\gamma$}{$\beta$}
		+\nu(\tddeux{$\alpha$}{$\beta$}+\tddeux{$\alpha$}{$\gamma$}+\tdun{$\beta$}\tdun{$\gamma$})
		+\nu^2(\tdun{$\alpha$}+\tdun{$\beta$}+\tdun{$\gamma$})+\nu^31,\\
		\phi_\nu(\tdtroisdeux{$\alpha$}{$\beta$}{$\gamma$})&=\tdtroisdeux{$\alpha$}{$\beta$}{$\gamma$}
		+\nu(\tddeux{$\alpha$}{$\beta$}+\tddeux{$\alpha$}{$\gamma$}+\tddeux{$\beta$}{$\gamma$})
		+\nu^2(\tdun{$\alpha$}+\tdun{$\beta$}+\tdun{$\gamma$})+\nu^31.
	\end{align*}
\end{exam}

This family of isomorphisms do not content any morphism between $\Delta_{0,\nu}$ and $\Delta_{0,\nu'}$ for distinct $\nu$ and $\nu'$.
Here is another family with such morphisms.

\begin{prop}
	For any $\nu \in \bfk$, we define an endomorphism $\theta_\nu$ of $\hrts$ by
	\[\theta_\nu(F)=\nu^{n_F}F.\]
	Then $\theta_\nu$ is a bialgebra morphism from $(\hrts,m,\Delta_{\lambda,\mu\nu})$ to $(\hrts,m,\Delta_{\lambda,\mu})$.
\end{prop}

\begin{proof}
	Obviously, $\theta_\mu$ is an algebra morphism. Let $F\in \rfs$.
	\begin{align*}
		\Delta_{\lambda,\mu}\circ \theta_\nu(F)=&\ -\lambda\left(\sum_{k=0}^{n_F} \nu^{n_F}F_{\mid I_k}\otimes F_{\mid J_k}\right)
		+\mu \left(\sum_{k=1}^{n_F}\nu^{n_F} F_{\mid I_{k-1}}\otimes F_{\mid J_k}\right)\\
		=&\ -\lambda\left(\sum_{k=0}^{n_F} \nu^{n_{F_{\mid I_k}}}F_{\mid I_k}\otimes \nu^{n_{F_{\mid J_k}}} F_{\mid J_k}\right)\\
		&\ +\mu\nu \left(\sum_{k=1}^{n_F}\nu^{n_{F_{\mid I_{k-1}}}} F_{\mid I_{k-1}}\otimes \nu^{n_{F_{\mid J_k}}}F_{\mid J_k}\right)\\
		=&\ (\theta_\nu \otimes \theta_\nu)\circ \Delta_{\lambda,\mu\nu}(F).\qedhere
	\end{align*}
\end{proof}

\begin{remark}
	In particular, if $\bfk$ is a field, we obtain that:
	\begin{itemize}
		\item If $\lambda \neq 0$, then for all $\mu,\mu'\in \bfk$,
		$(\hrts,m,\Delta_{\lambda,\mu})$ and $(\hrts,m,\Delta_{\lambda,\mu'})$ are isomorphic,through $\phi_{\frac{\mu-\mu'}{\lambda}}$.
		\item For all $\mu,\mu'\in \bfk\setminus \{0\}$,
		$(\hrts,m,\Delta_{0,\mu})$ and $(\hrts,m,\Delta_{0,\mu'})$ are isomorphic, through $\theta_{\frac{\mu}{\mu'}}$.
	\end{itemize}
	If $\lambda\neq \lambda'$, then $(\hrts,m,\Delta_{\lambda,\nu})$ and $(\hrts,m,\Delta_{\lambda',\nu'})$ are not isomorphic,
	as they are $\epsilon$-bialgebra of different weights.  Moreover, $(\hrts,m,\Delta_{0,0})$ is not isomorphic to any other
	$(\hrts,m,\Delta_{\lambda,\nu})$, as it is the only one with a zero coproduct.
\end{remark}

\section{Pre-Lie algebras of decorated planar rooted forests}\label{sec:preLie}
In this section, we recall the connection from weighted $\epsilon$-bialgebras to pre-Lie algebras.
Using Theorem~\mref{thm:rt2}, we then construct a new pre-Lie algebraic structure on decorated planar rooted forests.

\subsection{Pre-Lie algebras and infinitesimal unitary bialgebras}
In this subsection, we first recall the concept of the pre-Lie algebras and show the connection from weighted $\epsilon$-bialgebras to pre-Lie algebras.

\begin{defn}~\cite{Man11}
	A {\bf (left) pre-Lie algebra} is a $\bfk $-module $A$ together with a binary operation $\rhd: A\ot A \rightarrow A$ satisfying the left pre-Lie identity:
	\begin{align}
		(a\rhd b)\rhd c-a\rhd (b\rhd c)=(b\rhd a)\rhd c-b\rhd (a\rhd c)\, \text{ for }\, a,b,c\in A.
		\mlabel{eq:lpre}
	\end{align}
\end{defn}

Let $(A, \rhd)$ be a pre-Lie algebra and $\mathbf{gl}(A)$ the space of all linear maps $A\rightarrow A$. For any $a,b\in A$, let
\[L_a:\left\{\begin{array}{rcl}
	A&\longrightarrow&A\\
	b&\longmapsto&a \rhd b
\end{array}\right.\]
be the left multiplication operator. Let $L$ be the linear map defined by
\[L:\left\{\begin{array}{rcl}
	A&\longrightarrow& \mathbf{gl}(A)\\
	a&\longmapsto&L_a.
\end{array}\right.\]
The close relation between pre-Lie algebras and Lie algebras is characterized by the following result.

\begin{lemma}\label{lem:preL}
	\begin{enumerate}
		\item
		\cite[Theorem~1]{Ger63}
		Let $(A, \rhd)$ be a pre-Lie algebra. Define for elements in $A$ a new bilinear operation by setting
		\begin{align*}
			[a, b] :=a\rhd b-b\rhd a\,\text{ for }\, a,b\in A.
		\end{align*}
		Then $(A, [_{-}, _{-}])$ is a Lie algebra.
		\mlabel{lem:preL1}
		\item \cite[Proposition~1.2]{Bai}
		Eq.~(\mref{eq:lpre}) rewrites as
		\begin{align*}
			L_{[a,b]}=[L_a, L_b] = L_a\circ L_b -  L_b\circ L_a,
		\end{align*}
		which implies that $L: (A, [_{-}, _{-}]) \rightarrow \mathbf{gl}(A)$ with $a\mapsto L_a$ gives a representation of the Lie algebra $(A, [_{-}, _{-}])$.
		\mlabel{lem:preL2}
	\end{enumerate}
	
\end{lemma}
By Lemma~\mref{lem:preL}, a pre-Lie algebra induces a Lie algebra whose left multiplication operators give a representation of the associated commutator Lie algebra.

Let $(A, m, \Delta)$ be an $\epsilon$-bialgebra of weight $\lambda$. Define
\begin{align}
	\rhd: A\ot A \to A, \, a\ot b \mapsto a\rhd b:=\sum_{(b)}b^{(1)}a b^{(2)},
	\mlabel{eq:preope}
\end{align}
where $b^{(1)}, b^{(2)}$ are from Sweedler notation $\displaystyle \Delta (b)=\sum_{(b)}b^{(1)}\ot b^{(2)}$. The following result captures the connection from weighted $\epsilon$-bialgebras to pre-Lie algebras, which was studied in~\mcite{CLPZ}. For the completeness, we post it here.

\begin{theorem}
	Let $(A, m, \Delta)$ be an $\epsilon$-bialgebra of weight $\lambda$.
	Then $A$ equipped with the product $\rhd$ defined by Eq.~(\mref{eq:preope}) is a pre-Lie algebra.
	\mlabel{thm:preL}
\end{theorem}

\subsection{Pre-Lie algebras on decorated planar rooted forests}
In this subsection, as an application of Theorem~\ref{thm:preL}, we equip  $\hrts$ with a pre-Lie algebraic structure $(\hrts, \rhd_{\RT})$
and a Lie algebraic structure $(\hrts, [_{-}, _{-}]_{\RT})$.  We also give combinatorial descriptions of $\rhd_{\RT}$ and $[_{-}, _{-}]_{\RT}$, respectively.

\begin{theorem}\label{thm:preope}
	Let $\hrts=(\hrts,m,\col)$ be the $\epsilon$-unitary bialgebra of weight $\lambda$ in Theorem~\mref{thm:rt2}.
	\begin{enumerate}
		\item The pair $(\hrts, \rhd_{\lambda,\mu})$ is a pre-Lie algebra, where
		\begin{equation*}
			F_1\rhd_{\lambda,\mu} F_2: =\sum_{(F_2)}F_{2}^{(1)}F_1 F_{2}^{(2)} \,\text{ for } \, F_1, F_2\in \hrts.
		\end{equation*}
		
		\item The pair $(\hrts, [_{-}, _{-}]_{\lambda,\mu})$ is a Lie algebra, where
		\begin{align*}
			[F_1, F_2]_{\lambda,\mu}:=F_1\rhd_{\lambda,\mu} F_2-F_2\rhd_{\lambda,\mu} F_1 \,\text{ for } \, F_1, F_2\in \hrts.
		\end{align*}
	\end{enumerate}
	
\end{theorem}

\begin{proof}
	By Theorems~\mref{thm:rt2} and~\mref{thm:preL}, $(\hrts, \rhd_{\lambda,\mu})$ is a pre-Lie algebra.
	The remainder follows from Lemma~\mref{lem:preL}~(\mref{lem:preL1}).
\end{proof}

A combinatorial descriptions of $\rhd_{\lambda,\mu}$ and $[_{-}, _{-}]_{\lambda,\mu}$ on $\hrts$ can also be given.
With the notations in Lemma~\mref{lem:comid1} and Theorem~\mref{thm:comb}, we have

\begin{coro}\label{coro:preLcomb}
	For any $F_1, F_2\in \rfs$,
	\begin{align}
		F_1\rhd_{\lambda,\mu} F_2&=-\lambda \sum_{k=0}^{n_{F_2}} (F_2)_{\mid I_k} F_1 (F_2)_{\mid J_k}
		+\mu \sum_{k=1}^{n_{F_2}} (F_2)_{\mid I_{k-1}} F_1 (F_2)_{\mid J_k},
		\mlabel{eq:precom}
	\end{align}
	and
	\begin{align}
		\label{eq:liecom}
		[F_1, F_2]_{\lambda,\mu}=&-\lambda \sum_{k=0}^{n_{F_2}} (F_2)_{\mid I_k} F_1 (F_2)_{\mid J_k}
		+\mu \sum_{k=1}^{n_{F_2}} (F_2)_{\mid I_{k-1}} F_1 (F_2)_{\mid J_k}\\
		&+\lambda \sum_{k=0}^{n_{F_1}} (F_1)_{\mid I_k} F_2 (F_1)_{\mid J_k}
		-\mu \sum_{k=1}^{n_{F_1}} (F_1)_{\mid I_{k-1}} F_2 (F_1)_{\mid J_k}.
	\end{align}
\end{coro}
\begin{proof}
	It follows directly from Theorems~\mref{thm:comb} and~\mref{thm:preope}.
\end{proof}

\begin{exam}\label{exam:preL}
	Let $F_1=\tdun{$x$}$,  $F_2=\tddeux{$\alpha$}{$\beta$}$,  $F_3=\tdun{$\gamma$}\tdun{$y$}$.
	with $\alpha, \beta,  \gamma \in \Omega$ and $x, y\in X$.
	Then $F_1$ has two forest biideals $\emptyset$ , $\{\bullet_x\}$, $F_2$ has three forests biideals $\emptyset$, $\{\bullet_\beta\}$, $\{\bullet_\beta, \bullet_\alpha\}$, and
	$F_3$ has three forests biideals $\emptyset$, $\{\bullet_\gamma\}$ , $\{\bullet_\gamma, \bullet_y\}$.
	By Eqs.~(\mref{eq:precom}) and~(\mref{eq:liecom}),
	\begin{align*}
		F_1\rhd_{\lambda,\mu} F_2=&\ \mu(F_{2|\emptyset}F_1 F_{2|\{\bullet_\alpha\}}+F_{2|\{\bullet_\beta\}}F_1 F_{2|\emptyset})\\
		&-\lambda (F_{2|\emptyset}F_1 F_{2|\{\bullet_\beta,\bullet_\alpha\}}+F_{2|\{\bullet_\beta\}}F_1 F_{2|\{\bullet_\alpha\}}+F_{2|\{\bullet_\beta,\bullet_\alpha\}}F_1 F_{2|\emptyset})\\
		=&\ \mu(\etree F_1 \tdun{$\alpha$}+\tdun{$\beta$} F_1 \etree)-\lambda (\etree F_1 \tddeux{$\alpha$}{$\beta$} +\tdun{$\beta$}F_1 \tdun{$\alpha$}+\tddeux{$\alpha$}{$\beta$} F_1 \etree)\\
		=&\ \mu(\tdun{$x$} \tdun{$\alpha$}+\tdun{$\beta$}\tdun{$x$})-\lambda (\tdun{$x$} \tddeux{$\alpha$}{$\beta$} +\tdun{$\beta$}\tdun{$x$} \tdun{$\alpha$}+\tddeux{$\alpha$}{$\beta$} \tdun{$x$} ),\\
		F_2\rhd_{\lambda,\mu} F_1 =&\ \mu(\tddeux{$\alpha$}{$\beta$})-\lambda (\tdun{$x$} \tddeux{$\alpha$}{$\beta$}+\tddeux{$\alpha$}{$\beta$} \tdun{$x$} ),\\
		[F_1, F_2]_{\lambda,\mu}=&\ \mu(\tdun{$x$} \tdun{$\alpha$}+\tdun{$\beta$}\tdun{$x$}-\tddeux{$\alpha$}{$\beta$})-\lambda(\tdun{$\beta$}\tdun{$x$} \tdun{$\alpha$}),\\
		F_2\rhd_{\lambda,\mu} F_3=&\ \mu(\etree F_2 \tdun{$y$}+\tdun{$\gamma$} F_2 \etree)-\lambda (\etree F_2 \tdun{$\gamma$}\tdun{$y$}+\tdun{$\gamma$}F_2\tdun{$y$}+\tdun{$\gamma$}\tdun{$y$}F_2\etree) \\
		=&\ \mu(\tddeux{$\alpha$}{$\beta$}\tdun{$y$}+\tdun{$\gamma$}\tddeux{$\alpha$}{$\beta$})-\lambda (\tddeux{$\alpha$}{$\beta$} \tdun{$\gamma$}\tdun{$y$}+\tdun{$\gamma$}\tddeux{$\alpha$}{$\beta$}\tdun{$y$}
		+\tdun{$\gamma$}\tdun{$y$}\tddeux{$\alpha$}{$\beta$}).
	\end{align*}
	Moreover, using Theorem~\mref{thm:comb}, we have
	\begin{align*}
		\col(F_3)=&~\col(\tdun{$\gamma$}\tdun{$y$})=\mu(\etree \ot \tdun{$y$} + \tdun{$\gamma$} \ot \etree)-\lambda (\etree \ot \tdun{$\gamma$}\tdun{$y$}+\tdun{$\gamma$}\ot\tdun{$y$}+\tdun{$\gamma$}\tdun{$y$}\ot \etree),\\
		\col(F_2\rhd_{\lambda,\mu} F_3)=&~\col(\mu(\tddeux{$\alpha$}{$\beta$}\tdun{$y$}+\tdun{$\gamma$}\tddeux{$\alpha$}{$\beta$})-\lambda (\tddeux{$\alpha$}{$\beta$} \tdun{$\gamma$}\tdun{$y$}+\tdun{$\gamma$}\tddeux{$\alpha$}{$\beta$}\tdun{$y$}
		+\tdun{$\gamma$}\tdun{$y$}\tddeux{$\alpha$}{$\beta$}))\\
		=&~\mu^2(\etree \ot \tdun{$\alpha$}\tdun{$y$}+\tdun{$\beta$}\ot \tdun{$y$}+\tddeux{$\alpha$}{$\beta$} \ot \etree+\etree \ot \tddeux{$\alpha$}{$\beta$}+\tdun{$\gamma$}\ot \tdun{$\alpha$}+\tdun{$\gamma$}\tdun{$\beta$} \ot \etree)\\&
		-\mu\lambda(\etree \ot \tddeux{$\alpha$}{$\beta$}\tdun{$y$}+\tdun{$\beta$} \ot\tdun{$\alpha$}\tdun{$y$}+ \tddeux{$\alpha$}{$\beta$} \ot \tdun{$y$}+\tddeux{$\alpha$}{$\beta$}\tdun{$y$}\ot \etree+\etree \ot \tdun{$\gamma$}\tddeux{$\alpha$}{$\beta$}\\
		&+\tdun{$\gamma$} \ot \tddeux{$\alpha$}{$\beta$}+\tdun{$\gamma$}\tdun{$\beta$} \ot \tdun{$\alpha$}+\tdun{$\gamma$}\tddeux{$\alpha$}{$\beta$} \ot \etree+\etree \ot  \tdun{$\alpha$}\tdun{$\gamma$}\tdun{$y$}+\tdun{$\beta$} \ot \tdun{$\gamma$}\tdun{$y$} \\
		&+ \tddeux{$\alpha$}{$\beta$} \ot \tdun{$y$}+\tddeux{$\alpha$}{$\beta$}\tdun{$\gamma$}\ot \etree+\etree \ot \tddeux{$\alpha$}{$\beta$}\tdun{$y$}+\tdun{$\gamma$} \ot\tdun{$\alpha$}\tdun{$y$}+\tdun{$\gamma$}\tdun{$\beta$} \ot \tdun{$y$}\\
		&+\tdun{$\gamma$}\tddeux{$\alpha$}{$\beta$}\ot \etree+\etree \ot \tdun{$y$}\tddeux{$\alpha$}{$\beta$}+\tdun{$\gamma$} \ot \tddeux{$\alpha$}{$\beta$}+\tdun{$\gamma$}\tdun{$y$} \ot \tdun{$\alpha$}+\tdun{$\gamma$}\tdun{$y$}\tdun{$\beta$} \ot \etree)\\
		&+\lambda^2(\etree\ot \tddeux{$\alpha$}{$\beta$}\tdun{$\gamma$}\tdun{$y$}
		+\tdun{$\beta$}\ot \tdun{$\alpha$}\tdun{$\gamma$}\tdun{$y$}
		+\tddeux{$\alpha$}{$\beta$}\ot \tdun{$\gamma$}\tdun{$y$}
		+\tddeux{$\alpha$}{$\beta$}\tdun{$\gamma$}\ot \tdun{$y$}
		+\tddeux{$\alpha$}{$\beta$}\tdun{$\gamma$} \tdun{$y$}\ot \etree\\
		&+\etree \ot \tdun{$\gamma$}\tddeux{$\alpha$}{$\beta$}\tdun{$y$}
		+ \tdun{$\gamma$}\ot \tddeux{$\alpha$}{$\beta$}\tdun{$y$}
		+\tdun{$\gamma$}\tdun{$\beta$}\ot \tdun{$\alpha$}\tdun{$y$}
		+\tdun{$\gamma$}\tddeux{$\alpha$}{$\beta$}\ot \tdun{$y$}
		+\tdun{$\gamma$}\tddeux{$\alpha$}{$\beta$}\tdun{$y$}\ot \etree\\
		&+\etree \ot \tdun{$\gamma$} \tdun{$y$}\tddeux{$\alpha$}{$\beta$}
		+ \tdun{$\gamma$}\ot \tdun{$y$}\tddeux{$\alpha$}{$\beta$}
		+ \tdun{$\gamma$} \tdun{$y$}\ot \tddeux{$\alpha$}{$\beta$}
		+\tdun{$\gamma$}\tdun{$y$}\tdun{$\beta$}\ot \tdun{$\alpha$}
		+\tdun{$\gamma$} \tdun{$y$}\tddeux{$\alpha$}{$\beta$}\ot \etree).
	\end{align*}
	Applying Theorem~\mref{thm:preope}, we obtain
	\begin{align*}
		(F_1\rhd_{\lambda,\mu} F_2)\rhd_{\lambda,\mu} F_3=& \Big(\mu(\tdun{$x$} \tdun{$\alpha$}+\tdun{$\beta$}\tdun{$x$})-\lambda (\tdun{$x$} \tddeux{$\alpha$}{$\beta$} +\tdun{$\beta$}\tdun{$x$} \tdun{$\alpha$}+\tddeux{$\alpha$}{$\beta$} \tdun{$x$} )\Big)\rhd_{\lambda,\mu} F_3\\
		=&\mu^2(\tdun{$x$}\tdun{$\alpha$}\tdun{$y$}+\tdun{$\beta$}\tdun{$x$}\tdun{$y$}+\tdun{$\gamma$}\tdun{$x$}\tdun{$\alpha$}+\tdun{$\gamma$}\tdun{$\beta$}\tdun{$x$})\\
		&-\mu\lambda(\tdun{$x$}\tddeux{$\alpha$}{$\beta$}\tdun{$y$}+\tddeux{$\alpha$}{$\beta$}\tdun{$x$}\tdun{$y$}+\tdun{$\gamma$}\tddeux{$\alpha$}{$\beta$}\tdun{$x$}+\tdun{$\gamma$}\tdun{$x$}\tddeux{$\alpha$}{$\beta$}\\
		&+\tdun{$\beta$}\tdun{$x$}\tdun{$\alpha$}\tdun{$y$}+\tdun{$\gamma$}\tdun{$\beta$}\tdun{$x$}\tdun{$\alpha$}+\tdun{$x$}\tdun{$\alpha$}\tdun{$\gamma$}\tdun{$y$}+\tdun{$\beta$}\tdun{$x$}\tdun{$\gamma$}\tdun{$y$}\\
		&+\tdun{$\gamma$}\tdun{$x$}\tdun{$\alpha$}\tdun{$y$}+\tdun{$\gamma$}\tdun{$\beta$}\tdun{$x$}\tdun{$y$}+\tdun{$\gamma$}\tdun{$y$}\tdun{$x$}\tdun{$\alpha$}+\tdun{$\gamma$}\tdun{$y$}\tdun{$\beta$}\tdun{$x$})\\
		&+\lambda^2 (\tdun{$x$} \tddeux{$\alpha$}{$\beta$}\tdun{$\gamma$}\tdun{$y$} +\tdun{$\beta$}\tdun{$x$} \tdun{$\alpha$}\tdun{$\gamma$}\tdun{$y$}+\tddeux{$\alpha$}{$\beta$} \tdun{$x$}  \tdun{$\gamma$}\tdun{$y$}
		+\tdun{$\gamma$}\tdun{$x$} \tddeux{$\alpha$}{$\beta$}\tdun{$y$} +\tdun{$\gamma$}\tdun{$\beta$}\tdun{$x$} \tdun{$\alpha$}\tdun{$y$}\\
		&+\tdun{$\gamma$}\tddeux{$\alpha$}{$\beta$} \tdun{$x$} \tdun{$y$}+\tdun{$\gamma$}\tdun{$y$}\tdun{$x$} \tddeux{$\alpha$}{$\beta$} +\tdun{$\gamma$}\tdun{$y$}\tdun{$\beta$}\tdun{$x$} \tdun{$\alpha$}
		+\tdun{$\gamma$}\tdun{$y$}\tddeux{$\alpha$}{$\beta$} \tdun{$x$}) ,\\
		F_1\rhd_{\lambda,\mu} (F_2\rhd_{\lambda,\mu} F_3)=&\mu^2(\tdun{$x$} \tdun{$\alpha$}\tdun{$y$}+\tdun{$\beta$}\tdun{$x$} \tdun{$y$}+\tddeux{$\alpha$}{$\beta$}\tdun{$x$}+\tdun{$x$} \tddeux{$\alpha$}{$\beta$}+\tdun{$\gamma$}\tdun{$x$} \tdun{$\alpha$}+\tdun{$\gamma$}\tdun{$\beta$}\tdun{$x$}\\&
		-\mu\lambda(\tdun{$x$} \tddeux{$\alpha$}{$\beta$}\tdun{$y$}+\tdun{$\beta$} \tdun{$x$}\tdun{$\alpha$}\tdun{$y$}+ \tddeux{$\alpha$}{$\beta$} \tdun{$x$} \tdun{$y$}+\tddeux{$\alpha$}{$\beta$}\tdun{$y$}\tdun{$x$}+\tdun{$x$} \tdun{$\gamma$}\tddeux{$\alpha$}{$\beta$}\\
		&+\tdun{$\gamma$} \tdun{$x$} \tddeux{$\alpha$}{$\beta$}+\tdun{$\gamma$}\tdun{$\beta$} \tdun{$x$} \tdun{$\alpha$}+\tdun{$\gamma$}\tddeux{$\alpha$}{$\beta$} \tdun{$x$} +\tdun{$x$}  \tdun{$\alpha$}\tdun{$\gamma$}\tdun{$y$}+\tdun{$\beta$} \tdun{$x$} \tdun{$\gamma$}\tdun{$y$} \\
		&+ \tddeux{$\alpha$}{$\beta$}\tdun{$x$} \tdun{$y$}+\tddeux{$\alpha$}{$\beta$}\tdun{$\gamma$}\tdun{$x$}+\tdun{$x$} \tddeux{$\alpha$}{$\beta$}\tdun{$y$}+\tdun{$\gamma$} \tdun{$x$}\tdun{$\alpha$}\tdun{$y$}+\tdun{$\gamma$}\tdun{$\beta$} \tdun{$x$} \tdun{$y$}\\
		&+\tdun{$\gamma$}\tddeux{$\alpha$}{$\beta$}\tdun{$x$}+\tdun{$x$} \tdun{$y$}\tddeux{$\alpha$}{$\beta$}+\tdun{$\gamma$} \tdun{$x$} \tddeux{$\alpha$}{$\beta$}+\tdun{$\gamma$}\tdun{$y$} \tdun{$x$} \tdun{$\alpha$}+\tdun{$\gamma$}\tdun{$y$}\tdun{$\beta$} \tdun{$x$})\\
		&+\lambda^2(\tdun{$x$} \tddeux{$\alpha$}{$\beta$}\tdun{$\gamma$}\tdun{$y$}
		+\tdun{$\beta$}\tdun{$x$} \tdun{$\alpha$}\tdun{$\gamma$}\tdun{$y$}
		+\tddeux{$\alpha$}{$\beta$}\tdun{$x$} \tdun{$\gamma$}\tdun{$y$}
		+\tddeux{$\alpha$}{$\beta$}\tdun{$\gamma$}\tdun{$x$} \tdun{$y$}
		+\tddeux{$\alpha$}{$\beta$}\tdun{$\gamma$} \tdun{$y$}\tdun{$x$} \\
		&+\tdun{$x$} \tdun{$\gamma$}\tddeux{$\alpha$}{$\beta$}\tdun{$y$}
		+ \tdun{$\gamma$}\tdun{$x$} \tddeux{$\alpha$}{$\beta$}\tdun{$y$}
		+\tdun{$\gamma$}\tdun{$\beta$}\tdun{$x$} \tdun{$\alpha$}\tdun{$y$}
		+\tdun{$\gamma$}\tddeux{$\alpha$}{$\beta$}\tdun{$x$} \tdun{$y$}
		+\tdun{$\gamma$}\tddeux{$\alpha$}{$\beta$}\tdun{$y$}\tdun{$x$} \\
		&+\tdun{$x$} \tdun{$\gamma$} \tdun{$y$}\tddeux{$\alpha$}{$\beta$}
		+ \tdun{$\gamma$}\tdun{$x$} \tdun{$y$}\tddeux{$\alpha$}{$\beta$}
		+ \tdun{$\gamma$} \tdun{$y$}\tdun{$x$} \tddeux{$\alpha$}{$\beta$}
		+\tdun{$\gamma$}\tdun{$y$}\tdun{$\beta$}\tdun{$x$} \tdun{$\alpha$}
		+\tdun{$\gamma$} \tdun{$y$}\tddeux{$\alpha$}{$\beta$}\tdun{$x$}).
	\end{align*}
	Thus
	\begin{align*}
		F_1\rhd_{\lambda,\mu} (F_2\rhd_{\lambda,\mu} F_3)-(F_1\rhd_{\lambda,\mu} F_2)\rhd_{\lambda,\mu} F_3
		=&~\mu^2(\tddeux{$\alpha$}{$\beta$}\tdun{$x$}+\tdun{$x$}\tddeux{$\alpha$}{$\beta$})-\mu\lambda\Big((\tddeux{$\alpha$}{$\beta$}\tdun{$y$}\tdun{$x$}+\tdun{$x$}\tdun{$y$}\tddeux{$\alpha$}{$\beta$})\\
		&+(\tddeux{$\alpha$}{$\beta$}\tdun{$\gamma$}\tdun{$x$}+\tdun{$x$}\tdun{$\gamma$}\tddeux{$\alpha$}{$\beta$})+(\tddeux{$\alpha$}{$\beta$}\tdun{$x$}\tdun{$y$}+\tdun{x}\tddeux{$\alpha$}{$\beta$}\tdun{$y$})\\
		&+(\tdun{$\gamma$}\tddeux{$\alpha$}{$\beta$}\tdun{$x$}+\tdun{$\gamma$}\tdun{$x$}\tddeux{$\alpha$}{$\beta$})\Big)+\lambda^2\Big((\tddeux{$\alpha$}{$\beta$}\tdun{$\gamma$}\tdun{$x$} \tdun{$y$}+\tdun{$x$} \tdun{$\gamma$}\tddeux{$\alpha$}{$\beta$}\tdun{$y$})\\
		&+(\tddeux{$\alpha$}{$\beta$}\tdun{$\gamma$} \tdun{$y$}\tdun{$x$} +\tdun{$x$} \tdun{$\gamma$} \tdun{$y$}\tddeux{$\alpha$}{$\beta$})
		+(\tdun{$\gamma$}\tddeux{$\alpha$}{$\beta$}\tdun{$y$}\tdun{$x$} +\tdun{$\gamma$}\tdun{$x$} \tdun{$y$}\tddeux{$\alpha$}{$\beta$})\Big),
	\end{align*}
	which is symmetric in $F_1=\tdun{$x$}$ and $F_2=\tddeux{$\alpha$}{$\beta$}$ and hence
	\begin{align*}
		(F_1\rhd_{\lambda,\mu} F_2)\rhd_{\lambda,\mu} F_3-F_1\rhd_{\lambda,\mu} (F_2\rhd_{\lambda,\mu} F_3)=(F_2\rhd_{\lambda,\mu} F_1)\rhd_{\lambda,\mu} F_3-F_2\rhd_{\lambda,\mu} (F_1\rhd_{\lambda,\mu} F_3).
	\end{align*}
\end{exam}

\begin{remark}
	Let us emphasize that our pre-Lie algebra $(\hrts, \rhd_{\lambda,\mu})$ is different from the one
	introduced by Chapoton and Livernet~\mcite{CL01} from the viewpoint of operad.
	Under the framework of~\mcite{CL01}, Matt gave an example:
	\begin{align*}
		\tdun{$\beta$}\rhd \tddeux{$\alpha$}{$\beta$}=2\ \tdtroisun{$\alpha$}{$\beta$}{$\beta$}+\tdtroisdeux{$\alpha$}{$\beta$}{$\beta$},
	\end{align*}
	which is different from $\tdun{$\beta$}\rhd_{\lambda,\mu} \tddeux{$\alpha$}{$\beta$}= -\lambda (\tdun{$\beta$} \tddeux{$\alpha$}{$\beta$} +\tdun{$\beta$}\tdun{$\beta$} \tdun{$\alpha$}+\tddeux{$\alpha$}{$\beta$} \tdun{$\beta$} )$ if we replace $F_1$ by $\tdun{$\beta$}$ in Example~\mref{exam:preL}.
\end{remark}

\noindent
{\bf Acknowledgments.} This work is supported by the National Natural Science Foundation of China (12101316).

\noindent
{\bf Declaration of interests.} The authors have no conflicts of interest to disclose.

\noindent
{\bf Data availability.} Data sharing is not applicable as no new data were created or analyzed.

\end{document}

\textcolor{blue}{To keep for later}

\textcolor{blue}{I suggest to cut the part on the counit, as a more general result is obtained later.}

-----------------------------------------------------------------

When $\lambda=-1$, an $\epsilon$-unitary bialgebra of weight -1 with a counit can be constructed on $\hrts$.  Define a counit $\dpl$ : $\hrts$ $\to$ $\mathbf{k}$.For the initial step of dep$(F)=0$, we define
\begin{equation}
\varepsilon_{RT}(F) :=
\begin{cases}
	1_k & \text{if } F = 1, \\
	-\mu 1_k & \text{if } F = \bullet_x \text{ for some } x \in X, \\
	\varepsilon_{RT}(\bullet_{x_1}) \cdots \varepsilon_{RT}(\bullet_{x_m}) & \text{if } F = \bullet_{x_1} \cdots \bullet_{x_m} \text{ with } m \geq 2.	
\end{cases}
\mlabel{eq:ccounit}
\end{equation}
For the induction step of $\mathrm{dep}(F) \geq 1$, then $F \neq 1$ and $\mathrm{bre}(F) \geq 1$. If $\mathrm{bre}(F) = 1$, we have $F = B_\omega^{+}(\overline{F})$, for some $\overline{F} \in \rfs$ and $\omega \in \Omega$, and define
\begin{equation}
\varepsilon_{RT}(F) = \varepsilon_{RT}(B_{\omega}^{+}(\overline{F})) = -\mu \varepsilon_{RT}(\overline{F}).
\mlabel{eq:counitt}
\end{equation}
If $\mathrm{bre}(F) \geq 2$, we write $F = T_1 \cdots T_m$ with $m \geq 2$ and $T_1, \cdots, T_m \in \rfs$, and define
\begin{equation}
\varepsilon_{RT}(F) = \varepsilon_{RT}(T_1) \cdots \varepsilon_{RT}(T_m),
\mlabel{eq:couunit}
\end{equation}
where $\varepsilon_{RT}(T_1), \cdots, \varepsilon_{RT}(T_m)$ are defined in Eqs.\eqref{eq:ccounit} and \eqref{eq:counitt}.
\begin{lemma}\label{lem:cohom}
The linear map $\varepsilon_{RT} : H_{RT}(X, \Omega) \to $ $\mathbf{k}$ is an algebra homomorphism.
\end{lemma}
\begin{proof}
By Eq.\eqref{eq:ccounit}, we have $\varepsilon_{RT}(1) = 1_k$. We needs to show that $\varepsilon_{RT}(FF') = \varepsilon_{RT}(F)\varepsilon_{RT}(F')$, for $F, F' \in \rfs$. If $F = 1$ or $F' = 1$, without loss of generality, we take $F' = 1$. Then $\varepsilon_{RT}(F') = 1_k$, and so
$$
\varepsilon_{RT}(FF') = \varepsilon_{RT}(F) = \varepsilon_{RT}(F)1_k = \varepsilon_{RT}(F)\varepsilon_{RT}(F').
$$
If $F \neq 1$ and $F' \neq 1$, we write $F = T_1 \cdots T_m$ and $F' = T'_1 \cdots T'_{m'}$, where $m, m' \geq 1$ and $T_i, T_j \in \mathcal{F}_{RT}(X, \Omega)$ for $1 \leq i \leq m$ and $1 \leq j \leq m'$. By Eq.\eqref{eq:couunit}, we have

$$
\begin{aligned}
	\varepsilon_{RT}(FF') &= \varepsilon_{RT}(T_1 \cdots T_m T'_1 \cdots T'_{m'}) \\
	&= \varepsilon_{RT}(T_1) \cdots \varepsilon_{RT}(T_m) \varepsilon_{RT}(T'_1) \cdots \varepsilon_{RT}(T'_{m'}) \\
	&= \varepsilon_{RT}(T_1 \cdots T_m) \varepsilon_{RT}(T'_1 \cdots T'_{m'}) \\
	&= \varepsilon_{RT}(F) \varepsilon_{RT}(F').
\end{aligned}
$$
as required.
\end{proof}
\begin{theorem}
The quintuple $(\hrts,\, \conc, \etree, \,\col,\dpl)$ is an  $\epsilon$-unitary bialgebra of weight $-1$ (with counit).
\end{theorem}

\begin{proof}
By Theorem~\mref{thm:rt1},we only need to verify the counicity for $F \in \hrts $. We apply the induction on $\dep(F) \geq 0$ to check the counicity conditions:
\begin{align}
	(\dpl \ot \id)\col(F)=\beta_l(F) \quad\text{and}\quad
	(\id \ot \dpl)\col(F)=\beta_r(F),
\end{align}
where $\beta_{l}: \hrts \to \bfk \ot \hrts$ is defined by $F \mapsto 1_k \ot F$ and $\beta_{r}: \hrts \to \hrts \ot \bfk$ is defined by $F \mapsto F \ot 1_k$.
\\ \hspace*{1em} For the initial step of $\dep(F)=0$,we have $F=\bullet_{x_{1}}\bullet_{x_{2}}\cdots\bullet_{x_{m}}$ for some $m \geq 0$,with the convention that $F=1$.When $m=0$,we have
\begin{align*}
	&(\dpl \ot \id)\col(F)=(\dpl \ot \id)\col(1)=\dpl(1) \ot 1=1_k \ot 1=\beta_l(1), \\
	&(\id \ot \dpl)\col(F)=(\id \ot \dpl)\col(1)=1 \ot \dpl(1)=1 \ot 1_k=\beta_r(1).
\end{align*}
When $m \geq 1$,by Lemma~\mref{lem:rt11}
\begin{align*}
	&(\dpl \ot \id)\col(F) \\
	=&\ (\dpl \ot \id)\left(\mu \sum_{i=1}^{m}\bullet_{x_{1}}\cdots\bullet_{x_{i-1}}\otimes\bullet_{x_{i+1}}\cdots\bullet_{x_{m}}
	+
	\sum_{i=0}^{m}\bullet_{x_{1}}\cdots\bullet_{x_{i}}
	\otimes\bullet_{x_{i+1}}\cdots\bullet_{x_{m}}\right)\\
	=&\ \mu \sum_{i=1}^{m}\dpl(\bullet_{x_{1}}\cdots\bullet_{x_{i-1}})\otimes\bullet_{x_{i+1}}\cdots\bullet_{x_{m}}
	+
	\sum_{i=0}^{m}\dpl(\bullet_{x_{1}}\cdots\bullet_{x_{i}})
	\otimes\bullet_{x_{i+1}}\cdots\bullet_{x_{m}}\\
	=&\ \mu \sum_{i=1}^{m}\dpl(\bullet_{x_{1}})\cdots\dpl(\bullet_{x_{i-1}})\otimes\bullet_{x_{i+1}}\cdots\bullet_{x_{m}}
	+
	\sum_{i=0}^{m}\dpl(\bullet_{x_{1}})\cdots\dpl(\bullet_{x_{i}})
	\otimes\bullet_{x_{i+1}}\cdots\bullet_{x_{m}}\\
	&\hspace{9cm}(\text{by Eq.~(\mref{eq:ccounit})})\\
	=&\ -\sum_{i=1}^{m}(-\mu)^{i}1_k\otimes\bullet_{x_{i+1}}\cdots\bullet_{x_{m}}
	+
	\sum_{i=1}^{m}(-\mu)^{i}1_k
	\otimes\bullet_{x_{i+1}}\cdots\bullet_{x_{m}}+1_k\ot\bullet_{x_{1}}\cdots\bullet_{x_{m}}\\
	=&\ 1_k\ot F=\beta_l(F).
\end{align*}
On the other hand,
\begin{align*}
	&(\id \ot \dpl)\col(F) \\
	=&\ (\id \ot \dpl)\left(\mu \sum_{i=1}^{m}\bullet_{x_{1}}\cdots\bullet_{x_{i-1}}\otimes\bullet_{x_{i+1}}\cdots\bullet_{x_{m}}
	+
	\sum_{i=0}^{m}\bullet_{x_{1}}\cdots\bullet_{x_{i}}
	\otimes\bullet_{x_{i+1}}\cdots\bullet_{x_{m}}\right)\\
	=&\ \mu \sum_{i=1}^{m}\bullet_{x_{1}}\cdots\bullet_{x_{i-1}}\otimes\dpl(\bullet_{x_{i+1}}\cdots\bullet_{x_{m}})
	+
	\sum_{i=0}^{m}\bullet_{x_{1}}\cdots\bullet_{x_{i}}
	\otimes\dpl(\bullet_{x_{i+1}}\cdots\bullet_{x_{m}})\\
	=&\ \mu \sum_{i=1}^{m}\bullet_{x_{1}}\cdots\bullet_{x_{i-1}}\otimes\dpl(\bullet_{x_{i+1}})\cdots\dpl(\bullet_{x_{m}})
	+
	\sum_{i=0}^{m}\bullet_{x_{1}}\cdots\bullet_{x_{i}}\otimes\dpl(\bullet_{x_{i+1}})\cdots\dpl(\bullet_{x_{m}})\\
	&\hspace{9cm}(\text{by Eq.~(\mref{eq:ccounit})})\\
	=&\ -\sum_{i=1}^{m}\bullet_{x_{1}}\cdots\bullet_{x_{i-1}}\ot(-\mu)^{m-i+1}1_k
	+
	\sum_{i=0}^{m}\bullet_{x_{1}}\cdots\bullet_{x_{i}}
	\otimes(-\mu)^{m-i}1_k\\
	=&\ -\sum_{i=1}^{m}\bullet_{x_{1}}\cdots\bullet_{x_{i-1}}\ot(-\mu)^{m-i+1}1_k+\sum_{i=1}^{m}\bullet_{x_{1}}\cdots\bullet_{x_{i}}
	\otimes(-\mu)^{m-i+1}1_k+\bullet_{x_{1}}\cdots\bullet_{x_{m}}\ot 1_k\\
	=&\ F\ot 1_k=\beta_r(F).
\end{align*}
\hspace*{1em} Suppose that Eq.~(\mref{eq:couunit}) holds for \(\text{dep}(F) \leq n\) for an \(n \geq 0\) and consider the case of \(\text{dep}(F) = n+1\). We next apply the induction on breadth. Since \(\text{dep}(F) = n+1 \geq 1\), we have \(F \neq 1\) and \(\text{bre}(F) \geq 1\). When \(\text{bre}(F) = 1\), we may write \(F = B^+_\omega(\overline{F})\) for some \(\overline{F} \in \mathcal{F}(X, \Omega)\) and \(\omega \in \Omega\). Then
\begin{align*}
	(\dpl \ot \id)\col(F)
	=& (\dpl \ot \id)(F \ot 1+\mu\overline{F} \ot 1+(\id \ot B^+_\omega)\col(\overline{F}) )\quad(\text{by Eq.~(\mref{eq:cdbp})}) \\
	=&\ \dpl(B^+_\omega(F)) \ot 1+\mu\dpl(\overline{F}) \ot 1+(\id \ot B^+_\omega)(\dpl \ot \id)\col(\overline{F}) \\
	=&\ -\mu\dpl(\overline{F}) \ot 1+\mu\dpl(\overline{F}) \ot 1+(\id \ot B^+_\omega)(\dpl \ot \id)\col(\overline{F})\quad(\text{by Eq.~(\mref{eq:counitt})}) \\
	=&\ (\id \ot B^+_\omega)(1_k \ot \overline{F})\quad(\text{by the induction hypothesis})\\
	=&\ 1_k \ot F=\beta_l(F).
\end{align*}
On the other hand,
\begin{align*}
	(\id \ot \dpl)\col(F)
	=& (\id \ot \dpl)(F \ot \id+\mu\overline{F} \ot 1+(\id \ot B^+_\omega)\col(\overline{F}) )\quad(\text{by Eq.~(\mref{eq:cdbp})}) \\
	=&\ F \ot 1_k+\mu\overline{F} \ot 1_k+(\id \ot \dpl B^+_\omega)\col(\overline{F}) \quad(\text{by Eq.~(\mref{eq:ccounit})}) \\
	=&\ F \ot 1_k+\mu\overline{F} \ot 1_k-\mu(\id \ot \dpl)\col(\overline{F}) \quad(\text{by Eq.~(\mref{eq:counitt})}) \\
	=&\ F \ot 1_k+\mu\overline{F} \ot 1_k-\mu\overline{F} \ot 1_k\quad(\text{by the induction hypothesis})\\
	=&\  F \ot 1_k=\beta_r(F).
\end{align*}
\hspace*{1em} Assume that Eq.~(\mref{eq:couunit}) holds for \(\text{dep}(F) = n + 1\) and \(\text{bre}(F) \leq m\), in addition to \(\text{dep}(F) \leq n\) by the first induction hypothesis. Consider the case when \(\text{dep}(F) = n + 1\) and \(\text{bre}(F) = m + 1 \geq 2\). As in the proof of the coassociativity, let \(F = F_1F_2\) for some \(F_1, F_2 \in \mathcal{F}(X, \Omega)\) with \(0 < \text{bre}(F_1), \text{bre}(F_2) < \text{bre}(F)\). Thus
\begin{align*}
	(\dpl \ot \id)\col(F)=& (\dpl \ot \id)\col(F_{1}F_{2})\\
	=&\ (\varepsilon_{\text{RT}} \otimes \text{id})(F_1 \cdot \Delta_\epsilon(F_2) + \Delta_\epsilon(F_1) \cdot F_2 - F_1 \otimes F_2) \quad (\text{by Lemma~\mref {lem:colff}}) \\
	=&\ (\varepsilon_{\text{RT}} \otimes \text{id}) \left( \sum_{(F_2)} F_1 F_{2(1)} \otimes F_{2(2)} + \sum_{(F_1)} F_{1(1)} \otimes F_{1(2)} F_2 - F_1 \otimes F_2 \right)\quad(\text{by Eq.~(\mref{eq:dota})})\\
	=&\ \sum_{(F_2)} \varepsilon_{\text{RT}}(F_1 F_{2(1)}) \otimes F_{2(2)} + \sum_{(F_1)} \varepsilon_{\text{RT}}(F_{1(1)}) \otimes F_{1(2)} F_2 - \varepsilon_{\text{RT}}(F_1) \otimes F_2\\
	=&\ \sum_{(F_2)} \varepsilon_{\text{RT}}(F_1)\dpl( F_{2(1)}) \otimes F_{2(2)} + \sum_{(F_1)} \varepsilon_{\text{RT}}(F_{1(1)}) \otimes F_{1(2)} F_2 - \varepsilon_{\text{RT}}(F_1) \otimes F_2 \\
	&\hspace{9cm}(\text{by Lemma~\mref {lem:cohom}}) \\
	=&\ \dpl(F_1)[(\dpl \ot \id)\col(F_2))]+[(\dpl \ot \id)\col(F_1)]F_2-\varepsilon_{\text{RT}}(F_1) \otimes F_2\\
	=&\ \dpl(F_1)(1_k \ot F_2)+(1_k \ot F_1)F_2-\dpl(F_1)(1_k \ot F_2)\\
	&\hspace{7cm}(\text{by the induction hypothesis}) \\
	=&\ 1_k \ot F_1 F_2=1_k \ot F=\beta_{l}(F).
\end{align*}
On the other hand,
\begin{align*}
	(\id \ot \dpl)\col(F)=& (\id \ot \dpl)\col(F_{1}F_{2})\\
	=&\ (\id \otimes \dpl)(F_1 \cdot \Delta_\epsilon(F_2) + \Delta_\epsilon(F_1) \cdot F_2 - F_1 \otimes F_2) \quad (\text{by Lemma~\mref {lem:colff}}) \\
	=&\ (\id \otimes \dpl) \left( \sum_{(F_2)} F_1 F_{2(1)} \otimes F_{2(2)} + \sum_{(F_1)} F_{1(1)} \otimes F_{1(2)} F_2 - F_1 \otimes F_2 \right)\quad(\text{by Eq.~(\mref{eq:dota})})\\
	=&\ \sum_{(F_2)} F_1 F_{2(1)} \otimes \dpl(F_{2(2)}) + \sum_{(F_1)} F_{1(1)} \otimes \dpl(F_{1(2)} F_2) - F_1 \otimes \dpl(F_2)\\
	=&\ \sum_{(F_2)} F_1 F_{2(1)} \otimes \dpl(F_{2(2)}) + \sum_{(F_1)} F_{1(1)} \otimes \dpl(F_{1(2)})\dpl( F_2) - F_1 \otimes \dpl(F_2)\\
	&\hspace{9cm}(\text{by Lemma~\mref {lem:cohom}}) \\
	=&\ F_1[(\id \ot \dpl)\col(F_2)]+[(\id \ot \dpl)\col(F_1)]\dpl(F_2) - F_1 \otimes \dpl(F_2)\\
	=&\ F_1(F_2 \ot 1_k)+(F_1 \ot 1_k)\dpl(F_2)- (F_1 \ot 1_k)\dpl(F_2)\\
	&\hspace{7cm}(\text{by the induction hypothesis}) \\
	=&\ F_1 F_2 \ot 1_k=F \ot 1_k=\beta_{r}(F).
\end{align*}
This completes the induction on the breadth and hence the induction on the depth.
\end{proof}
-----------------------------------------------------------------

-------------------------------------------------------
\section{An alternative construction}

\subsection{A general construction on coalgebras}

\begin{prop}
Let $(C,\Delta)$ be a coassociative coalgebra (not necessarily counitary) and let $f\in C^*$.
We put $\Delta_f=(\id_C\otimes f\otimes \id_C)\circ \Delta^{(2)}$. Then $\Delta_f$ is coassociative.
\end{prop}

\begin{proof}
Let $x\in C$. We use Sweedler's notation:
\begin{align*}
	\Delta(x)&=\sum x^{(1)}\otimes x^{(2)},&
	\Delta^{(2)}(x)&=(\Delta\otimes \id_C)\circ \Delta(x)=\sum x^{(1)}\otimes x^{(2)}\otimes x^{(3)},
\end{align*}
etc. Then
\[\Delta_f(x)=\sum f\left(x^{(2)}\right)x^{(1)}\otimes x^{(3)}.\]
Therefore,
\begin{align*}
	(\Delta_f \otimes \id_C)\circ \Delta_f(x)&=\sum\sum f\left(x^{(2)}\right)f\left(x^{(1)(2)}\right)x^{(1)(1)}\otimes x^{(1)(3)}\otimes x^{(3)}\\
	&=\sum f\left(x^{(2)}\right)f\left(x^{4}\right)x^{(1)}\otimes x^{(3)}\otimes x^{(5)}\\
	&=\sum\sum f\left(x^{(2)}\right)f\left(x^{(3)(2)}\right)x^{(1)}\otimes x^{(3)(1)}\otimes x^{(3)(3)}\\
	&=(\id_C\otimes \Delta_f)\circ \Delta_f(x).
\end{align*}
So $\Delta_f$ is coassociative.
\end{proof}

\begin{remark}
If $C$ is counitary, of counit $\varepsilon$, then $\Delta_\varepsilon=\Delta$.
\end{remark}

\begin{prop}\label{propcounite}
Let $(C,\Delta)$ be a coassociative coalgebra and $f\in C^*$. Then $\Delta_f$ is counitary if, and only if,
$\Delta$ is counitary and $f$ is invertible for the convolution product $*$ induced by $\Delta$ on $C^*$.
If so, the counit of $\Delta_f$ is the inverse $f^{*-1}$ of $f$ for this product $*$.
\end{prop}

\begin{proof}
$\Longrightarrow$. Let us denote by $g$ the counit of $\Delta_f$. Then
\[(g\otimes \id_C)\circ \Delta_f=(g\otimes f\otimes \id_C)\circ (\Delta \otimes \id_C)\circ \Delta
=(g*f \otimes \id_C)\circ \Delta=\id_C.\]	
So $g*f$ is a left counit for $\Delta$. Similarly, $f*g$ is a right counit for $\Delta$.
As $\Delta$ is coassociative, $f*g=g*f$ is a counit for $\Delta$, and $g=f^{*-1}$.\\

$\Longleftarrow$. Let us put $g=f^{*-1}$. Then, as computed before,
\[(g\otimes \id_C)\circ \Delta_f=(g*f \otimes \id_C)\circ \Delta=(\varepsilon_\Delta \otimes \id)\circ \Delta=\id_C.\]
Similarly, $(\id_C\otimes g)\circ \Delta=\id_C$, so $g$ is a counit for $\Delta_C$. 	
\end{proof}

\subsection{Application to $\epsilon$-bialgebras}

\begin{defn}
Let  $(A,m,\Delta)$ be a unitary and counitary $\epsilon$-bialgebra of weight $-1$. An infinitesimal character of $A$
is a map $f\in A^*$ such that
\begin{align*}
	&\forall x,y\in A,&f(xy)=f(x)\varepsilon_\Delta(y)+\varepsilon_\Delta(x)f(y).
\end{align*}
\end{defn}

\begin{prop}\label{prop2}
Let  $(A,m,\Delta)$ be a counitary $\epsilon$-bialgebra of weight $-1$ and $f\in A^*$.
Then $(A,m,\Delta_f)$ is an $\epsilon$-bialgebra of weight $0$ if, and only if, $f$ is an infinitesimal character of $A$.
\end{prop}

\begin{proof}
Firstly, for any $x,y\in A$,
\begin{align*}
	\Delta^{(2)}(xy)&=\sum x^{(1)}\otimes x^{(2)}\otimes x^{(3)}y+\sum\sum x^{(1)}\otimes x^{(2)}y^{(1)}\otimes y^{(2)}\\
	&+\sum xy^{(1)}\otimes y^{(2)}\otimes y^{(3)}-\sum x\otimes y^{(1)}\otimes y^{(2)}-\sum x^{(1)}\otimes x^{(2)}\otimes y,
\end{align*}
so
\begin{align*}
	\Delta_f(xy)&=\Delta_f(x)y+\sum\sum f\left(x^{(2)}y^{(1)}\right)x^{(1)}\otimes y^{(2)}\\
	&+x\Delta_f(y)-\sum f\left(y^{(1)}\right) x\otimes y^{(2)}-\sum f\left(x^{(2)}\right) x^{(1)}\otimes y.
\end{align*}
So $(A,m,\Delta_f)$ is infinitesimal of weight $0$ if, and only if, for any $x,y\in A$,
\[\sum\sum f\left(x^{(2)}y^{(1)}\right)x^{(1)}\otimes y^{(2)}=\sum f\left(y^{(1)}\right) x\otimes y^{(2)}+\sum f\left(x^{(2)}\right) x^{(1)}\otimes y.\]

$\Longrightarrow$. Let us apply $\varepsilon_\Delta\otimes \varepsilon_\Delta$ to this. This gives
\[f(xy)=\varepsilon_\Delta(x)f(y)+f(x)\varepsilon_\Delta(y).\]
So $f$ is an infinitesimal character.

$\Longleftarrow$. Then
\begin{align*}
	\sum\sum f\left(x^{(2)}y^{(1)}\right)x^{(1)}\otimes y^{(2)}
	&=\sum\sum f\left(x^{(2)}\right)\varepsilon_\Delta\left(y^{(1)}\right)x^{(1)}\otimes y^{(2)}\\
	&+\sum\sum \varepsilon_\Delta\left(x^{(2)}\right)f\left(y^{(1)}\right)x^{(1)}\otimes y^{(2)}\\
	&=\sum f\left(x^{(2)}\right)x^{(1)}\otimes y+\sum f\left(y^{(1)}\right)x\otimes y^{(2)},
\end{align*}
so $\Delta_f$ is infinitesimal of weight $0$.
\end{proof}

\begin{coro}\label{corinfinitesimal}
Let  $(A,m,\Delta)$ be a counitary $\epsilon$-bialgebra of weight $-1$ and $f\in A^*$.
Then $(A,m,\Delta_f)$ is an $\epsilon$-bialgebra of weight $-\lambda$ if, and only if, $f+\lambda \varepsilon_\Delta$ is an infinitesimal character of $A$.
\end{coro}

\begin{proof}
We put $g=f+\lambda \Delta_\varepsilon$. Then
\[\Delta_g=\Delta_f+\lambda \Delta_{\varepsilon_\Delta}=\Delta_f+\lambda \Delta.\]
As $(A,m,\Delta)$ is infinitesimal of weight $-1$, $(A,m,\Delta_g)$ is infinitesimal of weight $0$ if, and only if,
$(A,m,\Delta_f)$ is infinitesimal of weight $-\lambda$. The result then comes directly from Proposition \ref{prop2}.
\end{proof}

\begin{prop}\label{propcouniteconnexe}
Let us assume that $(A,m,\Delta)$ is a connected, unitary and counitary $\epsilon$-bialgebra of weight $-1$, and let $f\in A^*$.
We denote by $*$ the convolution product on $A^*$ induced by $\Delta$.
Then $\Delta_f$ has a counit if, and only if, $f(1_A)$ is an invertible element of $\bfk$. If so, we put $f(1_A)=-\lambda$ and $g=f+\lambda \varepsilon_\Delta$.
The the counit of $\Delta_f$ is
\begin{align}
	\label{eq1}\varepsilon_{\Delta_f}&=-\sum_{k=0}^\infty \frac{1}{\lambda^{k+1}}g^{*k}.
\end{align}
\end{prop}

\begin{proof}
We have
\[\Delta_f(1_A)=-\lambda 1_A\otimes 1_A,\]
so, if $\lambda$ is not invertible in $\bfk$, $\Delta_f$ has no counit. Otherwise,
as $g(1_A)=0$, the formal series $h$ defined by (\ref{eq1}) makes sense when apply to any $x\in A$. Moreover,
\begin{align*}
	f*h&=(-\lambda \varepsilon_\Delta+g)*\left(-\sum_{k=0}^\infty \frac{1}{\lambda^{k+1}}g^{*k}\right)=\sum_{k=0}^\infty \frac{1}{\lambda^k}g^{*k}-\sum_{k=0}^\infty \frac{1}{\lambda^{k+1}}g^{*(k+1)}=\varepsilon_\Delta,
\end{align*}
so $h$ is a left inverse of $f$. It is similarly a right inverse.  By Proposition \ref{propcounite}, it is the counit of $\Delta_f$.
\end{proof}

\begin{exam} Let us consider the tensor algebra $T(V)$, as in Example \ref{exam:bialgebras}-(d).
Let $f\in V^*$. We extend it to $T(V)$ by
\[f(v_1\ldots v_n)=\begin{cases}
	f(v_1)\mbox{ if }n=1,\\
	0\mbox{ otherwise}.
\end{cases}\]	
Then $f$ is an infinitesimal character. The coproduct $\Delta_f$ is given by
\[\Delta_f(v_1\ldots v_n)=\sum_{i=1}^n f(v_i) v_1\ldots v_{i-1}\otimes v_{i+1}\ldots v_n.\]
With this coproduct and the usual concatenation product, $T(V)$ is an $\epsilon$-bialgebra of weight $0$.
\end{exam}

\begin{prop}\label{prop:antipode}
Let $(A,m,\Delta)$ be an infinitesimal Hopf algebra of weight $-1$. Its antipode is denoted by $S_0$.
Let $h$ be an infinitesimal character of $A$, such that $f=\varepsilon_\Delta+h$ is invertible, of inverse denoted by $g$
(which is the counit of $\Delta_f$). Then $(A,m,\Delta_f)$ is an infinitesimal Hopf algebra of weight $-1$,
and its antipode is $g\star \id_A\star g$.
\end{prop}

\begin{proof}
Let us prove that $T=g\star \id_A\star g$ in an antipode for $\Delta_f$. Let $x\in A$.
\begin{align*}
	T\star_f\id(x)&=m\circ (T\otimes \id)\circ \Delta_f(x)\\
	&=\sum T(x^{(1)})f(x^{(2)}) x^{(3)}\\
	&=\sum g(x^{(1)})S_0(x^{(2)}) g(x^{(3)})f(x^{(4)})x^{(5)}\\
	&=\sum g(x^{(1)})S_0(x^{(2)}) x^{(3)}&\mbox{ as $g\star f=\varepsilon_\Delta$}\\
	&=\sum g(x^{(1)})\varepsilon_\Delta(x^{(2)})1_A&\mbox{by definition of $S_0$}\\
	&=g(x)1_A.
\end{align*}
So $T$ is a left antipode for $(A,m,\Delta_f)$. Similarly, it is a right antipode.
\end{proof}

\subsection{Application to planar forests}

Let $F\in \rfs$. We denote by $n_F$ the number of vertices of $F$, and we index them according to the total order $\leq_{h,l}$:
\[v_{n_F}\leq_{h,l}\ldots \leq_{h,l} v_1.\]
For any $I\subseteq \{1,\ldots,n_F\}$ we denote by $F_I$ the element of $\rfs$ whose vertices are the elements of $I$,
with the partial orders $\leq_h$ and $\leq_l$ obtained by restriction to $I$.
Then (\cite{Foi09}) $\hrts$ is a connected, unitary and counitary $\epsilon$-bialgebra of weight $-1$ with the coproduct defined by
\[\Delta(F)=\sum_{k=0}^{n_F} F_{\{1,\ldots,k\}}\otimes F_{\{k+1,\ldots,n_F\}}.\]

We consider the map $f \in \hrts^*$ defined by
\begin{align*}
&\forall F\in \rfs,&f(F)&=\begin{cases}
	1\mbox{ if }|V(F)|=1,\\
	0\mbox{ otherwise}.
\end{cases}
\end{align*}
It is an infinitesimal character of $\hrts$.
By construction of $\Delta$, for any $F\in \rfs$,
\begin{align*}
\Delta_f(F)&=(\id_{\hrts} \otimes f\otimes \id_{\hrts})
\left(\sum_{0\leq k\leq l\leq n_F} F_{\{1,\ldots,k\}}\otimes  F_{\{{k+1},\ldots,l\}}
\otimes  F_{ \{{l+1},\ldots,n\}}\right)\\
&=\sum_{k=0}^{n_F-1} F_{\{1,\ldots,k\}}\otimes  f(F_{\{{k+1}\}})
\otimes  F_{ \{{k+2},\ldots,n_F\}}+0\\
&=\sum_{k=0}^{n_F-1} F_{\{1,\ldots,k\}}\otimes F_{\{k+2,\ldots,n_F\}}.
\end{align*}
This gives, with Proposition \ref{propcounite} and Corollary \ref{corinfinitesimal}, up to an reindexation:

\begin{theorem}
Let $\lambda,\mu\in \bfk$. We put $\Delta_{\lambda,\mu}=\Delta_{-\lambda\varepsilon_\Delta+\mu f}$.
For any $F\in \rfs$,
\[\Delta_{\lambda,\mu}(F)=-\lambda\left(\sum_{k=0}^{n_F} F_{\{1,\ldots,k\}}\otimes F_{\{k+1,\ldots,n_F\}}\right)
+\mu \left(\sum_{k=1}^{n_F} F_{\{1,\ldots,k-1\}}\otimes F_{\{k+1,\ldots,n_F\}}\right).\]
Moreover, $(\hrts,m,\Delta_{\lambda,\mu})$ is an infinitesimal bialgebra of weight $-\lambda$. It has a counit if, and only if, $\lambda$ is invertible in $\bfk$. If, so its counit is given  by
\[\varepsilon_{\lambda,\mu}(F)=-\frac{\mu^{n_F}}{\lambda^{n_F+1}}.\]
\end{theorem}

\begin{proof}
From Proposition \ref{propcouniteconnexe}, $\Delta_{\lambda,\mu}$ is counitary if, and only if, $\lambda$ is invertible in $\bfk$.
If so, still by Proposition \ref{propcouniteconnexe}, if $n\geq 1$,
\begin{align*}
	\varepsilon_{\lambda,\mu}(F)&=\sum_{k=1}^\infty
	\frac{(-1)^k\mu^k}{(-\lambda)^{k+1}}  f^{\otimes k}\circ \Delta^{(k-1)}(F)
	=-\sum_{k=1}^\infty
	\frac{\mu^k}{\lambda^{k+1}}  f^{\otimes k}\circ \Delta^{(k-1)}(F).
\end{align*}
As $f$ vanishes on any forest of degree $\neq 1$, $f^{\otimes k}\circ \Delta^{(k-1)}(F)=0$ if $k\neq n_F$, and
\[f^{\otimes n_F}\circ \Delta^{(n_F-1)}=f\left(F_{ \{1\}}\right)\ldots f\left(F_{ \{n_F\}}\right)+0=1,\]
so
\[\varepsilon_{\lambda,\mu}(F)=-\frac{\mu^n}{\lambda^{n+1}}+0.\qedhere\]
\end{proof}

The infinitesimal bialgebra $(\hrts,m,\Delta_{-1,0})$ is an infinitesimal Hopf algebra, and its antipode can be computed by Takeuchi's formula:
\[S_0(F)=\sum_{k=1}^{n_F}(-1)^k\sum_{\substack{V(F)=I_1\sqcup_{\leq_{h,l}}\ldots \sqcup_{\leq_{h,l}} I_k,\\ I_1,\ldots,I_k\neq \emptyset}} F_{\mid I_1}\ldots F_{\mid I_k}.\]
The theory of infinitesimal bialgebra shows that $S_0(F)$ is a primitive element (for $\Delta_{-1,0}$) if $F\neq 1$, and is zero if $\bre(F)\geq 2$.

\begin{prop}
$(\hrts,m,\Delta_{-1,\mu})$ in an infinitesimal Hopf algebra of weight $-1$. Its antipode is given by
\[S_\mu(F)=(-\mu)^{n_F}(n_F+1)+\sum_{1\leq k\leq l\leq n_F}(-\mu)^{n_F-l+k-1}S_0(F_{\mid \{k,\ldots,l\}}).\]
\end{prop}

\begin{proof}
We apply Proposition \ref{prop:antipode}: $S_\mu=g\star S0 \star g$. This gives
\begin{align*}
	S_\mu(F)&=\sum_{0\leq k\leq l\leq n_F}g(F_{\mid \{1,\ldots,k\}}) S_0(F_{\mid \{k+1,\ldots,l\}})g(F_{\mid \{l+1,\ldots,n_F\}})\\
	&=\sum_{0\leq k\leq l\leq n_F}(-\mu)^{k+n_F-l}S_0(F_{\mid \{k+1,\ldots,l\}})\\
	&=\sum_{0\leq k\leq n_F} (-\mu)^{n_F}+\sum_{0\leq k<l\leq n_F}(-\mu)^{k+n_F-l}S_0(F_{\mid \{k+1,\ldots,l\}})\\
	&=(-\mu)^{n_F}(n_F+1)+\sum_{1\leq k\leq l\leq n_F}(-\mu)^{n_F-l+k-1}S_0(F_{\mid \{k,\ldots,l\}}).
\end{align*}
where we have regrouped the terms where $l=k$,
for which $F_{\mid \{k+1,\ldots,l\}}=1$.
\end{proof}

\begin{exam}
Let $\alpha,\beta,\gamma \in \Omega$.
\begin{align*}
	S_\mu(\tdun{$\alpha$})&=-\tdun{$\alpha$}-2\mu,\\
	S_\mu(\tdun{$\alpha$}\tdun{$\beta$})&=\mu\tdun{$\alpha$}+\mu\tdun{$\beta$}+3\mu^2,\\
	S_\mu(\tddeux{$\alpha$}{$\beta$})&=-\tddeux{$\alpha$}{$\beta$}+\tdun{$\alpha$}\tdun{$\beta$}
	+\mu \tdun{$\alpha$}+\mu \tdun{$\beta$}+3\mu^2,\\
	S_\mu(\tdtroisdeux{$\alpha$}{$\beta$}{$\gamma$})&=-\tdtroisdeux{$\alpha$}{$\beta$}{$\gamma$}
	+\tddeux{$\beta$}{$\gamma$}\tdun{$\alpha$}+\tdun{$\gamma$}\tddeux{$\alpha$}{$\beta$}
	-\tdun{$\gamma$}\tdun{$\beta$}\tdun{$\alpha$}\\
	&-\mu(-\tddeux{$\beta$}{$\gamma$}+\tdun{$\gamma$}\tdun{$\beta$}-\tddeux{$\alpha$}{$\beta$}+\tdun{$\beta$}\tdun{$\alpha$})
	+\mu^2(-\tdun{$\alpha$}-\tdun{$\beta$}-\tdun{$\gamma$})-4\mu^3.
\end{align*}
\end{exam}

\begin{remark}
In particular, if $F=\bullet_{x_1}\ldots \bullet_{x_n}$ and $I\subseteq V(F)$, of cardinality $\geq 2$, then $\bre(F_{\mid I})\geq 2$, so $S_0(F_{\mid I})=0$. This gives
\[S_\mu(\bullet_{x_1}\ldots \bullet_{x_n})
=(n+1)(-\mu)^n+(-\mu)^{n-1}\sum_{i=0}^n S_0(\bullet_{x_i})
=(-1)^n\mu^{n-1}\left((n+1)\mu+\sum_{i=0}^n \bullet_{x_i}\right).\]
\end{remark}

\begin{prop}
For any $x\in \hrts$, for any $\omega \in \Omega$,
\[\Delta_{\lambda,\mu} \circ B^+_\omega(x)=-\lambda B^+_\omega(x)\otimes 1+\mu x\otimes 1+(\id_{\hrts}\otimes B^+_\omega)\circ\Delta_{\lambda,\mu}(x).\]
So $\Delta_{\lambda,\mu}$ is the coproduct $\Delta_\epsilon$ defined earlier.
\end{prop}

\begin{proof}
We know that for any $x\in \hrts$,
\[\Delta \circ B^+_\omega(x)=B^+_\omega(x)\otimes 1+(\id_{\hrts}\otimes B^+_\omega)\circ\Delta(x).\]
Therefore,
\[\Delta^{(2)}\circ B^+_\omega(x)=B^+_\omega(x)\otimes 1\otimes 1+(\id_{\hrts}\otimes B^+_\omega)\circ\Delta(x) \otimes 1
+(\id_{\hrts}\otimes\id_{\hrts}\otimes B^+_\omega)\circ\Delta(x),\]
which gives
\[\Delta_f\circ B^+_\omega(x)=0+(\id_{\hrts}\otimes f\circ B^+_\omega)\circ\Delta(x) \otimes 1+(\id_{\hrts}\otimes B^+_\omega)\circ \Delta_f(x).\]
Moreover, for any $F\in \hrts$,
\begin{align*}
	f\circ B^+_\omega(F)&=\begin{cases}
		1\mbox{ if }|V(B^+_\omega(F))|=1,\\
		0\mbox{ otherwise}
	\end{cases}\\
	&=\begin{cases}
		1\mbox{ if }F=1,\\
		0\mbox{ otherwise}
	\end{cases}\\
	&=\varepsilon_\Delta(F).
\end{align*}
So $f\circ B^+_\omega=\varepsilon_\Delta$, which implies that
\[\Delta_f\circ B^+_\omega(x)=x \otimes 1+(\id_{\hrts}\otimes B^+_\omega)\circ \Delta_f(x).\]
Finally,
\begin{align*}
	\Delta_{\lambda,\mu}\circ B^+_\omega(x)&=
	-\lambda \Delta \circ B^+_\omega(x)+\mu \Delta_f\circ B^+_\omega(x)\\
	&=-\lambda B^+_\omega(x)\otimes 1+\lambda (\id_{\hrts}\otimes B^+_\omega)\circ \Delta(x)+\mu x\otimes 1+\mu(\id_{\hrts}\otimes B^+_\omega)\circ \Delta_f(x)\\
	&=-\lambda B^+_\omega(x)\otimes 1++\mu x\otimes 1+(\id_{\hrts}\otimes B^+_\omega)\circ \Delta_{\lambda,\mu}(x). \qedhere
	\end{align*}\end{proof}

	-------------------------------------------------------